\documentclass[twoside,11pt]{article}

\usepackage[utf8]{inputenc}
\usepackage[T1]{fontenc}
\usepackage{url}
\usepackage{booktabs}
\usepackage{amsfonts}
\usepackage{nicefrac}
\usepackage{microtype}
\usepackage{xcolor}
\usepackage{enumitem}

\usepackage{amsmath}
\usepackage{amsthm} %
\usepackage{amssymb}
\usepackage{amstext}
\usepackage{amsfonts}
\theoremstyle{plain}
\newtheorem{theorem}{Theorem}
\newtheorem{lemma}[theorem]{Lemma}
\newtheorem{proposition}[theorem]{Proposition}
\newtheorem{corollary}[theorem]{Corollary}

\theoremstyle{definition}

\newtheorem{definition}[theorem]{Definition}
\newtheorem{example}[theorem]{Example}

\newtheorem{remark}[theorem]{Remark}

\newtheorem*{definition*}{Definition}
\newcommand{\jmin}{J_{x^\star_{\min}}}
\newcommand{\NN}{\mathbb{N}}

\usepackage{tikz}
\usetikzlibrary{positioning,shadows,fit,decorations.pathmorphing,matrix}
\usepackage{wrapfig}

\newcommand{\tp}{^{\textnormal{\textsf{T}}}}

\newcommand{\mA}{\mathcal{A}}
\newcommand{\mB}{\mathcal{B}}
\newcommand{\mC}{\mathcal{C}}
\newcommand{\mD}{\mathcal{D}}

\newcommand{\mI}{\mathcal{I}}
\newcommand{\mJ}{\mathcal{J}}

\newcommand{\mL}{\mathcal{L}}

\newcommand{\mO}{\mathcal{O}}

\newcommand{\mQ}{\mathcal{Q}}
\newcommand{\mR}{\mathcal{R}}
\newcommand{\mS}{\mathcal{S}}
\newcommand{\mT}{\mathcal{T}}
\newcommand{\mU}{\mathcal{U}}
\newcommand{\mV}{\mathcal{V}}

\newcommand{\mZ}{\mathcal{Z}}

\newcommand{\mbN}{\mathbb{N}}

\newcommand{\mbR}{\mathbb{R}}

\newcommand{\etz}{\varepsilon \rightarrow 0}

\newcommand{\xseps}{\xs_{\varepsilon}}
\newcommand{\xsepst}{\xs_{\varepsilon}(\theta)}

\newcommand{\range}{\mathrm{range}}

\newcommand{\sign}{\mathrm{sign}}

\newcommand{\prox}{\mathrm{prox}}

\newcommand{\dist}{\mathrm{dist}}
\newcommand{\setvalued}{\rightrightarrows}
\DeclareMathOperator*{\argmin}{argmin}

\newcommand{\xs}{{x}^\star}
\newcommand{\minsol}{x^\star_{\mathrm{min}}}
\newcommand{\diag}{\mathrm{diag}}
\newcommand{\conv}[1]{\mbox{conv}{#1}}
\newcommand{\eps}{\varepsilon}

\def\eqref#1{{\rm (\ref{#1})}}
\def\lemref#1{{\rm Lemma~\ref{#1}}}
\def\propref#1{{\rm Proposition~\ref{#1}}}
\def\thmref#1{{\rm Theorem~\ref{#1}}}
\def\defref#1{{\rm Definition~\ref{#1}}}
\def\secref#1{{\rm Section~\ref{#1}}}

\def\remref#1{{\rm Remark~\ref{#1}}}
\def\corref#1{{\rm Corollary~\ref{#1}}}

\newcommand{\Heps}{H_{\varepsilon}}
\newcommand{\Meps}{M_{\varepsilon}}
\newcommand{\Qeps}{Q_{\varepsilon}}

\newcommand{\zext}{z_{\varepsilon}(x, \theta)}

\usepackage[preprint]{jmlr2e}
\renewenvironment{proof}[1][]{%
    \par\noindent{\bf Proof}%
    \if\relax\detokenize{#1}\relax\else\ (#1)\fi\space\ignorespaces
}{\hfill\BlackBox\\[2mm]}

\usepackage{lastpage}
\jmlrheading{}{}{}{7/26; Revised /}{/}{21-0000}{Baptiste Plaquevent-Jourdain, Jalal Fadili and Antonio Silveti-Falls}

\ShortHeadings{Differentiating Minimal-Norm Solutions}{Baptiste Plaquevent-Jourdain, Jalal Fadili and Antonio Silveti-Falls}
\firstpageno{1}

\begin{document}

\title{Differentiating Minimal-Norm Solutions to Parametric Optimization Problems}

\author{\name Baptiste Plaquevent-Jourdain \email baptiste.plaquevent-jourdain@USherbrooke.ca \\
       \addr CVN, OPIS\\
       CentraleSupélec\\
       Gif-sur-Yvette, 91190, France\\
       Postdoctoral researcher, corresponding author
       \AND
       \name Jalal Fadili \email jalal.fadili@ensicaen.fr \\
       \addr GREYC CNRS\\
       ENSICAEN\\
       Caen, 14000, France
       \AND
       \name Antonio Silveti-Falls \email tonys.falls@gmail.com \\
       \addr CVN, OPIS\\
       CentraleSupélec\\
       Gif-sur-Yvette, 91190, France}

\editor{My editor}

\maketitle

\begin{abstract}%
Differentiating through parametric optimization problems is central to bilevel programming and meta-learning, often accomplished using approximate implicit differentiation. The implicit function theorem requires inverting a partial Jacobian of the optimality condition, which fails when there are many solutions.
Nonetheless, in such cases it is possible to relax invertibility to a strictly weaker uniform range condition, under which it is shown that the minimal-norm solution mapping admits generalized derivatives by using a limiting Tikhonov regularization argument and conservative set-valued field theory. With additional control on the eigenvalues of the generalized Hessians, a pseudoinverse formula is justified. This is established for a class of smooth convex objectives and extended to nonsmooth composite problems. These assumptions are verified for Least-Squares, Huber regression and LASSO. The resulting extension of nonsmooth implicit differentiation to ill-posed settings is examined experimentally on data poisoning and data hypercleaning problems.
\end{abstract}


\section{Introduction}\label{s_intro}
Parametric optimization is pervasive in modern machine learning, including bilevel programming problems. Examples include hyperparameter tuning \citep{pedregosa_2016,lorraine_vicol_duvenaud_2019,bertrand_klopfenstein_blondel_vaiter_gramfort_salmon_2020}, dictionary learning \citep{PeyreFadili11, mairal_bach_ponce_2011}, meta-learning \citep{franceschi_frasconi_salzo_grazzi_pontil_2018,engstrom2025optimizing}, data transformations (cleaning, poisoning, etc) \citep{wang_zhu_torralba_efros_2018,blondel_berthet_cuturi_frostig_hoyer_llinares-lopez_pedregosa_vert_2021}, model selection in statistics \citep{vaiter2015model,vaiter2017degrees}, and deep learning \citep{winston_kolter_2020}. All of these can be cast into the following template bilevel optimization problem, which we consider in this paper:
\begin{gather}\tag{Q}\label{Q}
        \min_{\theta\in\Theta}\mathcal{L}(x^\star(\theta),\theta) \\
        \tag{P}\label{P}
        \mathrm{s.t.}\quad x^\star(\theta) \in \mS(\theta):=\argmin_{x\in\mathbb R^n} f(x,\theta).
\end{gather}
Here $\theta\in\Theta\subset\mathbb R^m$ is the \textit{upper-level} variable and $x$ is the \textit{lower-level} variable. We work under the standing assumptions:
\begin{enumerate}[label={\bf(Reg)}]
\item \label{assum:f_reg} {[Regularity]} For every $\theta \in \Theta$, $f(\cdot, \theta)$ is $\mC^{1,1}(\mathbb{R}^n)$, convex, and $\mS(\theta) \neq \emptyset$. Furthermore, $F := \nabla_x f$ is locally Lipschitz continuous jointly in $(x, \theta)$. 
\end{enumerate}
\begin{enumerate}[label={\bf(D)}]
\item \label{assum:f_def} {[Definability]} The function $f$ and the set $\Theta$ are \textit{definable}. By definability of $f$, $F := \nabla_x f$ is also definable. Since $F$ is assumed locally Lipschitz, it is path-differentiable. 
\end{enumerate}
Throughout, definability is understood relative to some fixed o-minimal geometry \citep{vandendries_miller_1996,coste1999introduction} and will be used in all of our arguments. We also give a short overview to this theory in Appendix~\ref{s_definability}.

First-order methods for \eqref{Q}-\eqref{P} often require differentiating a selection of the lower-level solution map $\mS$. Under sufficient regularity assumptions and strong convexity of the lower-level objective $f$, the implicit function theorem (IFT) applied to the first-order optimality condition $F(x,\theta) = 0$ gives a Jacobian of the solution map $\theta \mapsto x^\star(\theta)$ in closed-form as $-H^{-1}M$, where $(H\ M) \in J_F(x^\star(\theta), \theta)$ is an element of the generalized Jacobian of $F$ \citep{bolte_le_pauwels_silveti-falls_2021,bolte_pauwels_silveti-falls_2024}.\footnote{We pick the letters $H$ and $M$ to represent the Hessian and Mixed derivative, since $F(x,\theta)=\nabla_xf(x,\theta)$, in parallel with the smooth case.} The nonsmooth IFT given in \citet{bolte_le_pauwels_silveti-falls_2021} justifies this formula with conservative Jacobians \citep{bolte_pauwels_2019} but it still requires invertibility of the associated $H$ matrices.

When the lower-level problem is ill-posed and admits multiple solutions, e.g., when $\mS(\theta)$ is not a singleton, $H$ is no longer guaranteed to be invertible and the formula $-H^{-1}M$ can fail. This is a common occurrence in practice, especially for overparameterized problems \citep{vicol_lorraine_pedregosa_duvenaud_grosse_2022}. Even a simple Least-Squares problem with $f(x,\theta) = \tfrac{1}{2}\|Ax - \theta\|_2^2$ for $A\in\mathbb{R}^{m\times n}$ with $m<n$ has a nontrivial affine subspace of solutions; the same degeneracy arises in Huber regression with a fat design matrix and in the LASSO \citep{tibshirani_1996, tibshirani_2011} with redundant features.

In practice, invertibility is treated as an afterthought despite some evidence showing that its absence could lead to serious training instabilities \cite[Section 5]{bolte_le_pauwels_silveti-falls_2021}. A solution seen in practice is to apply the pseudoinverse $H^\dag$ if $H$ is singular \citep{vicol_lorraine_pedregosa_duvenaud_grosse_2022}. Our goal is to rigorously study aspects of these practical approaches by connecting them to the \textit{minimal-norm solution}
\begin{equation*}
\minsol(\theta) := \argmin\{\|x\|_2:x\in \mS(\theta)\},
\end{equation*}
which is selected by \textit{Tikhonov regularization} \citep{tikhonov_1963}. Namely, for $\varepsilon>0$, define the $\eps$-regularized problem
\begin{equation}
\tag{$\mbox{P}_\eps$}\label{Pvarepsilon}
x^\star_\varepsilon(\theta):=\argmin_{x\in\mathbb R^n}\left\{f(x,\theta)+\tfrac{\varepsilon}{2}\|x\|_2^2\right\}.
\end{equation}

Classical results \citep{browder_1966,browder_1967} guarantee that $x_\varepsilon^\star(\theta)\to \minsol(\theta)$ as $\eps\to0^+$. Since $f(\cdot,\theta)$ is convex for any $\theta$, then for each fixed $\varepsilon>0$, the regularized objective in \eqref{Pvarepsilon} is strongly convex with (generalized) gradient $F_{\varepsilon}$. Hence nonsmooth implicit differentiation \citep{bolte_pauwels_silveti-falls_2024} gives a conservative Jacobian of $x^\star_\varepsilon$ of the form
\begin{equation*}
J_{x_\varepsilon^\star}(\theta) = \left\{-(H_{\varepsilon}+\varepsilon I)^{-1}M_{\varepsilon}\colon(H_{\varepsilon}\ M_{\varepsilon})\in J_{F}(x^\star_\varepsilon(\theta),\theta)\right\}.
\end{equation*}
The central question is therefore
\begin{center}
{\textit{Do these regularized conservative Jacobians have a meaningful limit as $\varepsilon\to0^+$, and how does this limit relate to $\minsol$ (cf. Figure~\ref{fig:ad}) ?
}}
\end{center}

\begin{wrapfigure}{r}{0.46\textwidth}
\centering
\vspace{-0.5cm}
\begin{tikzpicture}
    \draw (0,0) node(f) {$x^\star_{\eps}(\theta)$};
    \draw (3.1,0) node(df) {$x^{\star}_{\min}(\theta)$};
    \draw (0,-2.25) node(fun) {$J_{x^\star_\eps}(\theta)$};
    \draw (3.1,-2.25) node(dfun) {$\jmin(\theta)$};

    \draw[->] (f) -- (df)
        node[above,midway]{\footnotesize \citet{browder_1966}}
        node[below,midway]{\footnotesize $\eps\to0^+$};
    \draw[->] (f) -- (fun)
        node[rotate=90,above,midway]{\footnotesize implicit}
        node[rotate=90,below,midway]{\footnotesize diff.};
    \draw[red,thick,densely dashed,->] (fun) -- (dfun)
        node[below,midway]{\footnotesize\color{red} limit?};
    \draw[red,thick,densely dashed,->] (df) -- (dfun)
        node[rotate=90,below,midway]
        {\footnotesize\color{red} derivative?};
\end{tikzpicture}
\caption{Solid arrows are known; dashed arrows are studied in this
work. Standard IFT applies to the Tikhonov regularized solution
$x^\star_\eps$ but fails in general if $H$ is singular. Under suitable
assumptions, we show that the cluster points of the conservative
Jacobians along the bottom path are conservative for the minimal-norm
solution, justifying in addition the formula
$-H^\dagger M\in\jmin(\theta)$ for $\minsol$ along the right path when
the blocks $\Heps$ have controlled eigenvalues.}\vspace{-0.75cm}
\label{fig:ad}
\end{wrapfigure}
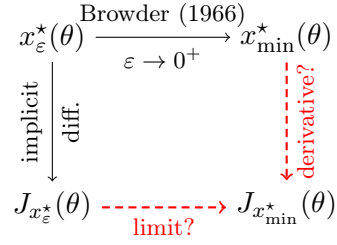
We answer this question affirmatively using recent results in the theory of path-differentiable functions \citep{schechtman2026gradient}, which give a precise characterization of limits of gradients and conservative mappings for definable mappings. Our results show that invertibility can be relaxed to a uniform \textit{range condition}, $\Meps\in\range(\Heps)$ around the solution, together with a uniform condition on the eigenvalues of the $\Heps$. Under these assumptions, the Tikhonov implicit derivatives converge to $-H_0^{\dagger} M_0$ as $\eps\to0^+$ for some pairs $(H_0, M_0) \in J_F(\minsol(\theta), \theta)$, and this limit defines an element of a conservative mapping of the minimal-norm solution $\minsol(\theta)$. In turn, this element can be used in first-order methods for the upper level. This extends implicit differentiation using conservative calculus from well-posed lower-level problems to ill-posed convex problems with many solutions.

The range condition is a strict relaxation of the usual invertibility condition required since, if the matrices $H_0$ are invertible, so are the matrices $\Heps$ for $\varepsilon$ small enough thus $M_{\varepsilon}\in\range(H_{\varepsilon})$ (by continuity of the determinant). The eigenvalues condition is not necessary for path-differentiability but only for the pseudoinverse formula, which must be corrected without this condition. In general, the matrices $\Heps$ shall be symmetric, which simplifies the statement of the eigenvalue condition: for functions that are twice differentiable (even without continuity of their Hessians), the Hessian is symmetric~\citep{dieudonne1969foundations}; for nonsmooth functions, every matrix in the Clarke Jacobian of a proximal operator is also symmetric and positive semidefinite, see \citep[\thmref{thm_CF_prox}]{patrinos_stella_bemporad_2014}.
We exhibit a problem (optimization over the $\ell_\infty$-ball, see Figure~\ref{fig:linf}) where the minimal-norm solution is discontinuous, so no conservative mapping can exist, illustrating the limits of the approach and consistent with known hardness results for bilevel optimization \citep{bolte2026geometric}.
\subsection{Main Contributions}
\begin{itemize}%
    \item We show path-differentiability of the minimal-norm solution to a parametric convex optimization problem: under definability and local boundedness of the regularized implicit Jacobians, the limits of sequences of regularized conservative Jacobians are conservative for the minimal-norm solution.
    \item Under a uniform range condition on the second-order derivatives, we justify local boun\-ded\-ness and thus path-differentiability of the minimal-norm solution. This condition is verified on Least-Squares, Huber regression, and LASSO.
    \item If in addition an eigenvalue condition holds, we justify an explicit formula for the implicit Jacobians in terms of the pseudoinverse of the Hessian. In general, a residual term appears resulting from vanishing curvature alongside the trajectory.
    \item We experimentally verify our work with problems involving ill-posed lower-level objectives. The benefits of differentiating the minimal-norm solution are showcased on data poisoning experiments with Least-Squares and Huber loss functions, which display the relevance of the minimal-norm solution compared to guaranteed nonminimal solutions. We also present a data hypercleaning experiment with LASSO lower-level that exposes a tradeoff between regularization and precision.
\end{itemize}
\subsection{Related Work}
Differentiating solutions is intrinsically related to sensitivity analysis \citep{bonnans_shapiro_2000}. Nonsmooth versions of the IFT first appear in \citet{clarke_1990, robinson_1991, dontchev_rockafellar_2009}. In \citet{adly_rockafellar_2021}, one-dimensional parametric variational inclusions are con\-si\-dered via graphical convergence, semi- and proto-differentiability. The role of (Tikhonov) regularization is also discussed in \cite{attouch_1996}. In \citet{arbel_mairal_2022} the authors introduce a parametric Morse-Bott property \citep[see also][]{bolte_le_pauwels_vaiter_2026} which implies that $f$ behaves quadratically around critical points. This is a more demanding property as it implies the range condition \citep[Proposition~6]{arbel_mairal_2022} though it captures the more general case where the lower-level problem is nonconvex, in which case Tikhonov regularization and Browder convergence cannot be applied. Recently, in \citet{masiha_shen_kiyavash_he_2026}, the authors consider a similar setting under a Polyak-\L ojasiewicz assumption (uniform in $\theta$) alongside third-order smoothness and nondegeneracy. This allows them to diffe\-rentiate an optimistic selection mapping and to obtain formulas also employing a kind of pseudoinverse. Another approach to ensure finding the minimal-norm solution is to use vanishing Tikhonov regularization $\varepsilon_k$ with a first-order solver. This technique was studied for unrolled differentiation in \citet{mehmood2024automatic}. Closest to our work is \citet{vicol_lorraine_pedregosa_duvenaud_grosse_2022}, which ensures convergence in a setting with quadratic lower and upper objectives. Despite this limitation, they validate the behavior of several algorithms on multiple problems. A related contribution is \citet{pauwels2023derivatives}, which showed the convergence of derivatives of the Sinkhorn algorithm in optimal transport without contractivity by using the spectral inverse (a refined inverse identical to pseudoinverse for symmetric matrices). In \citet{bambade_schramm_taylor_carpentier_2024}, the authors differentiate the solutions of a quadratic problem with parametric cost and constraints using Lagrangian-based techniques and minimal-norm solutions. We also mention that there are many works that use value functions or adjoint methods to transform the bilevel problem into a single-loop problem, for example \cite{suonpera2026single}.%
\begin{wrapfigure}{r}{0.46\textwidth}
\centering
\vspace{-.5cm}
\includegraphics[width=0.92\linewidth]{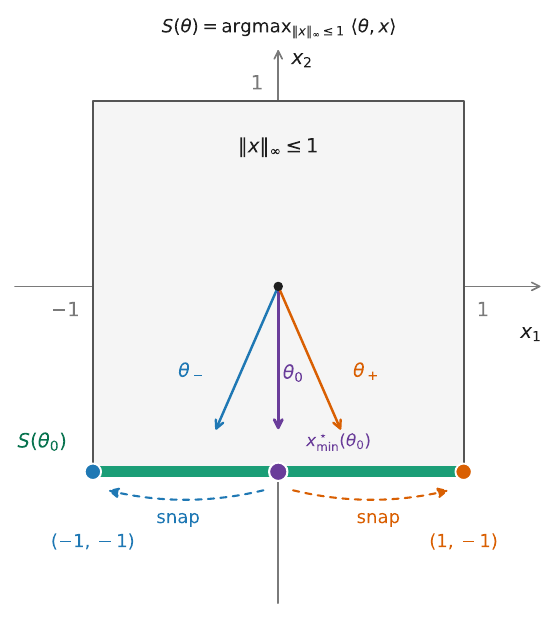}
\caption{A simple composite problem,
$\max_{\|x\|_{\infty}\leq1}\langle\theta,x\rangle$. The minimal-norm
solution is $\minsol(\theta)=\sign(\theta)$, with $\sign(0)=0$, which
is not continuous at the represented $\theta_0=(0,-1)$ and hence is
not path-differentiable; the range conditions must therefore fail
(see \remref{rem_failure_Rcomp_Linf}). Here
$\mS(\theta_0)$, shown in green, is an edge of the ball and
$\minsol(\theta_0)=(0,-1)$. Perturbing $\theta_0$ toward
$\theta_-$ or $\theta_+$ can collapse the solution set to a singleton
and make the minimal-norm solution jump to a corner.}\vspace{-.75cm}
\label{fig:linf}
\end{wrapfigure}
The paper is organized as follows. In \secref{s_nonsmooth_illposed}, we present the main theoretical tools we employ, conservative Jacobians and nonsmooth implicit function theorem, and how to apply them to ill-posed settings via Tikhonov regularization, see also \secref{s_definability}. Then, \secref{s_C11} presents our framework on smooth convex objectives, some of the proofs being in \secref{s_C11_proofs}. Under stronger assumptions, we present extensions for nonsmooth composite problems in \secref{s_involved_pbs}, completed in \secref{s_composite}. \secref{ss_convergence_rates} briefly discusses how regularization and accuracy counterbalance each other (see also \secref{s_detailed_rates}). Finally, \secref{s_experiments} presents experimental results, with additional details for reproducibility given in \secref{s_additional_experiments}.

\paragraph{Main Notation}
For a matrix $A \in \mbR^{m \times n}$, the transpose of $A$ is denoted by $A\tp \in \mbR^{n \times m}$, its null space by $\ker(A)$, its range space by $\range(A)$ and its pseudoinverse by $A^{\dagger} \in \mbR^{n \times m}$ (see \defref{def_pseudoinverse}). $A \in \mbR^{n \times n}$ is unitary if $A\tp A = I_n = AA\tp$ with $I_n$ (or $I$) the identity matrix of size $n \times n$. A matrix $A$ that is positive semidefinite (PSD) will be denoted by $A\succeq 0$. The spectrum of a square matrix is denoted $\mathrm{spec}$. For $q \in \mbR^n$, $\diag(q) \in \mbR^{n \times n}$ denotes the diagonal matrix whose entries are equal to $q_i$. For $A$ and $B$ with suitable dimensions, $(A\ B)$ is the horizontal concatenation of $A$ and $B$,$(A;B)$ the vertical one. Set-valued mappings are denoted by $\setvalued$. A function $F : \mbR^n \rightarrow \mbR^m$ is \textit{locally Lipschitz} if, for each $x \in \mbR^n$, there exists a neighborhood $\mU$ of $x$ such that $F$ is Lipschitz on $\mU$. The set of points at which $F$ is differentiable is denoted by $\mathrm{Diff}(F)$; at these points, $\nabla F(x) \in \mbR^{n \times m}$ is the gradient of $F$ and $\mathrm{Jac}F(x) = \nabla F(x)\tp$ its Jacobian. The class of extended real-valued convex, lower semicontinuous, and proper functions on $\mathbb{R}^n$ is denoted $\Gamma_0(\mbR^n)$. For $f \in \Gamma_0(\mbR^n)$, $\prox_f(x) := \argmin_y \{f(y) + \|y-x\|^2_2/2\}$ is the proximal operator of $f$ \citep{moreau_1965, bauschke_combettes_2017}. We refer to \secref{s_classic_results} for additional information.

\section{Implicit Differentiation and Conservative Mappings}\label{s_nonsmooth_illposed}
This section details the main analytical tools used in the paper, which are the conservative implicit function theorem \citep{bolte_le_pauwels_silveti-falls_2021} and a recent result from \citet{schechtman2026gradient} about limits of path-differentiable functions. We combine these in \thmref{thm_general} to give a sufficient condition that ensures the path-differentiability of the minimal-norm solution.
\subsection{Background}\label{ss_PathDiff_ConsJacs}
The notions of \textit{conservative Jacobian} (used interchangeably with \textit{conservative mapping}) and \textit{path-differentiable function} were introduced in the seminal paper of \citet{bolte_pauwels_2019} as a way to faithfully model automatic differentiation.
\begin{definition}
    [Conservative mappings {\citep[section 2]{bolte_pauwels_vaiter_2022}}] Let $F : \mbR^n \rightarrow \mbR^m$ be locally Lipschitz, then $J_F : \mbR^n \setvalued \mbR^{m \times n}$ is a \textit{conservative Jacobian} for the \textit{path-differentiable} function $F$ if $J_F$ is never empty, is locally bounded, has a closed graph, and satisfies, for any absolutely continuous curve $\mu : [0, 1] \rightarrow \mbR^n$ and almost all $t\in[0,1]$,
    \[\frac{\mathrm{d}}{\mathrm{d}t} F(\mu(t)) = V \dot{\mu}(t)\qquad \forall\ V \in J_F(\mu(t)).\]
\end{definition}
A conservative mapping appears as a generalized derivative of a locally Lipschitz function along absolutely continuous curves and generalizes the Clarke Jacobian \citep{clarke_1990}
\[
\partial F(x) = \conv{\left\{ \lim_{y \rightarrow x}\mathrm{Jac} F(y) : y \in \mathrm{Diff}(F) \right\}},
\]
where $\mathrm{Diff}$ is the differentiable domain of $F$, the set of points at which it is differentiable. 

Contrary to the Clarke Jacobian \citep[see][Example~1]{bolte_le_pauwels_silveti-falls_2021}, conservative mappings enjoy exact calculus rules such as sum and composition, which notably justifies automatic differentiation: its outputs are elements of a certain conservative Jacobian of the solution mapping \citep[see][section 5]{bolte_pauwels_2019}. For convex conservative mappings, the Clarke Jacobian is a minimal conservative mapping \citep[Corollary~1]{bolte_pauwels_2019}.

Almost everywhere, $J_F(x) = \{\mathrm{Jac} F(x)\} = \partial F(x)$ by Rademacher's theorem since $F$ is locally Lipschitz \citep{rademacher_1919}. If $J_F$ takes convex values, $\partial F(x) \subset J_F(x)$; conservative mappings are preserved by pointwise convexification but need not be convex in general.

Functions that have a conservative Jacobian are said to be \textit{path-differentiable}. Important examples include locally Lipschitz convex or concave functions and, crucially for this paper, locally Lipschitz definable mappings on open definable domains \citep{davis_drusvyatskiy_kakade_lee_2018, bolte_pauwels_2019}. Here definability is relative to some o-minimal structure fixed throughout the paper; see \secref{s_definability}. Semialgebraic problems are therefore covered as a special case.
The main property of conservative mappings employed in this paper is a nonsmooth IFT, which requires the existence of a locally Lipschitz implicit function. Another version \citep[Corollary~1]{bolte_le_pauwels_silveti-falls_2021} proves this existence but requires the conservative mapping to be convex.
\begin{theorem}
    [Path-differentiable IFT \citep{bolte_le_pauwels_silveti-falls_2021}]\label{thm_IFT_other} Let $F : \mbR^n \times \mbR^m \rightarrow \mbR^n$ be path-differentiable on an open set $\mU \times \mV \subset \mbR^{n} \times \mbR^m$ and $G : \mV \rightarrow \mU$ a locally Lipschitz function such that, $\forall\ y \in \mV$,
    \[F(G(y), y) = 0.\]
    Furthermore, assume that $\forall\ y \in \mV$, $\forall\ (H\ M) \in J_F(G(y), y)$ the matrix $H$ is invertible where $J_F$ is a conservative Jacobian for $F$. Then, $G$ is path-differentiable with conservative Jacobian given, $\forall\ y \in \mV$, by
    \[J_G : y \rightrightarrows \{-H^{-1}M : (H\ M) \in J_F(G(y), y)\}.\]
\end{theorem}
For strongly convex problems, it is possible to proceed as in \citet{bolte_pauwels_silveti-falls_2024}. Let $f : \mbR^n \times \mbR^m \rightarrow \mbR$ such that, for each $\theta \in \Theta \subset \mbR^m$, $f(\cdot, \theta)$ is definable, $\mC^{1,1}$ and strongly convex. Consider
\begin{equation}\label{pb_parameterized_simple}
    \theta \mapsto \xs(\theta) := \underset{x \in \mbR^n}{\mathrm{argmin}}\ f(x, \theta).
\end{equation}
For fixed $\theta$, this strongly convex problem has a unique solution $\xs(\theta)$ verifying the optimality condition $F(\xs(\theta), \theta) = 0$ where $F = \nabla_x f : \mbR^n \times \mbR^m \rightarrow \mbR^n$. Then, the $H$-blocks of elements $(H\ M) \in J_F(\xs(\theta), \theta)$ are invertible because they are generalized Hessians of a strongly convex $\mC^{1,1}$ function. A way to proceed \citep[the strategy of][]{bolte_pauwels_silveti-falls_2024} is to convexify the conservative Jacobian $J_F$, apply Corollary~1 of \citet{bolte_le_pauwels_silveti-falls_2021}, which gives existence of the implicit mapping, and then to apply \thmref{thm_IFT_other} using the given implicit mapping with the nonconvex conservative mapping.

In this paper, we concern ourselves with ill-posed problems that are not strongly convex (the matrices $H$ can be singular) such as Least-Squares \citep{fahrmeir_kneib_lang_marx_2013} or LASSO \citep{tibshirani_1996, osborne_presnell_turlach_2000, tibshirani_2013}. The approach we follow is to use Tikhonov regularization as in \eqref{Pvarepsilon} to circumvent ill-posedness. While the convergence of the solutions $\xsepst \rightarrow \minsol(\theta)$ is well-known \citep{tikhonov_1963,browder_1966, browder_1967}, the convergence of generalized derivatives studied in this paper is far more intricate and requires much more care (see Figure~\ref{fig:ad}).

\subsection{Limits of Conservative Jacobians}\label{ss_Schechtman_limits}
To pass from the regularized solutions to $\minsol$, we use a limiting result for definable conservative Jacobians from \citet{schechtman2026gradient}. We define $\xs$ on $\Theta\times \mathbb{R}_+$ by
\begin{equation*}
    \xs\colon(\theta,\varepsilon)\mapsto \begin{cases}
        \xseps(\theta) & \text{if }\eps>0,\\
        \minsol(\theta) & \text{if }\eps=0,
    \end{cases}
\end{equation*}
with $\minsol$ the minimal-norm solution mapping of \eqref{pb_parameterized_simple}. Given an open set $O\subset\Theta$, for every $\varepsilon>0$ let $J_{\xseps}:O\setvalued\mbR^{n\times m}$ be a chosen conservative Jacobian of $\theta\mapsto\xseps(\theta)$ on $O$. We collect these mappings in the family
\[
    J_{\xs}(\theta,\varepsilon):=J_{\xseps}(\theta),
    \qquad (\theta,\varepsilon)\in O\times\mbR_+^*.
\]
We emphasize that $\varepsilon$ is only an index and $J_{\xs}(\theta,\varepsilon)$ contains derivatives with respect to $\theta$ only, not with respect to $\varepsilon$ (which is how they are used in \citet{schechtman2026gradient}). For $\theta\in O$, define
\begin{equation}\label{eq:J_min}
J_{\minsol}(\theta)
:=
\left\{
V\in\mbR^{n\times m}:
\begin{array}{l}
\exists\ \theta_k\in O,\ \theta_k\to\theta,\ \varepsilon_k\downarrow0,\ V_k\to V
\text{ with }\\
V_k\in J_{\xs}(\theta_k,\varepsilon_k),\text{and }\xs(\theta_k,\varepsilon_k)\to\minsol(\theta)
\end{array}
\right\}.
\end{equation}
\begin{theorem}
    [Limit properties {\cite[Lemma~4.2, Theorem~4.3 and Corollary~4.5]{schechtman2026gradient}}]\label{thm_Schechtman} Fix $\bar\eps>0$. Suppose that (i) the restrictions of $\xs$ and $J_{\xs}$ to $O\times[0,\bar\eps]$ and $O\times]0,\bar\eps]$, respectively, are definable, and (ii) for every $0<\varepsilon\leq\bar\eps$, $J_{\xseps}$ is a conservative mapping of $\xseps$ on $O$. Suppose in addition that, for every compact set $K\subset O$,
    \begin{align*}
        (iii)\quad &\sup_{\theta\in K}
        \|\xs(\theta,\varepsilon)-\minsol(\theta)\|
        \underset{\etz}{\longrightarrow}0,\\
        (iv)\quad &\limsup_{\varepsilon_0\downarrow0}
        \sup\{\|V\|:V\in J_{\xseps}(\theta),\ 0<\varepsilon\leq\varepsilon_0,\ \theta\in K\}
        <\infty.
    \end{align*}
    Then $\minsol$ and $J_{\minsol}$ are definable on $O$, and $J_{\minsol}$ is a conservative mapping of $\minsol$ on $O$.
\end{theorem}

In \thmref{thm_Schechtman}, $(iii)$ is a form of continuity, while $(iv)$ relates to local boundedness of the derivatives which must hold for conservative mappings. As we will show later in \thmref{thm_loc_bound_deriv_imply_cont}, the setting we consider guarantees that $(iv) \implies (iii)$, so that we may focus solely on the boundedness of the sequence of elements coming from the conservative mappings.

\section{Problems with \texorpdfstring{$\mC^{1,1}(\mbR^n)$}{Lipschitz-smooth} Objectives} \label{s_C11}
We study in this section the conservative Jacobians of minimal-norm solutions to problems with $\mC^{1,1}$ lower-level objectives. We first introduce the general setting and illustrate it concretely on Least-Squares; this example motivates the uniform range condition, after which we study the pseudoinverse formula and apply the results to Huber regression.

We consider problem \eqref{P} under Assumptions \ref{assum:f_reg} and \ref{assum:f_def} on an open definable set $O\subset\Theta$. Let $F:=\nabla_x f$, since $f(\cdot,\theta)$ is convex, its solutions are characterized by $F(x,\theta)=0$. For $\varepsilon>0$, define
\begin{equation*}
    F_{\varepsilon}(x,\theta)
    :=F(x,\theta)+\varepsilon x
    =\nabla_x f(x,\theta)+\varepsilon x.
\end{equation*}
The regularized objective in \eqref{Pvarepsilon} is strongly convex and therefore has a unique solution $\xsepst$, characterized by
\begin{equation*}
    F_{\varepsilon}(\xsepst,\theta)=0.
\end{equation*}
Let $J_F$ be a definable conservative Jacobian of $F$. At the regularized solution $\xseps(\theta)$, write
\begin{equation*}
    (H_{\varepsilon}\ M_{\varepsilon})
    \in J_F(\xsepst,\theta).
\end{equation*}
Although $F$ has no explicit dependence on $\varepsilon$, the blocks are evaluated at $\xsepst$, and both may therefore vary along the regularized trajectory. We thus take the following conservative Jacobian for $F_{\varepsilon}$, evaluated along the regularized solution:
\begin{equation*}
    J_{F_{\varepsilon}}(\xsepst,\theta)
    :=\left\{\left(H_{\varepsilon}+\varepsilon I_n\ M_{\varepsilon}\right)
    \colon (H_{\varepsilon}\ M_{\varepsilon})
    \in J_F(\xsepst,\theta)\right\}.
\end{equation*}
When $J_F=\partial F$, the convexity of $f(\cdot,\theta)$ guarantees that every $H_{\varepsilon}$ is symmetric PSD, since these $x$-blocks are generalized Hessians of $f$. If we use a more general conservative Jacobian, we will require its $x$-blocks to be symmetric PSD. In either case, $H_{\varepsilon}+\varepsilon I_n$ is symmetric positive definite. \thmref{thm_IFT_other} then gives the conservative Jacobian
\begin{equation}\label{eq:regularized_implicit_jacobian}
    J_{\xseps}(\theta)
    =\left\{-\left(H_{\varepsilon}+\varepsilon I_n\right)^{-1}M_{\varepsilon}
    \colon (H_{\varepsilon}\ M_{\varepsilon})
    \in J_F(\xsepst,\theta)\right\}.
\end{equation}
For the remainder of this section, $J_{\xseps}$ denotes this implicit conservative Jacobian and $J_{\xs}(\theta,\varepsilon):=J_{\xseps}(\theta)$ denotes the associated family. To apply \thmref{thm_Schechtman} and conclude about $\minsol$, we do not need every selection in \eqref{eq:regularized_implicit_jacobian} to converge. What we need is a local uniform bound: for every compact set $K\subset O$, there must exist $C_K>0$ and $\varepsilon_K>0$ such that
\begin{equation}\label{eq:regularized_jacobian_bound_target}
    \sup\left\{
    \left\|\left(H_{\varepsilon}+\varepsilon I_n\right)^{-1}M_{\varepsilon}\right\|
    \colon
    \begin{array}{l}
        \theta\in K,\quad 0<\varepsilon\leq\varepsilon_K,\\
        (H_{\varepsilon}\ M_{\varepsilon})\in J_F(\xsepst,\theta)
    \end{array}
    \right\}
    \leq C_K.
\end{equation}
\subsection{A Motivating Example}
We start with an example that shows how a bound like \eqref{eq:regularized_jacobian_bound_target} can hold even when the limits are singular.
Here $\Theta=O=\mbR^m$. We will obtain a bound for every $\varepsilon>0$, so both $\varepsilon_K$ in \eqref{eq:regularized_jacobian_bound_target} and the global upper bound $\bar\varepsilon$ used below may be chosen arbitrarily and independently of $K$.
Let $A\in\mbR^{m\times n}$ with $m<n$ and consider the Least-Squares problem
\[\label{eq:least-squares}\tag{LS}
    \min\limits_{x\in\mbR^n}\frac{1}{2}\|Ax-\theta\|_2^2.
\]
For this problem, $F(x,\theta)=A\tp(Ax-\theta)$, and we take $J_F=\mathrm{Jac} F$. Along the regularized trajectory, the blocks introduced above are
\begin{equation*}
    H_{\varepsilon}\equiv A\tp A,
    \qquad
    M_{\varepsilon}\equiv-A\tp.
\end{equation*}
The set of solutions to \eqref{eq:least-squares} is $A^\dagger\theta+\ker(A)$, so its minimal-norm solution is $\minsol(\theta)=A^\dagger\theta$.
The Jacobian of $\minsol$ is $A^\dagger$. We now recover the same formula from the Tikhonov-regularized problem. Its solution and Jacobian are
\begin{equation*}
    \xsepst
    =\left(A\tp A+\varepsilon I_n\right)^{-1}A\tp\theta,
    \qquad
    \mathrm{Jac}\,\xseps(\theta)
    =\left(A\tp A+\varepsilon I_n\right)^{-1}A\tp.
\end{equation*}
Set $W_{\varepsilon}\equiv-A^\dagger$. The identity $A\tp A A^\dagger=A\tp$ (see \lemref{lem_pseudoinverse_properties}) gives $M_{\varepsilon}=H_{\varepsilon}W_{\varepsilon}$.
Consequently,
\begin{equation*}
    \mathrm{Jac}\,\xseps(\theta)
    =-\left(H_{\varepsilon}+\varepsilon I_n\right)^{-1}M_{\varepsilon}
    =-\left(H_{\varepsilon}+\varepsilon I_n\right)^{-1}
    H_{\varepsilon}W_{\varepsilon}.
\end{equation*}
Since $H_{\varepsilon}=A\tp A\succeq0$, the eigenvalues of $(H_{\varepsilon}+\varepsilon I_n)^{-1}H_{\varepsilon}$ are of the form $\lambda/(\lambda+\varepsilon)$ with $\lambda\geq0$. Hence $0\preceq\left(H_{\varepsilon}+\varepsilon I_n\right)^{-1}H_{\varepsilon}\preceq I_n$, and therefore $\|\mathrm{Jac}\,\xseps(\theta)\|\leq\|W_{\varepsilon}\|=\|A^\dagger\|$, uniformly in $\varepsilon>0$. Since $H_{\varepsilon}\equiv A\tp A$ is fixed along the trajectory,
\begin{equation*}
    \left(H_{\varepsilon}+\varepsilon I_n\right)^{-1}H_{\varepsilon}
    \underset{\etz}{\longrightarrow}
    (A\tp A)^\dagger A\tp A.
\end{equation*}
It follows that
\begin{equation*}
    \mathrm{Jac}\,\xseps(\theta)
    \underset{\etz}{\longrightarrow}
    -(A\tp A)^\dagger(-A\tp)
    =A^\dagger.
\end{equation*}
Thus the singularity of $H_{\varepsilon}=A\tp A$ is not an obstruction per se. The inverse $(H_{\varepsilon}+\varepsilon I_n)^{-1}$ is never applied to an element of $\ker(H_{\varepsilon})$, where it would blow up as $\varepsilon\to 0$, since $M_{\varepsilon}\in\range(H_{\varepsilon})$. Least-Squares is particularly nice because its blocks are fixed and do not depend on $\varepsilon$; in general a uniform control will be necessary.

\subsection{Uniform Range Condition}
The Least-Squares calculation rests on two ingredients: $H_{\varepsilon}\succeq0$ and a uniformly bounded factorization $M_{\varepsilon}=H_{\varepsilon}W_{\varepsilon}$. This inspires the following assumption requiring both along the regularized solutions.
\begin{enumerate}[label={\bf(R)}]
\item \label{assum:range} {[Uniform range condition]} Let $O\subset\Theta$ be open and definable, fix $\bar{\eps}>0$, and let $J_F$ be a definable conservative Jacobian for $F$. For every compact set $K\subset O$, there exists $C_K>0$ such that, for every $\theta\in K$, $0<\varepsilon\leq\bar{\eps}$, and $(H_{\varepsilon}\ M_{\varepsilon})\in J_F(\xsepst,\theta)$,
the block $H_{\varepsilon}$ is symmetric PSD and
\[
    M_{\varepsilon}=H_{\varepsilon}W_{\varepsilon}
    \qquad\text{for some }W_{\varepsilon}\in\mbR^{n\times m}
    \text{ satisfying }\|W_{\varepsilon}\|\leq C_K.
\]
\end{enumerate}
The factorization says that the columns of $M_{\varepsilon}$ lie in $\range(H_{\varepsilon})$, while the bound controls a choice of preimage uniformly in $\theta$, $\varepsilon$, and all admissible blocks. For $J_F=\partial F$, convexity of $f(\cdot,\theta)$ ensures that the $H$-blocks are PSD. For a more general conservative Jacobian, symmetry and positive semidefiniteness of its $x$-blocks are required.
\begin{example}
    [Assumption \ref{assum:range} fails] Consider the one-dimensional problem with $f(x, \theta) = x^4/4 + \theta x$. The optimality condition is $F(x, \theta) = x^3 + \theta = 0$, so $\Heps = 3(\xsepst)^2$ and $\Meps = 1$. Thus, for $\theta = 0$, $\xseps(0) = 0$ so $1 = \Meps \notin \range(\Heps) = \{0\}$. This agrees with the fact that $\minsol(\theta) = -\theta^{1/3}$ is not locally Lipschitz at $0$, thus not path-differentiable. 
\end{example}
\begin{theorem}
    [Regularity from uniformly bounded implicit Jacobians]\label{thm_loc_bound_deriv_imply_cont}
    Let $O\subset\Theta$ be open and definable and assume that \ref{assum:f_reg} and \ref{assum:f_def} hold. For every $\varepsilon>0$, let $J_{\xseps}$ be a conservative Jacobian of $\xseps$ on $O$. Suppose that, for every compact set $K\subset O$, there exist $C_K>0$ and $\varepsilon_K>0$ such that
    \begin{equation}\label{eq:bounded_derivates}
        \sup\left\{\|V\|:\theta\in K,\ 0<\varepsilon\leq\varepsilon_K,\ V\in J_{\xseps}(\theta)\right\}
        \leq C_K.
    \end{equation}
    Then $\minsol$ is locally Lipschitz on $O$, the joint selection $\xs$ is continuous on $O\times\mbR_+$, and $\xseps\to\minsol$ uniformly on compact subsets of $O$ as $\etz$. If \ref{assum:range} holds and $J_{\xseps}$ is chosen as the implicit conservative Jacobian in \eqref{eq:regularized_implicit_jacobian}, then \eqref{eq:bounded_derivates} holds with $\varepsilon_K=\bar\eps$ and all subsequent conclusions follow.
\end{theorem}
\begin{proof}
    We first verify the last claim. Assume \ref{assum:range}, take $J_{\xseps}$ from \eqref{eq:regularized_implicit_jacobian}, fix a compact set $K\subset O$, and let $\theta\in K$, $0<\varepsilon\leq\bar\eps$, and $V\in J_{\xseps}(\theta)$. There is a block $(H_{\varepsilon}\ M_{\varepsilon})\in J_F(\xsepst,\theta)$ such that
    \[
        V=-(H_{\varepsilon}+\varepsilon I_n)^{-1}M_{\varepsilon}
        =-(H_{\varepsilon}+\varepsilon I_n)^{-1}H_{\varepsilon}W_{\varepsilon}.
    \]
    The eigenvalues of $(H_{\varepsilon}+\varepsilon I_n)^{-1}H_{\varepsilon}$ are of the form $\lambda/(\lambda+\varepsilon)\in[0,1]$. Hence
    \[
        0\preceq(H_{\varepsilon}+\varepsilon I_n)^{-1}H_{\varepsilon}\preceq I_n
        \quad\text{and}\quad
        \|V\|\leq\|W_{\varepsilon}\|\leq C_K
    \]
    as claimed.

    We now use only \eqref{eq:bounded_derivates}. Fix $\theta_0\in O$ and choose $r>0$ such that the compact convex set $K_0:=\overline B(\theta_0,r)$ ($\overline{B}$ is the closed ball) is contained in $O$. For $\theta,\eta\in K_0$, let $\mu(t)=\eta+t(\theta-\eta)$. If $0<\varepsilon\leq\varepsilon_{K_0}$, path-differentiability gives, for almost every $t\in[0,1]$,
    \[
        \frac{d}{dt}\xseps(\mu(t))
        =V(t)(\theta-\eta),
        \qquad
        V(t)\in J_{\xseps}(\mu(t)).
    \]
    Using this and invoking \eqref{eq:bounded_derivates} gives
    \begin{align*}
        \|x_\varepsilon^\star(\theta)-x_\varepsilon^\star(\eta)\|=\left\|\int_0^1
        \frac{d}{dt}x_\varepsilon^\star(\mu(t))
        dt\right\|
        &\leq
        \int_0^1
        \left\|
        \frac{d}{dt}x_\varepsilon^\star(\mu(t))
        \right\|dt \leq
        C_{K_0}\|\theta-\eta\|.
    \end{align*}
    Since $\xseps\to\minsol$ as $\etz$ pointwise in $\theta$ (see Theorem~\ref{thm_Browder}), continuity of $\|\cdot\|$ finally gives
    \begin{equation*}
            \|\minsol(\theta)-\minsol(\eta)\|
                =\lim\limits_{\etz}\|\xseps(\theta)-\xseps(\eta)\| \leq C_{K_0} \|\theta - \eta\|.
    \end{equation*}
    Thus, $\minsol$ is locally Lipschitz on $O$. 

    Consider now a sequence $(\theta_k, \varepsilon_k) \rightarrow (\theta_0, 0)$ with $\varepsilon_k \geq 0$. For $k$ large enough, $\theta_k$ and $\theta_0$ lie in $K_0$, hence
    \begin{equation*}
        \begin{aligned}
            \|\xs(\theta_k, \varepsilon_k) - \xs(\theta_0,0)\|
            &\leq \|\xs_{\varepsilon_k}(\theta_k)-\xs_{\varepsilon_k}(\theta_0)\| + \|\xs_{\varepsilon_k}(\theta_0)-\minsol(\theta_0)\|\\
            &\leq C_{K_0} \|\theta_k - \theta_0\| + \|\xs_{\varepsilon_k}(\theta_0) - \minsol(\theta_0)\| \rightarrow 0
        \end{aligned}
    \end{equation*}
    because $\theta_k\to\theta_0$ and $\xs(\theta_0,\varepsilon_k)\to\minsol(\theta_0)$.
    This proves continuity on $O\times\{0\}$. At points with $\varepsilon>0$, joint continuity in $(\theta,\varepsilon)$ follows from strong convexity and the implicit function theorem applied with both variables as parameters \citep{bolte_pauwels_silveti-falls_2024}. Hence $\xs$ is continuous on $O\times\mbR_+$.

    Finally, let $K\subset O$ be compact and fix $a>0$. The restriction of $\xs$ to the compact set $K\times[0,a]$ is uniformly continuous. Since $\xs(\theta,0)=\minsol(\theta)$, it follows that
    \begin{equation*}
        \sup_{\theta\in K}\|\xseps(\theta)-\minsol(\theta)\|
        \longrightarrow0
        \qquad\text{as }\varepsilon\downarrow0
    \end{equation*}
    as desired.
\end{proof}
We point out that the arguments above give joint continuity in $(\theta,\varepsilon)$, but not joint local Lipschitz continuity.

\begin{remark}[Definability of the construction]\label{rem_verif_assumptions}
\propref{prop:def_tikhonov} gives the definability of $\xs$, while \propref{prop_defin_derivative} gives the definability of $J_{\xs}$ and the conservativity of each fixed-$\varepsilon$ slice $J_{\xseps}$. These results verify $(i)$ and $(ii)$ of \thmref{thm_Schechtman}. The compact-uniform convergence in $(iii)$ follows from \thmref{thm_loc_bound_deriv_imply_cont}, while Assumption~\ref{assum:range} gives the local bound in $(iv)$.
\end{remark}

\begin{theorem}[Path-differentiability under the uniform range condition]\label{thm_general}
    Assume that \ref{assum:f_reg}, \ref{assum:f_def} and \ref{assum:range} hold, and let $J_{\xseps}$ be the implicit conservative Jacobian in \eqref{eq:regularized_implicit_jacobian} for $0<\varepsilon\leq\bar\eps$. Then $\minsol$ is path-differentiable on $O$. Moreover, $J_{\minsol}$ defined in \eqref{eq:J_min} is a definable conservative Jacobian of $\minsol$.
\end{theorem}

\begin{proof}
    By \remref{rem_verif_assumptions}, all the hypotheses of \thmref{thm_Schechtman} hold for the restricted family on $O\times]0,\bar\eps]$. Hence $J_{\minsol}$ is a definable conservative Jacobian of $\minsol$, which is therefore path-differentiable on $O$.
\end{proof}

By joint continuity, the condition $\xs(\theta_k,\varepsilon_k)\to\minsol(\theta)$ used in \eqref{eq:J_min} holds automatically in our settings. Thus $J_{\minsol}$ can be rewritten as
\[
    J_{\minsol}\colon\theta
    \rightrightarrows\left\{V\in\mbR^{n\times m}:
    \begin{array}{l}
        \exists\ (\theta_k,\varepsilon_k,V_k)_{k\in\NN}\ \text{with }
        \theta_k\in O,\ 0<\varepsilon_k\leq\bar\eps,\\
        \theta_k\to\theta,\ \varepsilon_k\downarrow0,\ V_k\to V,\ V_k\in J_{\xs}(\theta_k,\varepsilon_k)
    \end{array}
    \right\}.
\]

\subsection{Pseudoinverse Formula}
\thmref{thm_general} bounds the regularized conservative Jacobians in a way that ensures path-differentiability but does not identify their cluster points. The Least-Squares example suggests the formula $-H^{\dagger}M$ but, as the following example shows, this formula can fail when curvature vanishes along the regularized trajectory.

\begin{example}[Slowly vanishing curvature]\label{ex_vanishing_eigenvalues}
    Let $O:=]-1,1[$ and $f(x,\theta):=\psi(x-\theta)$, where
    \[
        \psi(t):=\begin{cases}
            t^4/4, & |t|\leq 1,\\
            |t|-3/4, & |t|>1.
        \end{cases}
    \]
    The function $\psi$ is convex, semialgebraic, and belongs to $\mC^{1,1}(\mathbb{R})$. Its unique minimizer is zero, so $\minsol(\theta)=\theta$ and $\partial\minsol/\partial\theta\equiv 1$. At the unregularized solution, however, the two blocks $H$ and $M$ are both $0$ (since $\theta \in O = ]-1, +1[$), and hence $-H^{\dagger}M=0$.

    Fix $\theta\neq 0$. The regularized solution lies between $0$ and $\theta$, so $|\xsepst-\theta|<1$ and the quartic branch of $\psi$ is active. The optimality condition and the corresponding blocks are
    \[
        (\xsepst-\theta)^3+\varepsilon\xsepst=0,
        \qquad H_{\varepsilon}=3(\xsepst-\theta)^2,
        \qquad M_{\varepsilon}=-H_{\varepsilon},
        \qquad W_{\varepsilon}=-1.
    \]
    Rearranging the optimality condition and dividing by $\varepsilon$ gives
    \[
        \left(\frac{\xsepst-\theta}{\varepsilon^{1/3}}\right)^3=-\xsepst.
    \]
    Since $\xsepst\to\theta$ and the real cube-root function is continuous, it follows that
    \[
        \frac{\xsepst-\theta}{\varepsilon^{1/3}}\longrightarrow-\sqrt[3]{\theta},
        \qquad\text{or equivalently}\qquad
        \xsepst - \theta\sim-\sqrt[3]{\theta}\,\varepsilon^{1/3}.
    \]
    Therefore,
    \[
        H_{\varepsilon}\sim3|\theta|^{2/3}\varepsilon^{2/3},
        \qquad -(H_{\varepsilon}+\varepsilon)^{-1}M_{\varepsilon}\longrightarrow 1.
    \]
    Thus the regularized derivatives recover the derivative of $\minsol$, although the limiting blocks alone give $0$. The mismatch is due to $H_{\varepsilon}$ vanishing more slowly than $\varepsilon$.
\end{example}

In general, consider a sequence $\varepsilon_k\downarrow 0$ and convergent selected blocks $H_k$. After passing to a subsequence, we may suppose that each ratio $\lambda_i(H_k)/\varepsilon_k$ has a limit in $[0,+\infty]$. For an eigenvalue such that $\lambda_i(H_k)\to 0$, the corresponding term has one of the following limits:
\[
    \frac{\lambda_i(H_k)}{\varepsilon_k+\lambda_i(H_k)}\longrightarrow
    \left\{\begin{array}{lll}
        0, & \lambda_i(H_k)=o(\varepsilon_k), & \textnormal{fast vanishing},\\
        c/(1+c), & \lambda_i(H_k)/\varepsilon_k\to c\in]0,+\infty[, & \textnormal{critically vanishing},\\
        1, & \varepsilon_k=o(\lambda_i(H_k)), & \textnormal{slow vanishing}.
    \end{array}\right.
\]
Eigenvalues with a positive limit also have multiplier one, as prescribed by $H^{\dagger}H$, while identically zero eigenvalues belong to the fast regime and therefore cause no difficulty. Thus agreement with $H^{\dagger}H$ requires every eigenvalue converging to zero to vanish faster than $\varepsilon_k$.

To summarize, if $p>1$, $0<q<1$, $c>0$, and $H_{\varepsilon}=\diag(0,\varepsilon^p,c\varepsilon,\varepsilon^q,1)$, then
\[
    (H_{\varepsilon}+\varepsilon I_n)^{-1}H_{\varepsilon}\longrightarrow\diag(0,0,c/(1+c),1,1)\neq H^{\dagger}H=\diag(0,0,0,0,1).
\]
The critical and slow regimes can retain a nonzero contribution in directions that belong to $\ker H$ at the limit. This is the contribution represented by the residual $S$ below.

\begin{remark}[Reachable set]\label{rem_reachability}
    Fix $\theta\in O$. Since $\xsepst\to\minsol(\theta)$ and $J_F$ is locally bounded, every sequence of blocks along the regularized trajectory admits a convergent subsequence. The closed graph property of $J_F$ then guarantees its limit in $J_F(\minsol(\theta),\theta)$. We call the set of all such limits the \emph{reachable} set denoted by
    \[
        \mR(\theta):=\left\{(H\ M)\in J_F(\minsol(\theta),\theta):
        \begin{array}{l}
            \exists\ \varepsilon_k\downarrow0,\ (H_k\ M_k)\in J_F(x_{\varepsilon_k}^{\star}(\theta),\theta),\\
            (H_k\ M_k)\to(H\ M)
        \end{array}\right\}.
    \]
    A corresponding cluster point of the regularized derivatives belongs to $J_{\minsol}(\theta)$ by taking $\theta_k\equiv\theta$ in \eqref{eq:J_min}. However, \eqref{eq:J_min} also permits $\theta_k\to\theta$ so that the full set $J_{\minsol}(\theta)$ may contain additional elements.
\end{remark}

Example~\ref{ex_vanishing_eigenvalues} shows that vanishing curvature from the limiting block can affect the $-H^{\dagger}M$ formula. The next theorem makes this precise by introducing a residual $S$; its proof is given in \secref{ss_proof_residual_formula}.

\begin{theorem}[Residual formula]\label{thm_formula_memory}
    Assume that \ref{assum:f_reg}, \ref{assum:f_def} and \ref{assum:range} hold, and fix $\theta\in O$. Let $\varepsilon_k\downarrow0$ with $0<\varepsilon_k\leq\bar\eps$, and select $(H_k\ M_k)\in J_F(x_{\varepsilon_k}^{\star}(\theta),\theta)$. Set
    \[
        R_k:=(H_k+\varepsilon_k I_n)^{-1}H_k,\qquad
        W_k:=H_k^{\dagger}M_k,\qquad
        V_k:=-(H_k+\varepsilon_k I_n)^{-1}M_k=-R_kW_k.
    \]
    Then $V_k\in J_{\xs}(\theta,\varepsilon_k)$, and the sequences $(W_k)_k$ and $(V_k)_k$ are bounded. If $V_k\to V$ along a subsequence, then, after passing to a further common subsequence,
    \[
        (H_k\ M_k)\to(H\ M)\in\mR(\theta),\qquad W_k\to W,\qquad R_k\to R.
    \]
    Moreover, $M=HW$ and $S:=R-H^{\dagger}H=\Pi_{\ker H}R\,\Pi_{\ker H}$ with $0\preceq S\preceq\Pi_{\ker H}$,
    while the limit satisfies
    \[
        V=-H^{\dagger}M-SW\in J_{\minsol}(\theta),
        \qquad HV+M=0.
    \]
\end{theorem}
The matrix $H^{\dagger}H$ is the identity when restricted to $\range(H)$, whereas $S$ acts only on $\ker(H)$. Thus $S$ accounts for the regularized curvature that survives in directions where the limiting curvature is zero. The correction $-SW$ does not affect the limiting equation because $HS=0$. We give conditions in the following subsection to ensure $S=0$ although, for a given limit, the pseudoinverse formula already holds once the weaker condition $SW=0$ is satisfied.
\begin{remark}
    \label{rem_quartic_example} [Illustration on the quartic example] For the quartic example and fixed $\theta\neq0$, we have $H=M=0$, $W=-1$, and $R=S=1$, so the residual formula gives $V=1$. At $\theta=0$, the regularized solution is identically zero, hence the fixed trajectory gives $R=S=V=0$ and therefore $0\in J_{\minsol}(0)$. On the other hand, $1\in J_{\minsol}(\theta)$ for every $\theta$. Since $J_{\minsol}$ has closed graph, letting $\theta\to0$ shows that $1\in J_{\minsol}(0)$ as well. The theorem therefore identifies some elements of $J_{\minsol}$ without necessarily characterizing the whole set. In practice, one admissible element is sufficient for a gradient-type method applied to \eqref{Q}, see \secref{sss_Outer}.
\end{remark}

\subsubsection{Conditions Yielding the Formula}\label{sss_conditions}
We now give two pointwise conditions that force $S=0$, both stated with respect to a fixed $\theta\in O$. The first rules out small positive eigenvalues, while the second allows them provided they vanish faster than $\varepsilon$.
\begin{enumerate}[label={\bf(G)},nosep]
\item \label{assum:eigenvalues_gap} {[Spectral gap]} there exist $\lambda^*>0$ and $\varepsilon_G\in]0,\bar\eps]$ such that, for every $0<\varepsilon\leq\varepsilon_G$ and every $(H_{\varepsilon}\ M_{\varepsilon})\in J_F(\xsepst,\theta)$,
\[\operatorname{spec}(H_{\varepsilon})\subset\{0\}\cup[\lambda^*,+\infty[.
\]
\end{enumerate}
\begin{enumerate}[label={\bf(V)},nosep]
\item \label{assum:eigenvalues_vanishing} {[Fast vanishing]} for every sequence $\varepsilon_k\downarrow0$ with $0<\varepsilon_k\leq\bar\eps$ and every selection $(H_k\ M_k)\in J_F(x_{\varepsilon_k}^{\star}(\theta),\theta)$ such that $(H_k)_k$ converges, the ordered eigenvalues $\lambda_1(H_k)\geq\cdots\geq\lambda_n(H_k)$ satisfy
\[
    \lambda_i(H_k)\to0\quad\Longrightarrow\quad\lambda_i(H_k)=o(\varepsilon_k),
    \qquad i=1,\ldots,n.
\]
\end{enumerate}
\begin{theorem}[Reachable pseudoinverse formula]\label{thm_formula}
    Assume \ref{assum:f_reg}, \ref{assum:f_def} and \ref{assum:range}, fix $\theta\in O$, and suppose that either \ref{assum:eigenvalues_gap} or \ref{assum:eigenvalues_vanishing} holds at $\theta$. Then the conservative Jacobian $J_{\minsol}$ defined in \eqref{eq:J_min} satisfies
    \begin{equation}\label{formula_eigenvalues_controlled}
        \left\{-H^{\dagger}M \colon (H\ M)\in \mR(\theta)\right\}\subset J_{\minsol}(\theta).
    \end{equation}
\end{theorem}
The proof is given in \secref{ss_proof_pseudoinverse_formula}. Condition~\ref{assum:eigenvalues_gap} includes the uniformly strongly convex setting but also permits eigenvalues that are identically zero. It may also be stated as an \textit{eventual rank stability}, in the sense that $\mathrm{rank}(\Heps) = \mathrm{rank}(H)$ for $\varepsilon$ small enough. Without either spectral condition, Assumptions~\ref{assum:f_reg},~\ref{assum:f_def} and~\ref{assum:range} still give bounded derivatives and path-differentiability, but the residual $S$ may be nonvanishing. This significantly generalizes \citet[Statement~3.1]{vicol_lorraine_pedregosa_duvenaud_grosse_2022}, which considered only the fully quadratic setting.

The relationship between these conditions and standard nonsingularity assumptions is direct. Earlier implicit-differentiation results assume nonsingular $H$-blocks at the solution \citep{bolte_le_pauwels_silveti-falls_2021} (thus near the solution by continuity of the determinant), or uniform nonsingularity \citep{blondel_berthet_cuturi_frostig_hoyer_llinares-lopez_pedregosa_vert_2021}, while related range and spectral conditions appear in \citet{pauwels2023derivatives}. If the $H_{\varepsilon}$-blocks are uniformly positive definite on compact parameter sets and the corresponding $M_{\varepsilon}$-blocks are uniformly bounded, then $W_{\varepsilon}=H_{\varepsilon}^{-1}M_{\varepsilon}$ verifies Assumption~\ref{assum:range}, and the same lower bound gives \ref{assum:eigenvalues_gap}. Unlike uniform nonsingularity, however, Assumption~\ref{assum:range} and \ref{assum:eigenvalues_gap} allow rank-deficient blocks.

As a consistency check, \thmref{thm_formula} recovers the Least-Squares calculation above.
\begin{corollary}[Least-Squares]\label{coro_LS}
For the Least-Squares problem, $A^{\dagger}\in J_{\minsol}(\theta)$ for all $\theta\in\mathbb{R}^m$.
\end{corollary}
\begin{proof}
For the classical singleton Jacobian used in the motivating example, $O=\mathbb{R}^m$, $\mR(\theta)=\{(A\tp A\ {-A\tp})\}$, and \ref{assum:eigenvalues_gap} holds because the blocks do not depend on $\varepsilon$. Hence \thmref{thm_formula} gives $-(A\tp A)^{\dagger}(-A\tp)=(A\tp A)^{\dagger}A\tp=A^{\dagger}$.
\end{proof}

\subsubsection{Outer problem}\label{sss_Outer}
Before discussing the case of Huber regression, we discuss what can be achieved for the outer problem once elements of $J_{\minsol}$ are obtained. Let $\mT(\theta) := \mL(\minsol(\theta), \theta)$. 

By \thmref{thm_loc_bound_deriv_imply_cont}, $\theta \mapsto (\minsol(\theta), \theta)$ is locally Lipschitz and definable thus path-differentiable. Assume $\mL$ is locally Lipschitz and definable with conservative Jacobian $J_{\mL}$. Then $\mT$ is also locally Lipschitz and definable thus path-differentiable using the composition rules of conservative calculus \citep{bolte_pauwels_2019}, with conservative Jacobian given by
\[J_{\mT} : \theta \setvalued \left\{(L_x\ L_{\theta})(V; I_m) = L_x V + L_{\theta} : (L_x \ L_{\theta}) \in J_{\mL}(\minsol(\theta), \theta), V \in J_{\minsol}(\theta)\right\}.\]
From there, the stochastic subgradient algorithm for conservative fields \citep{davis_drusvyatskiy_kakade_lee_2018, bolte_pauwels_2019} can be employed for the global problem.

\subsection{Application to Huber Regression}\label{ss_Huber}
Huber regression \citep{huber1964} illustrates how Assumption~\ref{assum:range} can hold even when the Hessian blocks change rank along the regularized trajectory. Let $A\in\mbR^{m\times n}$ and consider
\begin{equation}\tag{Hub}\label{H}
    \min_{x\in\mbR^n}\phi(Ax-\theta),
    \qquad \phi(y):=\sum_{i=1}^m h(y_i),
    \qquad h(t):=\begin{cases}t^2/2,& |t|\leq1,\\ |t|-1/2,& |t|>1.\end{cases}
\end{equation}
The function $\phi\colon\mbR^m\to\mbR$ is convex and semialgebraic, with $\phi\in\mC^{1,1}(\mbR^m)\setminus\mC^2(\mbR^m)$. For $y\in\mbR^m$, every $Q=\diag(q)\in\partial(\nabla\phi)(y)$ is diagonal by separability, with
\[
    q_i\in\begin{cases}
        \{1\},& |y_i|<1,\\
        [0,1],& |y_i|=1,\\
        \{0\},& |y_i|>1.
    \end{cases}
\]
Writing $F(x,\theta)=A\tp\nabla\phi(Ax-\theta)$, we use the definable conservative Jacobian
\[
    J_F(x,\theta):=\left\{(A\tp QA\ {-A\tp Q}):Q\in\partial(\nabla\phi)(Ax-\theta)\right\}.
\]
Along the regularized solution, write $Q_{\varepsilon}\in\partial(\nabla\phi)(A\xsepst-\theta)$ and
\[
    H_{\varepsilon}:=A\tp Q_{\varepsilon}A,
    \qquad M_{\varepsilon}:=-A\tp Q_{\varepsilon}.
\]
The matrices $H_{\varepsilon}$ are symmetric PSD, although their rank may change with $\varepsilon$. We can nonetheless state a bound that is uniform over all the possible matrices $Q_{\varepsilon}$ by defining
\begin{equation}\label{eq:chi_A}
    \chi(A):=\max\left\{\|(A_{R,C})^{-1}\|:|R|=|C|\geq1,\ A_{R,C}\text{ nonsingular}\right\},
\end{equation}
where $R\subset\{1,\ldots,m\}$ and $C\subset\{1,\ldots,n\}$, taking $\chi(A)=0$ when $A=0$. This constant is finite because $A$ has only finitely many square submatrices. This leads to two estimates, given in the next corollary; the first estimate controls the regularized systems uniformly in both $\delta$ and $Q$, while the second gives the corresponding bound at $\delta=0$. The proof in \secref{ss_proof_weighted_gram} first establishes a general diagonal-Gram estimate that also serves for the LASSO.

\begin{corollary}[Uniform bound for weighted Gram systems]\label{coro_scaled_gram_weighted}
Let $A\in\mbR^{m\times n}$. For every diagonal matrix $Q$ with diagonal in $\mathbb{R}^m_+$ and every $\delta>0$,
\[
    \left\|(\delta I_n+A\tp QA)^{-1}A\tp Q\right\|\leq\chi(A)
    \qquad\text{and}\qquad
    \left\|(\sqrt{Q}A)^{\dagger}\sqrt{Q}\right\|\leq\chi(A).
\]
\end{corollary}

Set $B_{\varepsilon}:=\sqrt{Q_{\varepsilon}}A$ and $W_{\varepsilon}:=-B_{\varepsilon}^{\dagger}\sqrt{Q_{\varepsilon}}$. Since $B_{\varepsilon}\tp B_{\varepsilon}B_{\varepsilon}^{\dagger}=B_{\varepsilon}\tp$,
\[
    H_{\varepsilon}W_{\varepsilon}
    =-B_{\varepsilon}\tp B_{\varepsilon}B_{\varepsilon}^{\dagger}\sqrt{Q_{\varepsilon}}
    =-B_{\varepsilon}\tp\sqrt{Q_{\varepsilon}}
    =M_{\varepsilon}.
\]
The second estimate in \corref{coro_scaled_gram_weighted} gives then $\|W_{\varepsilon}\|=\| (\sqrt{Q_{\varepsilon}}A)^{\dagger}\sqrt{Q_{\varepsilon}}\|\leq\chi(A)$.
Thus both the range factorization and its bound are uniform over $\theta$, $\varepsilon$, and every admissible $Q_{\varepsilon}$. This is summarized in the next corollary, whose proof is given in \secref{ss_proof_huber_pathdiff}.

\begin{corollary}[Huber regression]\label{coro_Huber_app}
For every fixed $A\in\mbR^{m\times n}$, Assumptions~\ref{assum:f_reg}, \ref{assum:f_def} and~\ref{assum:range} hold for \eqref{H} with $\Theta=O=\mbR^m$ and any $\bar\eps>0$; the range factors satisfy the global bound $\|W_{\varepsilon}\|\leq\chi(A)$. Consequently, $\minsol\colon\mbR^m\to\mbR^n$ is path-differentiable, and $J_{\minsol}$ defined in \eqref{eq:J_min} is a definable conservative Jacobian.
\end{corollary}

For $\theta\in\mbR^m$, define the reachable Huber matrices by
\begin{equation}\label{eq:huber_matrices}
    \mR^{\mathrm{Hub}}(\theta):=\left\{Q\in\partial(\nabla\phi)(A\minsol(\theta)-\theta):
    \begin{array}{l}
        \exists\ \varepsilon_k\downarrow0,\ Q_{\varepsilon_k}\in\partial(\nabla\phi)(Ax_{\varepsilon_k}^{\star}(\theta)-\theta),\\
        Q_{\varepsilon_k}\to Q
    \end{array}\right\}.
\end{equation}
These are precisely the matrices that can be obtained by keeping $\theta$ fixed and following the regularized solutions.

\begin{proposition}[Limiting derivatives for Huber regression]\label{prop_huber_some_limits}
For every $\theta\in\mbR^m$ and every $Q\in\mR^{\mathrm{Hub}}(\theta)$,
\[
    (A\tp QA)^{\dagger}A\tp Q
    =(\sqrt{Q}A)^{\dagger}\sqrt{Q}
    \in J_{\minsol}(\theta).
\]
\end{proposition}
The proof is given in \secref{ss_proof_huber_limits}. \propref{prop_huber_some_limits} thus gives the pseudoinverse element associated with every reachable $Q$. It does not identify every cluster point of the regularized derivatives because the entries of $Q_{\varepsilon}$ at the transition points of the Huber loss may vary with $\varepsilon$, and the eigenvalues of $A\tp Q_{\varepsilon}A$ may enter the critical or slow regimes described in subsection~\ref{sss_conditions}.

\section{Composite Problems}\label{s_involved_pbs}
We move on to composite problems involving sums of parametric functions. Parallel to our presentation of \secref{s_C11}, we first introduce the notation and regularized implicit-differentiation calculus in the general setting. We then specialize it to the LASSO, whose uniform derivative bound will motivate the reduced range conditions used in the general results. Finally, we give some comments on potential convergence rates. 

Let $O\subset\Theta$ be open and definable, and consider two functions $f:\mbR^n\times O\to\mbR$ and $g:\mbR^n\times O\to\mbR\cup\{+\infty\}$. For every $\theta\in O$, consider
\begin{equation*}
    \tag{P-Comp}\label{pb_composite}
    \min_{x\in\mbR^n}\big\{f(x,\theta)+g(x,\theta)\big\}
\end{equation*}
under the following two assumptions.
\begin{enumerate}[label={\bf(D')}]
\item \label{assum:f_c} {[Composite definability]} The functions $f$ and $g$ as well as the set $\Theta$
are definable. This implies $\nabla_x f$ and $\mathrm{prox}_{\gamma g}$ for $\gamma > 0$ are definable, see \remref{rem_definability_composite}. 
\end{enumerate}
\begin{enumerate}[label={\bf(Reg')}]
\item \label{assum:reg_c} {[Composite regularity]} For every $\theta\in O$,
$f(\cdot,\theta)\in\mC^{1,1}(\mbR^n)$ and $f(\cdot,\theta)\in\Gamma_0(\mbR^n)$, $g(\cdot,\theta)\in\Gamma_0(\mbR^n)$, and the solution set $\mS_c(\theta):=\argmin_{x\in\mbR^n}\{f(x,\theta)+g(x,\theta)\}$ is nonempty. The map
$\nabla_x f:\mbR^n\times O\to\mbR^n$ is jointly locally Lipschitz and, for
some $L\geq0$,
\[
    \|\nabla_x f(x,\theta)-\nabla_x f(x',\theta)\|
    \leq L\|x-x'\|
    \qquad
    \text{for all }x,x'\in\mbR^n\text{ and }\theta\in O.
\]
Fix $\bar\eps>0$ and a positive stepsize $\gamma<2/(L+2\bar\eps)$. For this $\gamma$, the parameterized proximal map
\[
    G^\gamma:\mbR^n\times O\to\mbR^n,
    \qquad
    G^\gamma(z,\theta):=\prox_{\gamma g(\cdot,\theta)}(z),
\]
is jointly locally Lipschitz.
\end{enumerate}
\begin{remark}[Definability of $\nabla_x f$ and $G^\gamma$] \label{rem_definability_composite}
The graphs of $\nabla_x f$ and $G^\gamma$ admit first-order descriptions in terms of $f$ and $g$. By closure of definable sets under first-order quantification, or equivalently under Boolean operations and coordinate projection, both maps are definable; in the semialgebraic case, this is the Tarski--Seidenberg theorem (see \secref{ss_definability_intro} and \citet{coste1999introduction, vandendries_miller_1996}). Their local Lipschitz continuity in \ref{assum:reg_c} therefore makes their joint Clarke Jacobians conservative Jacobians.
\end{remark}

Under \ref{assum:f_c}, $\mS_c(\theta)$ is a nonempty closed convex set. We
denote its minimal-norm element by
\[
    \minsol(\theta)
    :=\argmin_{x\in\mS_c(\theta)}\frac12\|x\|_2^2.
\]
For $0<\varepsilon\leq\bar\eps$, define the Tikhonov-regularized solution
\begin{equation*}
    \tag{$\mbox{P-Comp}_\eps$}\label{pb_composite_epsilon}
    \xsepst
    :=\argmin_{x\in\mbR^n}
    \left\{f(x,\theta)+g(x,\theta)
    +\frac{\varepsilon}{2}\|x\|_2^2\right\}.
\end{equation*}
The regularized objective is $\varepsilon$-strongly convex, so this solution is
single-valued. For the fixed $\gamma$ in \ref{assum:reg_c}, define the forward
step, the forward--backward operator, and its fixed-point residual by
\begin{equation}\label{composite_fixed_point}
\begin{aligned}
    z_\varepsilon(x,\theta)
    &:=x-\gamma\big(\nabla_x f(x,\theta)+\varepsilon x\big),\\
    \varphi_\varepsilon^\gamma(x,\theta)
    &:=G^\gamma\big(z_\varepsilon(x,\theta),\theta\big),
    &\qquad
    \Phi_\varepsilon^\gamma(x,\theta)
    &:=x-\varphi_\varepsilon^\gamma(x,\theta).
\end{aligned}
\end{equation}
Fermat's rule and the standard proximal characterization
\citep[see][]{bauschke_combettes_2017} give
\[
    x\text{ solves \eqref{pb_composite_epsilon}}
    \quad\Longleftrightarrow\quad
    x=\varphi_\varepsilon^\gamma(x,\theta)
    \quad\Longleftrightarrow\quad
    \Phi_\varepsilon^\gamma(x,\theta)=0.
\]

We use the residual equation $\Phi_\varepsilon^\gamma(x,\theta)=0$ to
characterize the solutions whose dependence on $\theta$ we wish to
differentiate. Write the joint Clarke Jacobians blockwise as
\[
    (H^f\ M^f)\in J_{\nabla_x f}(x,\theta)
    :=\partial(\nabla_x f)(x,\theta),
    \qquad
    (Q\ P)\in J_{G^\gamma}(z,\theta)
    :=\partial G^\gamma(z,\theta),
\]
where $H^f,Q\in\mbR^{n\times n}$ and $M^f,P\in\mbR^{n\times m}$. To keep
the $x$- and $\theta$-blocks from each joint Clarke selection paired, set
\[
    \mD_\varepsilon(x,\theta)
    :=\left\{
    \begin{array}{l|l}
    D=(Q,P,H^f,M^f)
    & (Q\ P)\in J_{G^\gamma}(z_\varepsilon(x,\theta),\theta),\\[-2pt]
    & (H^f\ M^f)\in J_{\nabla_x f}(x,\theta)
    \end{array}
    \right\}.
\]
For $D\in\mD_\varepsilon(x,\theta)$, define
\begin{equation}\label{composite_sets_matrices}
\begin{aligned}
    \Heps(D)
    &:=I_n-Q\big((1-\gamma\varepsilon)I_n-\gamma H^f\big),\\
    \Meps(D)
    &:=\gamma QM^f-P.
\end{aligned}
\end{equation}
The chain rule applied to these conservative mappings gives a conservative Jacobian of the residual
\begin{equation}\label{composite_constructed_field}
    J_{\Phi_\varepsilon^\gamma}(x,\theta)
    :=\big\{(\Heps(D)\ \Meps(D)):D\in\mD_\varepsilon(x,\theta)\big\}.
\end{equation}

\begin{proposition}[Conservative Jacobian of the regularized composite problem]
\label{prop_composite_regularized}
Assume \ref{assum:f_c} and \ref{assum:reg_c}. For every
$0<\varepsilon\leq\bar\eps$, the map $\theta\mapsto\xsepst$ is
path-differentiable on $O$, with the definable conservative Jacobian
\begin{equation}\label{composite_regularized_jacobian}
    J_{\xseps}(\theta)
    :=\left\{
    -\Heps(D)^{-1}\Meps(D):
    D\in\mD_\varepsilon(\xsepst,\theta)
    \right\}.
\end{equation}
\end{proposition}
The proof is given in \secref{ss_proof_composite_regularized} and is a
straightforward application of
\citet[Theorem~3.5 and Remark~3.6]{bolte_pauwels_silveti-falls_2024}.

\subsection{LASSO}\label{ss_LASSO}

Fix a matrix $A\in\mbR^{m\times n}$ and a penalty parameter $\lambda\in\mbR$, and regard the observation vector $\theta\in\mbR^m$ as the parameter. For every $\theta$, consider the LASSO problem \citep{tibshirani_1996, tibshirani_2011, osborne_presnell_turlach_2000, tibshirani_2013}
\begin{equation*}\tag{LASSO}\label{LASSO}
    \min_{x\in\mbR^n}
    \left\{
        \frac{1}{2}\|Ax-\theta\|_2^2
        + e^\lambda\|x\|_1
    \right\}.
\end{equation*}
This is an instance of \eqref{pb_composite} with $O=\mbR^m$, $f(x,\theta)=\frac12\|Ax-\theta\|_2^2$, and $g(x,\theta)=e^\lambda\|x\|_1$. Both functions are semialgebraic, hence definable, and their $x$-sections belong to $\Gamma_0(\mbR^n)$. Since $e^\lambda>0$, the LASSO objective is coercive and therefore attains its minimum. Thus \ref{assum:f_c} holds.

For $0<\varepsilon\leq\bar\eps$, its Tikhonov regularization, also called elastic net \citep{zou_hastie_2005,mairal_bach_ponce_2011}, is
\begin{equation*}\tag{$\mbox{LASSO}_{\eps}$}\label{LASSO_eps}
    \xsepst
    :=
    \argmin_{x\in\mbR^n}
    \left\{
        \frac{1}{2}\|Ax-\theta\|_2^2
        + e^\lambda\|x\|_1
        + \frac{\varepsilon}{2}\|x\|_2^2
    \right\}.
\end{equation*}
The gradient $\nabla_x f(x,\theta)=A\tp(Ax-\theta)$ is jointly affine and is
$\|A\|^2$-Lipschitz in $x$, while
$G^\gamma=\prox_{\gamma e^\lambda\|\cdot\|_1}$ is the soft-thresholding
operator and is jointly globally Lipschitz. Thus \ref{assum:reg_c} also holds,
and the fixed stepsize chosen above must satisfy
$0<\gamma<2/(\|A\|^2+2\bar\eps)$. Substituting the blocks
$H^f=A\tp A$, $M^f=-A\tp$, and $P=0$ into
\propref{prop_composite_regularized} gives the following formula.

\begin{corollary}[Conservative Jacobian of the regularized LASSO]
\label{coro_LASSO_setup}
For every $0<\varepsilon\leq\bar\eps$, the map
$\theta\mapsto\xsepst$ is path-differentiable, with conservative
Jacobian
\begin{equation}\label{LASSO_Jxseps}
    J_{\xseps}(\theta)
    :=
    \left\{
        \left(I_n-Q\big(I_n-\gamma(A\tp A+\varepsilon I_n)\big)\right)^{-1}
        \gamma QA\tp:
        Q\in\partial\prox_{\gamma e^\lambda\|\cdot\|_1}
        \big(z_{\varepsilon}(\xsepst,\theta)\big)
    \right\}.
\end{equation}
\end{corollary}

The coordinatewise formula for the Clarke Jacobian of the soft-thresholding operator $\prox_{\gamma e^{\lambda}\|\cdot\|_1}$ is 
\[
\partial\prox_{\gamma e^\lambda\|\cdot\|_1}(z)
=\left\{\diag(q):q\in\mbR^n,\quad
q_i\in
\begin{cases}
\{1\}, & \mathrm{if}\ |z_i|>\gamma e^\lambda,\\
[0,1], & \mathrm{if}\ |z_i|=\gamma e^\lambda,\\
\{0\}, & \mathrm{if}\ |z_i|<\gamma e^\lambda,
\end{cases}
 \forall\ i\in[1:n]\right\}.
\]
This formula follows from the separability of the proximal map; the corresponding calculation is written in \corref{coro_CF_l1}. Along the regularized trajectory, an admissible selection
\[
    \Qeps\in
    \partial\prox_{\gamma e^\lambda\|\cdot\|_1}
    \big(z_\varepsilon(\xsepst,\theta)\big)
\]
may vary with $\varepsilon$. In particular, neither its support nor the positive diagonal entries in $\Qeps$ need remain fixed as $\varepsilon\downarrow0$. This is not an issue if we can show an estimate uniform over all admissible selections. The following proposition proves the stronger bound over the entire set $\diag([0,1]^n)$ which always contains the admissible matrices $\Qeps$.
\begin{proposition}
    [LASSO uniform derivative bound]\label{prop_LASSO_unif_bound} With $\chi(A)$ defined in \eqref{eq:chi_A}, for every $0<\varepsilon\leq\bar\eps$ and every $Q \in \diag([0, 1]^n)$,
    \[\left\|\left(I_n - Q(I_n - \gamma (A\tp A + \varepsilon I_n))\right)^{-1}\gamma Q A\tp\right\| \leq \chi(A).\]
    In particular, condition $(iv)$ of \thmref{thm_Schechtman} holds for the LASSO with the bound $C_K = \chi(A)$, for every compact set $K$.
\end{proposition}
The proof is given in \secref{ss_proof_LASSO_uniform_bound}. Although the
residual $x$-block being inverted is generally nonsymmetric, the resulting
derivative matrix admits a symmetric reduction to the positive support of $Q$. If $Q=0$,
this matrix vanishes. Otherwise, set
\[
    S_*:=S_*(Q):=\{i:Q_{ii}>0\},
    \qquad
    Q_*:=Q_{S_*,S_*}\succ0.
\]
The rows outside $S_*$ vanish, while the restriction to the rows in $S_*$ is
\[
    \left(
        \frac{Q_*^{-1}-I_{|S_*|}}{\gamma}
        +\varepsilon I_{|S_*|}
        +A_{:,S_*}\tp A_{:,S_*}
    \right)^{-1}
    A_{:,S_*}\tp .
\]
The diagonal regularizer in this expression is positive definite, but as $Q$
and $\varepsilon$ vary, its smallest entry may tend to zero and its largest
entry may diverge. Nevertheless, \lemref{lem_scaled_gram_bound} bounds this
block by $\chi(A_{:,S_+})\leq\chi(A)$. Thus the estimate is uniform in the
positive support and the fractional entries of $Q$, as well as in
$\varepsilon$.

\subsection{Uniform Control in the General Case}\label{ss_decomposition_general}
We now return to the general composite problem. The LASSO calculation shows
that the $x$-blocks in \eqref{composite_constructed_field} need not be
symmetric, even though they arise from convex data. We therefore expose the
positive semidefinite system carried by $\range(Q)$ before formulating the
uniform assumptions.
The blocks inherit structure from the assumed convexity of the problem. By
\thmref{thm_CF_prox}, $Q$ is symmetric with
$0\preceq Q\preceq I_n$. At every point where $\nabla_x f$ is
differentiable, its $x$-block is a symmetric Hessian satisfying
$0\preceq\nabla_{xx}^2f\preceq LI_n$ by convexity and the $L$-Lipschitz
continuity of $\nabla_x f$. These properties are preserved under limits and
convex combinations, so every $H^f$ is symmetric with
$0\preceq H^f\preceq LI_n$. In particular,
$\Pi_Q:=QQ^\dagger=Q^\dagger Q$ is the orthogonal projection onto
$\range(Q)$, while $I_n-\Pi_Q$ is the orthogonal projection onto $\ker(Q)$.
For $D=(Q,P,H^f,M^f)$, define
\begin{equation}
    \label{composite_Sigma_V}
    \begin{aligned}
        \Sigma(D)
        &:=Q^\dagger-\Pi_Q+\gamma\Pi_QH^f\Pi_Q,\\
        V(D)
        &:=Q^\dagger\Meps(D)
          -\gamma\Pi_QH^f(I_n-\Pi_Q)\Meps(D),\\
        T_\varepsilon(D)
        &:=\Sigma(D)+\gamma\varepsilon\Pi_Q.
    \end{aligned}
\end{equation}
These matrices control the inversion of the $x$-blocks through the following
lemma, proven in \secref{ss_proof_projected_inversion}.
\begin{lemma}[Projected inversion]\label{lem_projected_inversion}
Let $\gamma,\varepsilon>0$, let $Q,H^f\in\mbR^{n\times n}$ be symmetric
with $0\preceq Q\preceq I_n$ and $H^f\succeq0$, and let
$M\in\mbR^{n\times m}$. Set
\[
\begin{aligned}
    H&:=I_n-Q\big((1-\gamma\varepsilon)I_n-\gamma H^f\big),
    &\Pi&:=QQ^\dagger,\\
    \Sigma&:=Q^\dagger-\Pi+\gamma\Pi H^f\Pi,
    &V&:=Q^\dagger M-\gamma\Pi H^f(I_n-\Pi)M,\\
    T_\varepsilon&:=\Sigma+\gamma\varepsilon\Pi.
\end{aligned}
\]
Then $\Sigma$ is symmetric positive semidefinite and
$\Pi\Sigma=\Sigma\Pi=\Sigma$. Moreover, $T_\varepsilon$ vanishes on
$\ker(Q)$ and is positive definite on $\range(Q)$, with
\[
    T_\varepsilon|_{\range(Q)}
    \succeq\gamma\varepsilon I_{\range(Q)}.
\]
The matrix $H$ is invertible, and
\begin{equation}\label{projected_inversion_components}
    (I_n-\Pi)H^{-1}M=(I_n-\Pi)M,
    \qquad
    \Pi H^{-1}M=T_\varepsilon^\dagger V.
\end{equation}
Consequently,
\[
    H^{-1}M=(I_n-\Pi)M+T_\varepsilon^\dagger V.
\]
If $\range(V)\subseteq\range(\Sigma)$, then
\begin{equation}\label{projected_inversion_bound}
    \|T_\varepsilon^\dagger V\|\leq\|\Sigma^\dagger V\|,
    \qquad
    \|H^{-1}M\|
    \leq\|(I_n-\Pi)M\|+\|\Sigma^\dagger V\|.
\end{equation}
In particular, if $V=\Sigma W$, then $\|\Sigma^\dagger V\|\leq\|W\|$,
so the two right-hand bounds in \eqref{projected_inversion_bound} may be
replaced by $\|W\|$ and $\|(I_n-\Pi)M\|+\|W\|$, respectively.
\end{lemma}
We now apply \lemref{lem_projected_inversion} to the regularized solution map. For
$D\in\mD_\varepsilon(\xsepst,\theta)$,
\propref{prop_composite_regularized} identifies
$-\Heps(D)^{-1}\Meps(D)$ as an element of the conservative Jacobian
\eqref{composite_regularized_jacobian}. Taking $H=\Heps(D)$ and
$M=\Meps(D)$ in \lemref{lem_projected_inversion} gives
\begin{equation}\label{composite_projected_components}
\begin{aligned}
    (I_n-\Pi_Q)\Heps(D)^{-1}\Meps(D)
    &=(I_n-\Pi_Q)\Meps(D),\\
    \Pi_Q\Heps(D)^{-1}\Meps(D)
    &=T_\varepsilon(D)^\dagger V(D).
\end{aligned}
\end{equation}
These identities determine the $\ker(Q)$ component without an inversion and
reduce the $\range(Q)$ component to the symmetric positive definite system with
$T_\varepsilon(D)$. Consequently,
\begin{equation}
    \label{comp_subspace_decomposition}
    -\Heps(D)^{-1}\Meps(D)
    =-(I_n-\Pi_Q)\Meps(D)-T_\varepsilon(D)^\dagger V(D).
\end{equation}
For the limiting result, the matrices in
\eqref{comp_subspace_decomposition} must remain uniformly bounded as
$\varepsilon\downarrow0$. The Tikhonov regularization does not damp the $\ker(Q)$ component (the first term), which must therefore be bounded directly. On
$\range(Q)$, \eqref{projected_inversion_bound} gives the required control
when $V(D)$ belongs to $\range(\Sigma(D))$ and
$\Sigma(D)^\dagger V(D)$ remains bounded. These requirements motivate the
following assumptions.

The decomposition is pointwise in $Q$. In particular, it does not assert
that $\Pi_Q$ varies continuously when $\mathrm{rank}(Q)$ changes; we return to
this issue when studying cluster points.
\begin{enumerate}[label={\bf(R')}]
\item \label{assum:range_c} {[Composite uniform range assumption]} With $O$
and $\bar\eps$ as above, suppose that for every compact set $K\subset O$
there exists $C_K>0$ such that, for every $\theta\in K$, every
$\varepsilon\in]0,\bar\eps]$, and every
$D\in\mD_\varepsilon(\xsepst,\theta)$,
\[
    \range(V(D))\subseteq\range(\Sigma(D)),
    \qquad
    \|\Sigma(D)^\dagger V(D)\|\leq C_K.
\]
\end{enumerate}
The range inclusion gives
$\Sigma(D)\Sigma(D)^\dagger V(D)=V(D)$. Conversely, if
$V(D)=\Sigma(D)W$, then
$\Sigma(D)^\dagger V(D)=\Sigma(D)^\dagger\Sigma(D)W$ is obtained by
projecting each column of $W$ orthogonally onto $\range(\Sigma(D))$, and
hence $\|\Sigma(D)^\dagger V(D)\|\leq\|W\|$. Thus \ref{assum:range_c} is
equivalent to the existence of uniformly bounded solutions of
$\Sigma(D)W=V(D)$, and is the exact analog of \ref{assum:range} for the
reduced system on $\range(Q)$.
\begin{enumerate}[label={\bf(M')}]
\item \label{assum:m_term_c} {[Composite uniform kernel assumption]} With $O$
and $\bar\eps$ as above, suppose that for every compact set $K\subset O$
there exists $B_K\geq0$ such that, for every $\theta\in K$, every
$\varepsilon\in]0,\bar\eps]$, and every
$D=(Q,P,H^f,M^f)\in\mD_\varepsilon(\xsepst,\theta)$,
\[
    \|(I_n-\Pi_Q)\Meps(D)\|
    =\|(I_n-\Pi_Q)P\|\leq B_K.
\]
\end{enumerate}
Since $(I_n-\Pi_Q)Q=0$ and $\Meps(D)=\gamma QM^f-P$,
\[
\begin{aligned}
    (I_n-\Pi_Q)\Meps(D)
    &=-(I_n-\Pi_Q)P,\\
    V(D)
    &=\gamma\Pi_QM^f-Q^\dagger P
      +\gamma\Pi_QH^f(I_n-\Pi_Q)P.
\end{aligned}
\]
Thus \ref{assum:m_term_c} controls only the $\theta$-block $P$ of
$J_{G^\gamma}$.
\begin{remark}
    [Sufficient conditions for \ref{assum:m_term_c}]\label{rem_Mc_sufficient}
    Assumption \ref{assum:m_term_c} holds in either of the following
    situations. (a) If $g$ does not depend on $\theta$, then $P=0$, so one
    may take $B_K=0$; moreover, $V(D)=\gamma\Pi_QM^f$.
    (b) Suppose that $\minsol$ is independently known to be locally bounded
    on $O$. The Tikhonov estimate
    $\|\xsepst\|\leq\|\minsol(\theta)\|$ (\thmref{thm_Browder}) and
    \ref{assum:reg_c} then imply that, for $\theta\in K$ and
    $0<\varepsilon\leq\bar\eps$, the points
    $(z_\varepsilon(\xsepst,\theta),\theta)$ remain in a compact subset of
    $\mbR^n\times O$. The Clarke Jacobian of the locally Lipschitz map
    $G^\gamma$ is bounded on this set, and hence so is
    $\|(I_n-\Pi_Q)P\|\leq\|P\|$.
\end{remark}
With those assumptions in mind, we can state the path-differentiability theorem. Its proof is given in \secref{ss_proof_composite_pathdiff}.
\begin{theorem}
    [Path-differentiability of the composite case]\label{thm_composite_pathdiff} Assume that \ref{assum:f_c}, \ref{assum:reg_c}, \ref{assum:range_c} and \ref{assum:m_term_c} hold. Then $J_{\minsol}$ defined in \eqref{eq:J_min} is a definable conservative Jacobian of the minimal-norm solution $\minsol$ of \eqref{pb_composite}. In particular, $\minsol$ is path-differentiable on $O$.
\end{theorem}
The uniform bound underlying the theorem ensures that derivative sequences have cluster points but does not identify their form. Two losses of continuity must be distinguished. The projector $Q\mapsto QQ^\dagger$ may jump when the rank of $Q$ changes, and, even after this projector stabilizes, eigenvalues of the reduced matrix $\Sigma$ may vanish relative to $\varepsilon$.

Under the assumptions of \thmref{thm_composite_pathdiff}, fix $\theta\in O$, let $\varepsilon_k\downarrow0$ with $0<\varepsilon_k\leq\bar\eps$, and select
\[
    D_k=(Q_k,P_k,H_k^f,M_k^f)
    \in\mD_{\varepsilon_k}(\xs_{\varepsilon_k}(\theta),\theta),
    \qquad
    D_k\to D_0=(Q_0,P_0,H_0^f,M_0^f).
\]
For the results below, write
\[
\begin{gathered}
    \Pi_k:=\Pi_{Q_k},\qquad
    \mathcal M_k:=\gamma Q_kM_k^f-P_k,\qquad
    \Sigma_k:=\Sigma(D_k),\\
    V_k:=V(D_k),\qquad
    T_k:=T_{\varepsilon_k}(D_k),\\
    J_k:=-H_{\varepsilon_k}(D_k)^{-1}\mathcal M_k
       =-(I_n-\Pi_k)\mathcal M_k-T_k^\dagger V_k,
\end{gathered}
\]
and define $\Pi_0$, $\mathcal M_0$, $\Sigma_0$, and $V_0$ analogously
from $D_0$.

\begin{proposition}[Rank-changing cluster points]
\label{prop_composite_cluster}
The limit satisfies $D_0\in\mD_0(\minsol(\theta),\theta)$, the sequence
$(J_k)_k$ is bounded, and every cluster point belongs to
$J_{\minsol}(\theta)$. If $J$ is such a cluster point, then, after passing to
a subsequence realizing it and extracting further, there exist an orthogonal
projector $\bar\Pi$ and a matrix $\bar Y$ such that
\[
    \Pi_k\to\bar\Pi,\qquad
    T_k^\dagger V_k\to\bar Y,\qquad
    J=-(I_n-\bar\Pi)\mathcal M_0-\bar Y.
\]
Moreover,
\[
    \bar\Pi Q_0=Q_0\bar\Pi=Q_0,\qquad
    \range(Q_0)\subseteq\range(\bar\Pi),\qquad
    \bar\Pi\bar Y=\bar Y.
\]
\end{proposition}
The proof is given in \secref{ss_proof_composite_cluster}.

\begin{corollary}[Stable-$Q$ residual formula]
\label{coro_composite_stable_Q}
Suppose that $\operatorname{rank}(Q_k)=\operatorname{rank}(Q_0)$ for all
sufficiently large $k$, and let $J$ be a cluster point of $(J_k)_k$. Set
$W_k:=\Sigma_k^\dagger V_k$ and $R_k:=T_k^\dagger\Sigma_k$. Then
\[
    Q_k^\dagger\to Q_0^\dagger,\qquad
    \Pi_k\to\Pi_0,\qquad
    \Sigma_k\to\Sigma_0,\qquad
    V_k\to V_0.
\]
Along every subsequence realizing $J$, a further subsequence satisfies
$W_k\to W$ and $R_k\to R$. Set
$P_\Sigma:=\Sigma_0^\dagger\Sigma_0$ and $S:=R-P_\Sigma$. Then
\[
    \Sigma_0W=V_0,\qquad
    0\preceq S\preceq\Pi_0-P_\Sigma,\qquad
    \Sigma_0S=S\Sigma_0=0,
\]
and
\begin{equation}
\label{composite_cluster_residual}
    J=-(I_n-\Pi_0)\mathcal M_0
      -\Sigma_0^\dagger V_0-SW.
\end{equation}
\end{corollary}
The proof is given in \secref{ss_proof_composite_stable_Q}.

For a sequence satisfying the stable-rank condition in
\corref{coro_composite_stable_Q}, consider either of the following spectral
alternatives:
\[
    (\mathrm{G}_\Sigma)\qquad
    \operatorname{spec}(\Sigma_k)
    \subset\{0\}\cup[\lambda^*,+\infty[
    \quad\text{eventually, for some }\lambda^*>0,
\]
or, writing the eigenvalues in nonincreasing order,
\[
    (\mathrm{V}_\Sigma)\qquad
    \lambda_i(\Sigma_k)\to0
    \quad\Longrightarrow\quad
    \lambda_i(\Sigma_k)=o(\varepsilon_k),
    \qquad i=1,\ldots,n.
\]

\begin{corollary}[Pseudoinverse cluster formula]
\label{coro_composite_pseudoinverse}
Under the stable-rank hypothesis of \corref{coro_composite_stable_Q}, if
either $(\mathrm{G}_\Sigma)$ or $(\mathrm{V}_\Sigma)$ holds, then the
representation in \corref{coro_composite_stable_Q} may be taken with $S=0$.
Equivalently, every cluster point $J$ of $(J_k)_k$ satisfies
\begin{equation}
\label{composite_cluster_pseudoinverse}
    J=-(I_n-\Pi_0)\mathcal M_0-\Sigma_0^\dagger V_0.
\end{equation}
\end{corollary}
The proof is given in \secref{ss_proof_composite_pseudoinverse}.

\emph{Stable reduced rank.}
If, in addition to stable rank of $Q_k$,
$\operatorname{rank}(\Sigma_k)=\operatorname{rank}(\Sigma_0)$ for all
sufficiently large $k$, then $(\mathrm{G}_\Sigma)$ holds. Indeed,
$\Sigma_k\to\Sigma_0$, so their positive eigenvalues stay uniformly away
from zero.

\begin{remark}
    [Failure of \ref{assum:range_c} in Figure \ref{fig:linf}]\label{rem_failure_Rcomp_Linf} Consider the problem $\max_{\|x\|_{\infty}\leq 1}\langle \theta,x\rangle$, at $\theta_0 = (0, -1)$. Separability yields, after some algebra, that 
    \[(\xsepst)_i = \max(\min(+1, \theta_i / \varepsilon), -1) = \Pi_{[-1, +1]}(\theta_i/\varepsilon), \quad [J_{\xseps}(\theta)]_i = \left\{ \begin{array}{cc} 0 & |\theta_i| > \varepsilon, \\ 1/\varepsilon & |\theta_i| < \varepsilon, \\ {[}0, 1/\varepsilon] & |\theta_i| = \varepsilon.\end{array} \right.\]
    Since $\theta_1 = 0$ but $\theta_2 \neq 0$, $[J_{\xseps}(\theta)]_1 = 1/\varepsilon$ which diverges as expected since the minimal-norm solution is not continuous. Formally, one can show that for some blocks $D$ and some blocks $Q$, $\range(V(D)) \subsetneq \range(\Sigma(D))$.
\end{remark}

\subsection{Return to the LASSO}\label{ss_return_LASSO}
\begin{corollary}[LASSO path-differentiability]
\label{coro_LASSO_pathdiff}
The minimal-norm solution $\minsol$ of \eqref{LASSO} is
path-differentiable as a function of $\theta$.
\end{corollary}
\begin{proof}
For the LASSO, $P=0$, $H^f=A\tp A$, $M^f=-A\tp$, and
$Q\in\diag([0,1]^n)$. Hence \ref{assum:m_term_c} holds with $B_K=0$, and
\[
    V(D)=-\gamma\Pi_QA\tp.
\]
If $Q=0$, then $\Sigma(D)=V(D)=0$, so \ref{assum:range_c} is immediate.
Otherwise, let $S_+=S_+(Q)$ and $Q_+$ be as in \secref{ss_LASSO}, and set
\[
    \Delta_Q:=\frac{Q_+^{-1}-I_{|S_+|}}{\gamma}\succeq0.
\]
The nonzero blocks are
\[
    \Sigma(D)_{S_+,S_+}
    =\gamma\big(\Delta_Q+A_{:,S_+}\tp A_{:,S_+}\big),
    \qquad
    V(D)_{S_+,:}=-\gamma A_{:,S_+}\tp,
\]
all other rows and columns of $\Sigma(D)$ vanish, as do all other rows of
$V(D)$. Consequently,
\eqref{comp_subspace_decomposition} is supported on $S_+$ and reduces there
to
\[
    \big(\Delta_Q+\varepsilon I_{|S_+|}
       +A_{:,S_+}\tp A_{:,S_+}\big)^{-1}A_{:,S_+}\tp,
\]
which is the active-space formula obtained in \secref{ss_LASSO}. Moreover,
\[
    \range(A_{:,S_+}\tp)
    =\range(A_{:,S_+}\tp A_{:,S_+})
    \subseteq
    \range\big(\Delta_Q+A_{:,S_+}\tp A_{:,S_+}\big).
\]
Thus $\range(V(D))\subseteq\range(\Sigma(D))$, and
\[
    \big[\Sigma(D)^\dagger V(D)\big]_{S_+,:}
    =
    -\big(\Delta_Q+A_{:,S_+}\tp A_{:,S_+}\big)^\dagger A_{:,S_+}\tp,
    \qquad
    \big[\Sigma(D)^\dagger V(D)\big]_{S_+^c,:}=0.
\]
For $\delta>0$, \lemref{lem_scaled_gram_bound} applied with
$D=\Delta_Q+\delta I_{|S_+|}$ bounds the nonzero block by
$\chi(A_{:,S_+})\leq\chi(A)$. Letting $\delta\downarrow0$ gives the displayed
pseudoinverse because every column of $A_{:,S_+}\tp$ belongs to the range of
the limiting positive semidefinite matrix. Hence \ref{assum:range_c} holds with
$C_K=\chi(A)$, and the conclusion follows from
\thmref{thm_composite_pathdiff}.
\end{proof}

We now identify explicit elements of $J_{\minsol}(\theta)$ obtained from
binary soft-threshold selections along the regularized trajectory. Fix
$\theta\in\mbR^m$, set $\tau:=\gamma e^\lambda$, and define
\[
\begin{aligned}
    \mQ_\varepsilon(\theta)
    &:=
    \partial\prox_{\gamma e^\lambda\|\cdot\|_1}
    \big(z_\varepsilon(\xsepst,\theta)\big),
    &&0<\varepsilon\leq\bar\eps,\\
    \mQ_0(\theta)
    &:=
    \partial\prox_{\gamma e^\lambda\|\cdot\|_1}
    \big(z_0(\minsol(\theta),\theta)\big).
\end{aligned}
\]
The fixed-parameter reachable set is
\begin{equation}\label{LASSO_reachable_set}
    \mR^{\mathrm{LASSO}}(\theta)
    :=
    \left\{
    Q\in\mQ_0(\theta):
    \begin{array}{l}
        \exists\ \varepsilon_k\downarrow0,
        Q_k\in\mQ_{\varepsilon_k}(\theta),\\
        Q_k\to Q
    \end{array}
    \right\}.
\end{equation}
Since \eqref{eq:J_min} also permits $\theta_k\to\theta$, this fixed-parameter
set need not describe every element of $J_{\minsol}(\theta)$. For
$i\in[1:n]$, let
\[
    \rho_i(\varepsilon;\theta)
    :=
    \left|
    \big[z_\varepsilon(\xsepst,\theta)\big]_i
    \right|-\tau.
\]
Each $\rho_i(\cdot;\theta)$ is definable, so, for all sufficiently small
$\varepsilon>0$, it is either positive, negative, or identically zero.
\begin{corollary}[Reachable binary LASSO derivatives]
\label{coro_LASSO_reachable}
Fix $\theta\in\mbR^m$ and set
\[
    \mR_{\mathrm{bin}}^{\mathrm{LASSO}}(\theta)
    :=
    \mR^{\mathrm{LASSO}}(\theta)\cap\diag(\{0,1\}^n).
\]
A binary diagonal matrix $Q$ belongs to
$\mR_{\mathrm{bin}}^{\mathrm{LASSO}}(\theta)$ if and only if, for every
$i\in[1:n]$,
\[
Q_{ii}\in
\begin{cases}
\{1\},&
    \rho_i(\varepsilon;\theta)>0
    \text{ for all sufficiently small }\varepsilon,\\
\{0\},&
    \rho_i(\varepsilon;\theta)<0
    \text{ for all sufficiently small }\varepsilon,\\
\{0,1\},&
    \rho_i(\varepsilon;\theta)=0
    \text{ for all sufficiently small }\varepsilon.
\end{cases}
\]
In particular,
$\mR_{\mathrm{bin}}^{\mathrm{LASSO}}(\theta)\neq\varnothing$, and
\begin{equation}\label{LASSO_reachable_derivatives}
    \left\{
    (AQ)^\dagger:
    Q\in\mR_{\mathrm{bin}}^{\mathrm{LASSO}}(\theta)
    \right\}
    \subseteq J_{\minsol}(\theta).
\end{equation}
\end{corollary}
The proof is given in \secref{ss_proof_LASSO_reachable}.
\begin{remark}
[Differentiation with respect to $\theta$ and $\lambda$]
\label{rem_LASSO_parameters}
All LASSO derivatives above are taken with respect to the observation $\theta$,
with $\lambda$ fixed. In this parameterization, the proximal map is independent
of $\theta$, hence $P=0$, and \propref{prop_LASSO_unif_bound} gives a uniform
bound without any rank assumption on $A$. If $\lambda$ is differentiated
instead, it enters the soft-threshold level: $M^f=0$, while the proximal
parameter block $P$ need not vanish. The argument of
\propref{prop_LASSO_unif_bound} therefore does not apply directly.
\citet[Proposition~5 and Section~C.3]
{bolte_le_pauwels_silveti-falls_2021} obtain path-differentiability of the
LASSO solution map without Tikhonov regularization by assuming that
$A_{:,E}\tp A_{:,E}$ is invertible, where $E$ is the equicorrelation set. A very similar assumption arises in \citet{bertrand_klopfenstein_massias_blondel_vaiter_gramfort_salmon_2022}.
Their condition is sufficient for that $\lambda$-sensitivity result, whereas our $\theta$-sensitivity result for the minimal-norm selection requires no such condition.
\end{remark}

\subsection{Convergence Rates with Fixed Tikhonov Regularization}\label{ss_convergence_rates}
This section discusses the accuracy algorithms can reach, in the context of Tikhonov regula\-ri\-zation, following \citet{grazzi_pontil_salzo_2024, bolte_pauwels_vaiter_2022}. We focus on problem \eqref{pb_composite} under the same assumptions. Recall from \eqref{composite_fixed_point} the fixed-point formulation:
\begin{equation}
    \label{fixed_point_for_rates} x = \prox_{\gamma g(\cdot,\theta)}(z_{\varepsilon}(x, \theta)) = \varphi^{\gamma}_{\varepsilon}(x, \theta), \qquad \zext = x - \gamma\varepsilon x - \gamma\nabla_x f(x, \theta).
\end{equation}
The mapping $\varphi^{\gamma}_{\varepsilon}$ is forward-backward due to the proximal (implicit/backward) step on $\gamma g$ composed with the gradient-descent (explicit/forward) step on $\gamma f$. A key property to study the sequence of Jacobians is that the derivatives of $\varphi_{\varepsilon}^{\gamma}$ w.r.t. $x$ are contracting.
Let $\overline{\varepsilon} > 0$ be such that $\varepsilon \in ]0, \overline{\varepsilon}]$, for $\gamma \in ]0, 2/(2\overline{\varepsilon} + L)[$, $\varphi^{\gamma}_{\varepsilon}$ has contracting $x$-derivatives. Indeed, using \eqref{composite_sets_matrices},
\[0 \preceq H^f \preceq L I_n, \qquad (1 - \gamma\varepsilon - \gamma L)I_n \preceq I_n - \gamma\varepsilon I_n - \gamma H^f \preceq (1 - \gamma\varepsilon) I_n,\]
so the eigenvalues of $I_n - \gamma\varepsilon I_n - \gamma H^f$ belong to $]-1, +1[$. One concludes using that $\|\Qeps\| \leq 1$. Once a suitable $\gamma$ is set, we can use fixed-point algorithms such as in \citet[Section~4]{grazzi_pontil_salzo_2024} \citep[see also][Section~4.1]{bolte_pauwels_vaiter_2022}, using \eqref{fixed_point_for_rates} and leading to the iterations
\begin{equation}
    \label{rates_algo_iterates} x^0_{\varepsilon}(\theta) \in \mbR^n, \qquad x^{k+1}_{\varepsilon}(\theta) = \varphi^{\gamma}_{\varepsilon}(x_{\varepsilon}^k(\theta), \theta).
\end{equation}
Now, consider the derivatives. For fixed $\varepsilon > 0$, due to the contractivity of the $x$-derivatives, the fixed-point approach yields a rather natural conservative Jacobian \citep{bolte_pauwels_vaiter_2022}:
\[J^{\mathrm{fix}}_{\xseps}(\theta) := \mathrm{fix}\left(J_{\varphi_{\varepsilon}^{\gamma}}(\xsepst, \theta)\right)\]
where $\mathrm{fix}$ denotes the (set of) fixed-points of a set of matrices, to be understood as:
\[\begin{split}
    \mathrm{fix}\left(J_{\varphi^{\gamma}_{\varepsilon}}(\xsepst, \theta)\right) = &\ J_{\varphi^{\gamma}_{\varepsilon}}(\xsepst, \theta)\left[\mathrm{fix}\left(J_{\varphi^{\gamma}_{\varepsilon}}(\xsepst, \theta)\right)\right], \\
    & \ J_{\varphi^{\gamma}_{\varepsilon}}(\xsepst, \theta)(\mZ) = \{AZ + B : Z \in \mZ, (A\ B) \in J_{\varphi^{\gamma}_{\varepsilon}}(\xsepst, \theta)\},
\end{split}\]
where $A$ is the $x$-block and $B$ the $\theta$-block (thus $\|A\| < 1$). Notably, $J^{\mathrm{imp}}_{\xseps}(\theta) \subset J^{\mathrm{fix}}_{\xseps}(\theta)$ \citep[][Remark~2]{bolte_pauwels_vaiter_2022} when both are computed with the same $J_{\varphi^{\gamma}_{\varepsilon}}$. Elements of $J^{\mathrm{fix}}_{\xseps}(\theta)$ can be computed by the following iterations \citep[Section~4]{grazzi_pontil_salzo_2024}, akin to \eqref{rates_algo_iterates}:
\[\begin{array}{llll}
    (\mathrm{ITER}) & J^0_{\xseps}(\theta) = \{0\}, & J^{k+1}_{\xseps}(\theta) = H^k J^{k}_{\xseps}(\theta) + M^k, & (H^k\ M^k) \in \partial \varphi^{\gamma}_{\varepsilon}(x^k_{\varepsilon}(\theta), \theta).
\end{array}\]
An improved version fixes $N \in \mbN$, computes $x^N_{\varepsilon}(\theta)$ and, for $k \leq N$, uses:
\[\begin{array}{llll}
    (\mathrm{FIXN}) & J^{N, 0}_{\xseps}(\theta) = \{0\}, & J^{N, k+1}_{\xseps}(\theta) = H^N J^{N, k}_{\xseps}(\theta) + M^N, & (H^N\ M^N) \in \partial \varphi^{\gamma}_{\varepsilon}(x^N_{\varepsilon}(\theta), \theta).
\end{array}\]
The version $(\mathrm{FIXN})$ yields better results since $x^N_{\varepsilon}(\theta)$ is typically closer to $\xsepst$ than $x^k_{\varepsilon}(\theta)$. Indeed, \thmref{thm_GPS} gives some upper bound between the excess (denoted $e$, $e(\mC, \mD) = \sup_{C \in \mC} \inf_{D \in \mD}\|C - D\|$) and an element computed by the iterations $(\mathrm{FIXN})$.
\begin{theorem}
    [Rates of convergence, {\cite[Theorem 4.1 simplified]{grazzi_pontil_salzo_2024}}]\label{thm_GPS} Let $\varepsilon > 0$ and $\gamma > 0$ small enough. One has, denoting $k$ the iteration number and $\Delta_0 := \|x^0_{\varepsilon}(\theta) - \xsepst\|$,
    \begin{equation}\label{main_bounds_GPS}
        \begin{split}
        \|x_{\varepsilon}^k(\theta) - \xsepst\| & \leq C_1(1 - \gamma\varepsilon)^k = O((1 - \gamma\varepsilon)^k) \\
        e(J^{N, k}_{\xseps}(\theta), J_{\xseps}^{\mathrm{fix}}(\theta)) & \leq \frac{B_{\theta}}{\gamma\varepsilon}(1-\gamma\varepsilon)^k + \frac{B_{\theta} + \gamma\varepsilon}{\gamma\varepsilon} \left(L' + \frac{M_{\theta}}{R_{\theta}}\right) \frac{1 - (1 - \gamma\varepsilon)^k}{\gamma\varepsilon} (1-\gamma\varepsilon)^N \Delta_0
    \end{split}
    \end{equation}
    where $\varphi_{\varepsilon}^{\gamma}$ is $L'$-Lipschitz smooth, $B_{\theta}$, $M_{\theta}$ and $R_{\theta}$ are other constants of $\varphi^{\gamma}_{\varepsilon}$.
\end{theorem}
The following remark exposes, in the H\"olderian setting (see \defref{def_Holder_function} and \thmref{thm_holder_bound}), via \citet{maulen-soto_fadili_attouch_2025}, a simple but important fact of regularization techniques.
\begin{remark}
    [Rate for $x$ in the H\"olderian setting]\label{rem_iterate_holder} Suppose that $f + g$ verifies a H\"olderian error bound $f(x, \theta) + g(x, \theta) - f(\minsol(\theta), \theta) - g(\minsol(\theta), \theta) \geq \kappa \dist(x, \mS(\theta))^p$ for some $p \geq 1$ and $\kappa > 0$ (see \thmref{thm_holder_bound}). Then there exists $C_0$ and $\varepsilon^*$ such that, for any $\varepsilon \leq \varepsilon^*$, one has
    \[\|x^k_{\varepsilon}(\theta) - \minsol(\theta)\| \leq \underbrace{ C_1 (1 - \gamma\varepsilon)^k}_{\textnormal{approximation error}} + \underbrace{C_0\varepsilon^{\frac{1}{2p}}}_{\textnormal{regularization error}}.\]
    Indeed, $\|x^k_{\varepsilon}(\theta) - \minsol(\theta)\| \leq \|x^k_{\varepsilon}(\theta) - \xsepst\| + \|\xsepst - \minsol(\theta)\| \leq C_1(1-\gamma\varepsilon)^k + C_0\varepsilon^{\frac{1}{2p}}$ using the triangle inequality, \eqref{main_bounds_GPS} and then Theorem~\ref{thm_holder_bound}.
\end{remark}
When $\etz$, so does $C_0\varepsilon^{\frac{1}{2p}}$ but $C_1(1-\gamma\varepsilon)^k$ increases: the fidelity to the original problem is counterbalanced by slower convergence speed; it is easier to solve a more regularized problem; see also \corref{coro_estimates_delta}. For the conservative mappings, we could similarly use
\begin{equation*}
\label{excess_triangular_majoration}e(J^{k}_{\xseps}(\theta), J_{\minsol}(\theta)) \leq e(J^{k}_{\xseps}(\theta), J_{\xseps}^{\mathrm{fix}}(\theta)) + e(J_{\xseps}^{\mathrm{fix}}(\theta), J_{\minsol}(\theta)).
\end{equation*}
The first term, by contractivity, can be estimated in terms of $\varepsilon$ and $k$ by \thmref{thm_GPS}. However, $J^{\mathrm{fix}}_{\xseps}$ is not given by an IFT-type formula so different assumptions would be necessary to deal with the second term. Using $e(\mA, \mC) \leq e(\mA, \mB) + e(\mB, \mC)$ \citep[see][Lemma~B.1(i)]{grazzi_pontil_salzo_2024}, we have
\begin{equation}\label{excess_triangular_majoration_2}
    e(J^{k}_{\xseps}(\theta), J_{\minsol}(\theta)) \leq \underbrace{e(J^{k}_{\xseps}(\theta), J_{\xseps}^{\mathrm{fix}}(\theta))}_{\textnormal{approximation error}} + \underbrace{e(J_{\xseps}^{\mathrm{fix}}(\theta), J^{\mathrm{imp}}_{\xseps}(\theta))}_{\textnormal{nonsmoothness error}} + \underbrace{e(J^{\mathrm{imp}}_{\xseps}(\theta), J_{\minsol}(\theta))}_{\textnormal{regularization error}}.
\end{equation}
Here, the second term represents some ``derivative gap'' between two specific conservative mappings: since in general $J^{\mathrm{imp}} \subsetneq J^{\mathrm{fix}}$, this term is nonzero. Finally, the third term can be $o(1)$ by \thmref{thm_Schechtman}, but without explicit rate. Those are available for simple problems such as Least-Squares (see \propref{prop_rates_LS}), and in some settings where the solution is unique (see \citet{pauwels2023derivatives} in optimal transport). The estimations of the nonsmoothness and regularization errors are interesting open questions.
\begin{remark}
    In \citet{bolte_pauwels_vaiter_2022} and \citet{grazzi_pontil_salzo_2024}, the explicit upper bounds are for the strongly convex setting. In \citet[theorem 1]{blondel_berthet_cuturi_frostig_hoyer_llinares-lopez_pedregosa_vert_2021}, the strong bounds are obtained in the smooth case, with global uniform invertibility and Lipschitz bound.
\end{remark}

\section{Numerical Experiments}\label{s_experiments}
Full descriptions of the experimental details and data used are given in Appendix~\ref{s_additional_experiments}.
\subsection{Data Poisoning: Ill-posed Least-Squares}\label{ss_riboflavin_poisoning}
Data \textit{poisoning} \citep{xiao_biggio_brown_fumera_eckert_roli_2015, mei_zhu_2015} is a bilevel optimization problem modeling an adversary whose goal is to modify training data to degrade the performance of a learned model. The lower-level problem fits the model to the poisoned data, and the upper-level problem selects the data modification based on some upper-level objective. We study an instance with a Least-Squares lower-level problem using the Riboflavin regression data \citep{buhlman_kalisch_meier_2014}, whose feature dimension $4088$ is greater than its $71$ observations, so that the solution set is a nontrivial affine subspace. This setting allows us to precisely characterize minimal and nonminimal solutions in closed-form.

Denote the lower-level design matrix $A\in\mbR^{n_{\rm tr}\times p}$ and clean response $\theta_0\in\mbR^{n_{\rm tr}}$, and the upper-level design matrix $A_{\rm val}\in\mbR^{n_{\rm val}\times p}$ and response $b_{\rm val}\in\mbR^{n_{\rm val}}$. The adversary replaces $\theta_0$ by $\theta$, while $\xs\in\mbR^p$ represents the fitted model. Poisoning under a Euclidean budget $\tau > 0$ can then be written
\begin{equation*}
    \underbrace{\max_{\|\theta-\theta_0\|_2\leq \tau}\
    \tfrac{1}{2n_{\rm val}}\|A_{\rm val}x^\star(\theta)-b_{\rm val}\|^2}_{\text{upper-level problem}}
    \quad\text{subject to}\quad
    x^\star(\theta)\in
    \underbrace{\argmin_{x\in\mbR^p}\
    \tfrac{1}{2n_{\rm tr}}\|Ax-\theta\|^2}_{\text{lower-level problem}}.
\end{equation*}
The minimal-norm solution has closed-form $\minsol(\theta) = A^\dagger\theta$. Nonminimal solutions can be systematically constructed given some $q_k\in \ker(A)$ with $\|q_k\|_2=1$ by forming
\begin{equation}\label{eq_nonminimal_selection}
    x^\star_{c,k}(\theta)=\minsol(\theta)+c\,\|\minsol(\theta_0)\|_2\,q_k.
\end{equation}
We can then evaluate and optimize a single-level objective in $\theta$ for each selection, which we denote $\Phi_{c,k}(\theta):=\tfrac{1}{2n_{\rm val}}\|A_{\rm val}x^\star_{c,k}(\theta)-b_{\rm val}\|^2$, with $x^\star_{0,0}(\theta):=\minsol(\theta)$ by convention. The difference between hypergradients for two solutions can be computed
\begin{equation}\label{eq_constant_gradient_difference}
    \nabla_\theta\Phi_{c,k}(\theta)-\nabla_\theta\Phi_{\min}(\theta)
    =\frac{c\,\|\minsol(\theta_0)\|_2}{n_{\rm val}}\,(A^\dagger)^\top A_{\rm val}^\top A_{\rm val}q_k. 
\end{equation}
For each seed (corresponding to one split of the data), we draw $20$ such perturbations $q_k\in\ker(A)$, giving a controlled family of exact, nonminimal solutions indexed by the magnitude of the offset $c$. 

For each seed and each solution selection, we run $150$ iterations of normalized projected-momentum ascent on the upper-level objective value with the exact gradient. We report the relative poison damage $d_{c,k}:=\Phi_{c,k}(\theta_{150})/\Phi_{c,k}(\theta_0)-1$ for each selection $k$ and for different magnitudes $c$. Figure~\ref{fig:riboflavin-ls} reports the median of this quantity across $20$ seeds, showing the poisoning effect decreasing monotonically as the selected solution drifts farther from minimal-norm one. Table~\ref{tab:riboflavin-ls-win-counts} gives the number of seeds for which the minimal-norm solutions yielded a better upper-level objective value. 
\begin{table}[htbp]
    \centering
    \caption{Riboflavin Least-Squares: number of paired reporting splits in which the minimal-norm has greater relative poison damage than the within-split median over the $20$ nonminimal seeds. ``Better'' means a greater upper-level objective value.}
    \label{tab:riboflavin-ls-win-counts}
    \begin{tabular}{lcccccc}
        \toprule
        $c$ & $1$ & $2$ & $3$ & $5$ & $7.5$ & $10$ \\
        \midrule
        Minimal-norm better (out of $20$) & $14$ & $16$ & $19$ & $19$ & $20$ & $20$ \\
        \bottomrule
    \end{tabular}
\end{table}
\begin{figure}[htbp]
      \centering
      \includegraphics[width=0.62\linewidth]{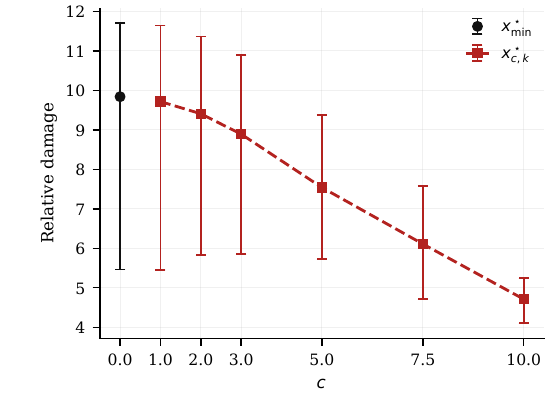}
      \caption{Riboflavin Least-Squares: final relative damage after $150$ iterations, with $c=0$ denoting the minimal norm selection. Markers are medians over the $20$ seeds and error bars are $95.86\%$ intervals for the median. Full details in Appendix~\ref{ss_appendix_riboflavin_poisoning}.}
      \label{fig:riboflavin-ls}
\end{figure}
\subsection{Data Poisoning: Ill-posed Huber Regression}\label{ss_nyc311_huber}
We consider now data poisoning with a different lower-level problem: Huber regression using the NYC Open Data service-request records \citep{nyc_open_data_311}. The detailed preprocessing protocol is given in Appendix~\ref{ss_appendix_nyc311_huber}.

Suppose the lower data have request-type/incident-ZIP incidence matrix $A\in\mbR^{n_{\rm tr}\times p}$ and clean response $\theta_0\in\mbR^{n_{\rm tr}}$, and the validation data have design $A_{\rm val}$ and response $b_{\rm val}$. The adversary replaces $\theta_0$ by $\theta$, while $x\in\mbR^p$ contains request information. Denoting $h_\delta$ as the standard Huber loss with threshold $\delta>0$, the bilevel problem is then
\begin{equation}\label{eq_nyc311_huber_bilevel}
    \underbrace{\max_{\|\theta-\theta_0\|_2\leq\tau}
    \tfrac{1}{2n_{\rm val}}\|A_{\rm val}\xs(\theta) - b_{\rm val}\|^2}_{\text{upper-level problem}}
    \quad\text{subject to}\quad
    \xs(\theta)\in
    \underbrace{\argmin_{x\in\mbR^p}
    \sum_{i=1}^{n_{\rm tr}}h_\delta\bigl((Ax-\theta)_i\bigr)}_{\text{lower-level problem}}.
\end{equation} 

The lower-level problem is ill-posed since, depending on the split, the number of columns of the design matrix $p$ is between $131$ and $149$ and the kernel has dimension $p-\operatorname{rank}(A)$ between $2$ and $9$. Thus, for any lower solution $x$ and any $z\in\ker(A)$, $x+z$ gives exactly the same lower predictions and objective. In contrast to \secref{ss_riboflavin_poisoning}, the minimal-norm solution for this problem is not known analytically. But, by rewriting the Huber loss using convex duality, one can compute the polyhedral set $\mS(\theta)$ and solve a quadratic problem to obtain the minimal-norm solution, see Appendix~\ref{ss_appendix_nyc311_huber}. Using directions $q_k\in \ker(A)$ with $\|q_k\|_2=1$ we can generate perturbed nonminimal solutions
\begin{equation*}
    x^\star_{c,k}(\theta)=\minsol(\theta)+c\,\|\minsol(\theta_0)\|_2\,q_k, \qquad c\in\{0.5,1,2,3\}.
\end{equation*}

We follow \propref{prop_huber_some_limits} to compute an element $V$ in $J_{\minsol}(\theta)$. At each $\theta$, we identify a reachable Huber matrix $Q$ as defined in \eqref{eq:huber_matrices} by solving the Tikhonov problem over a fixed, decreasing grid of $\varepsilon$ values. Each solution $\xs_{\varepsilon}$ determines the diagonal matrix $Q_\varepsilon$ with
\[
(Q_\varepsilon)_{ii}\in
\begin{cases}
\{1\},
& \left|(A\xs_\varepsilon(\theta)-\theta)_i\right|<\delta,\\
[0,1],
& \left|(A\xs_\varepsilon(\theta)-\theta)_i\right|=\delta,\\
\{0\},
& \left|(A\xs_\varepsilon(\theta)-\theta)_i\right|>\delta
\end{cases}
\]
taking different values depending on if the residual is in the quadratic or linear regime. We take $Q=Q_{\varepsilon_{\min}}$, where $\varepsilon_{\min}$ is the smallest value in the grid, and check that $Q_\varepsilon$ does not change over the final three values of $\varepsilon$. This gives numerical evidence that $Q$ represents a reachable Huber matrix selected by the Tikhonov path. We can then compute an element of the conservative Jacobian
\begin{equation}\label{eq_nyc311_implicit_jacobian}
    V=(A\tp QA)^\dagger A\tp Q
\end{equation}
and compute the implicit hypergradient
\[
   V^\top\frac{2}{n_{\rm val}}A_{\rm val}^\top
      (A_{\rm val}\xs_{c,k}-b_{\rm val}).
\]

Just as in \secref{ss_riboflavin_poisoning}, for each selection we do $150$ iterations of projected normalized gradient ascent with momentum on the constrained problem in \eqref{eq_nyc311_huber_bilevel} and plot the median relative poison damage in Figure~\ref{fig:nyc311-huber}. In this problem setting, using the minimal-norm is better on all $20$ splits at every $c$. Appendix~\ref{ss_appendix_nyc311_huber} gives the full details.

\begin{figure}[htbp]
    \centering
    \begin{minipage}[t]{0.49\linewidth}
        \centering
        \includegraphics[width=\linewidth]{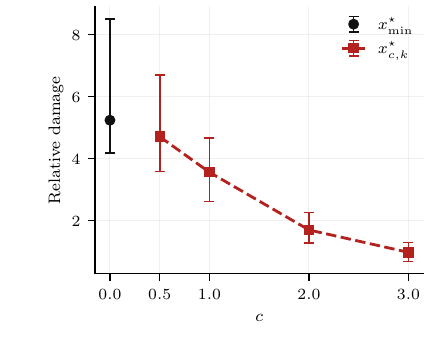}
    \end{minipage}\hfill
    \begin{minipage}[t]{0.49\linewidth}
        \centering
        \includegraphics[width=\linewidth]{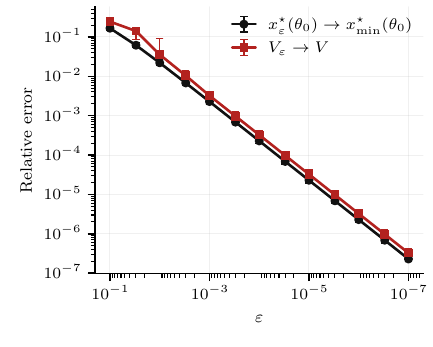}
    \end{minipage}
    \caption{NYC $311$ Huber poisoning. \emph{Left:} each selection's relative poison damage $d_{c,k}$. \emph{Right:} the relative solution error $\|\xs_\varepsilon(\theta_0)-\minsol(\theta_0)\|_2/\|\minsol(\theta_0)\|_2$ and relative derivative error $\|V_\varepsilon-V\|_{\rm F}/\|V\|_{\rm F}$ along the Tikhonov path at the clean response, where $V_\varepsilon=(A\tp Q_\varepsilon A+\varepsilon I_p)^{-1}A\tp Q_\varepsilon$ and $V=(A\tp QA)^\dagger A\tp Q$. In both plots, markers are medians over the $20$ declared splits and whiskers are $95.86\%$ intervals for the median. Full details are in Appendix~\ref{s_data_poison}.}
    \label{fig:nyc311-huber}
\end{figure}

\subsection{Data Hypercleaning: Ill-posed LASSO}\label{ss_digits_lasso_distillation}
Data \textit{hypercleaning} seeks to modify a dataset on which a model can be trained to improve performance on held out data \citep{franceschi_frasconi_salzo_grazzi_pontil_2018}. We adapt this construction to the scikit-learn Digits data set so that the lower-level problem fits Corollary~\ref{coro_LASSO_pathdiff}. For each seed, we collect one example per class into $A\in\mbR^{10\times64}$ and then learn new soft-labels for each example. The upper variable is thus a synthetic response matrix $\theta\in\mbR^{10\times10}$, initialized at $I_{10}$ and constrained rowwise to the probability simplex. Let $e^{\lambda}=10^{-3}$, the lower-level problem is
\begin{equation}\label{eq_digits_lasso_lower}
X^\star_\varepsilon(\theta) \in \argmin_{X\in\mbR^{64\times10}}
    \frac{1}{2\cdot10}\|AX-\theta\|_F^2
    +e^{\lambda}\|X\|_1+\frac{\varepsilon}{2}\|X\|_F^2.
\end{equation}
At $\varepsilon=0$, the positive $\ell_1$ penalty makes the solution set nonempty and compact. To ensure existence of multiple solutions, we average pairs of consecutive columns in $A$. The minimal-norm solution can be found by splitting equally within duplicate pairs.

We solve the lower-level problem starting from $X_0=0$ using $N\in\{10^3, 10^4, 10^5, 2.5 \cdot 10^5\}$ iterations of forward-backward,
\begin{equation}\label{eq_digits_lasso_ista}
X_{k+1}=\operatorname{soft-thresholding}_{\gamma e^{\lambda}}\!\left(
X_k-\gamma\left[A^\top(AX_k-\theta)/10+\varepsilon X_k\right]\right),
\qquad \gamma=0.9/1.1
\end{equation}
with $\varepsilon\in\{10^{-6},3\cdot10^{-6},10^{-5},3\cdot10^{-5},10^{-4},3\cdot10^{-4},10^{-3}\}$. The final iterate $X_N$ defines a support set $S_N$ which we use to solve a system with matrix $\tfrac{1}{10}A_{:,S_N}^\top A_{:,S_N}+\varepsilon I$; at $\varepsilon=0$ its pseudoinverse is used.

We do $100$ iterations of projected normalized gradient ascent with momentum on the upper-level objective. Figure~\ref{fig:digits-lasso-implicit-response} reports the median upper-level objective value over five seeds. Of the values we tested, the best $\varepsilon$ moves from $3\cdot10^{-4}$ at $N=10^{3}$, to $10^{-6}$ at $N=2.5\cdot10^5$. Thus using less iterations on the lower-level problem benefits from additional regularization, whereas a more accurate solve can use less Tikhonov bias, which is consistent with the tradeoff in \remref{rem_iterate_holder}.
\begin{figure}[t]
    \centering
    \includegraphics[width=0.88\linewidth]{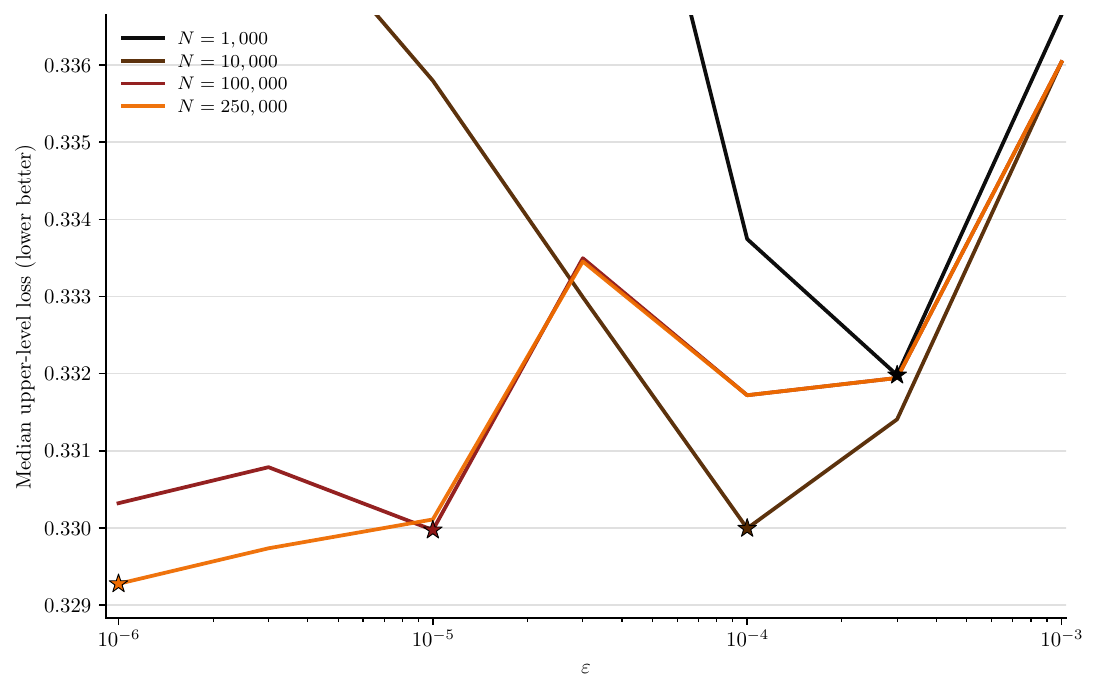}
    \caption{Final upper-level objective value versus the Tikhonov parameter $\varepsilon$. Each curve fixes the number $N$ of forward-backward iterations on the lower-level problem and corresponds to a median over $5$ runs. Stars mark the minimum of each curve over the tested values of $\varepsilon$. More details available in Appendix~\ref{ss_additional_distillation}}
    \label{fig:digits-lasso-implicit-response}
\end{figure}

\section{Concluding remarks}

This paper presents an approach to deal with bilevel problems in which the lower-level is convex but ill-posed and possesses multiple solutions. We proceed by adding Tikhonov regularization, which makes the lower-level problem strongly convex, which allows to use implicit differentiation and output a regularized hypergradient that can be used in a first-order method for the upper level. While the regularized solution converges towards the minimal-norm solution of the lower-level problem, the convergence of derivatives is the main question considered in the paper. An interesting direction for future research would be the identification of conditions yielding convergence rates of conservative Jacobians with respect to the Tikhonov regularization parameter.

\textit{Use of LLMs}\quad All of our experiments involved hyperparameter sweeps and designs whose parallelization was aided through the use of ChatGPT-5.5 and Claude Opus 4.6. ChatGPT 5.6 Sol and Claude Fable 5 were used extensively to review and polish later drafts of the manuscript.

\acks{This work was supported by the French National Research Agency (ANR) under grant ANR-25-CE23-3749 (project SIMPLES). Antonio Silveti-Falls and Baptiste Plaquevent-Jourdain thank Sholom Schechtman for fruitful discussions. }

\appendix
\section{Additional Experimental Results}\label{s_additional_experiments}
The experiments use NumPy, JAX \citep{bradbury_frostig_hawkins_johnson_leary_maclaurin_necula_paszke_vanderplas_wandermanmilne_zhang_2018}, scikit-learn \citep{pedregosa_varoquaux_gramfort_michel_thirion_grisel_blondel_prettenhofer_weiss_dubourg_vanderplas_passos_cournapeau_brucher_perrot_duchesnay_2011}, and SciPy \citep{virtanen_2020} on an Apple M1 CPU in float64.

\subsection{Data Poisoning}\label{s_data_poison}

This appendix describes the protocol for the two experiments of \secref{ss_riboflavin_poisoning} and \secref{ss_nyc311_huber}.

\subsubsection{Riboflavin least squares}\label{ss_appendix_riboflavin_poisoning}

\textit{Exact lower-level solutions}\quad We solve every lower-level Least-Squares problem directly by a truncated SVD. This yields the pseudoinverse, the minimal-norm solution, and the null space used throughout the experiment. 

\textit{Data and splits}\quad The Riboflavin data set has $71$ observations and $4088$ features. We randomly split this data into training and validation sets 20 times corresponding to seeds $s$. For each seed $s$, we permute the observations with \texttt{numpy.random.default\_rng(s)} and assign the first $45$ to the lower-level problem and the next $12$ to upper-level. Feature and response means and population standard deviations are fitted on the $45$ lower-level rows only and applied to both partitions. We denote the resulting lower-level design matrix by $A$ and clean response by $\theta_0$, and the upper-level design matrix and response by $A_{\rm val}$ and $b_{\rm val}$, respectively; the remaining $14$ observations are not used.

\textit{Nonminimal selections}\quad For each split of the data, we draw $g_k\sim N(0,I_p)$ for $k=0,\dots,19$, setting $z_k=\Pi_{\ker(A)}(g_k)$ and $q_k=z_k/\|z_k\|_2$; a draw with $\|z_k\|_2\leq10^{-12}\|g_k\|_2$ is rejected and replaced by the next draw from the same stream. Selections are formed as in \eqref{eq_nonminimal_selection} for $c\in\{1,2,3,5,7.5,10\}$, with the radius computed from the minimal-norm solution corresponding to the clean data and then held fixed along the whole upper-level trajectory. This is an isotropic sampling on the null space at a given radius $c$; it is not uniform over the entire, unbounded solution set.

\textit{Solving the upper-level problem}\quad Each selection $\xs_{c,k}$ defines a single-level objective $\Phi_{c,k}(\theta)$. We set the budget or the radius of the constraint to be $\tau=0.1\,n_{\rm tr}\operatorname{std}(\theta_0)=4.5$. Write $\eta_t=\theta_t-\theta_0$ and initialize $\eta_0=v_0=0$. Trajectories are generated using forward-backward, for $t=0,\dots,149$,
\begin{equation}\label{eq:forward-backward}
    u_t=0.01\,\tau\,\frac{h_t}{\|h_t\|_2+10^{-12}},\qquad
    v_{t+1}=0.9\,v_t+u_t,\qquad
    \eta_{t+1}=\Pi\left(\eta_t+v_{t+1}\right),
\end{equation}
with $\theta_t=\theta_0+\eta_t$ and $h_t=\nabla\Phi_{\min}(\theta_t)$ or $h_t=\nabla\Phi_{c,k}(\theta_t)$, where $\Pi$ is the Euclidean projection onto the constraint set.

\subsubsection{NYC 311 Huber regression}\label{ss_appendix_nyc311_huber}

\textit{Data and splits}\quad 
We use $20$ Tuesday windows from ISO weeks $2$--$21$ of 2023 in the NYC Open Data service-request records \citep{nyc_open_data_311}. Within each window, eligible requests are ordered by creation time, with the unique request key breaking ties; the first $250$ form the lower-level and the next $1000$ are validation candidates. If $t_i$ is the observed number of hours from creation to closure, the unstandardized response is
\[
    r_i=\log\left(1+\min\{t_i,720\}\right).
\]
Any missing closures are assigned the $720$-hour cap. Let $\bar r_{\rm tr}$ and $s_{\rm tr}$ be the mean and standard deviation of the $250$ lower-level responses. The clean lower-level response is $\theta_0=(r_{\rm tr}-\bar r_{\rm tr})/s_{\rm tr}$, and validation responses use the same transformation; we denote the resulting validation vector by $b_{\rm val}$. These quantities derived from the lower-level data are held fixed during the bilevel optimization.

The lower-level design matrix $A$ is the incidence matrix of the bipartite graph joining request types to incident ZIP codes, and thus every row contains exactly one request-type and one ZIP indicator. Across the $20$ windows the number of columns $p$ is between $131$ and $149$, while $\operatorname{rank}(A)$ is between $125$ and $147$, meaning $\dim\ker(A)$ is between $2$ and $9$.

\textit{Huber parameter and complete solution set}\quad We use the general Huber loss function with parameter $\delta>0$
\[
    h_\delta(r)=
    \begin{cases}
      \tfrac12 r^2, & |r|\leq\delta,\\
      \delta|r|-\tfrac12\delta^2, & |r|>\delta.
    \end{cases}
\]
Following \citet{huber1964} we have
\[
    x_{\rm LS}=A^\dagger\theta_0,\qquad
    \widehat\sigma
    =\left(
      \frac{\|Ax_{\rm LS}-\theta_0\|_2^2}
           {250-\operatorname{rank}(A)}
    \right)^{1/2},\qquad
    \delta=1.345\,\widehat\sigma.
\]
The value of $\delta$ is computed from the clean lower partition and held fixed throughout the upper optimization. 

For a given $\theta$, to compute the lower-level solution set $\mS(\theta)$, an approach is to use the convex conjugate to rewrite
\[h_{\delta} = \delta \left((\|\cdot\|_1)^* + \frac{\delta}{2}\|\cdot\|^2_2 \right)^*.\]
Using the componentwise Huber loss, the dual problem to computing $x^{\star}(\theta) \in \mS(\theta)$ reads 
\begin{equation}\label{eq_nyc311_huber_dual}
    u^\star(\theta)
    =\argmin_{u}\ \frac12\|u+\theta\|_2^2
    \quad\text{s.t.}\quad A^\top u=0,
    \quad \|u\|_\infty\leq\delta.
\end{equation}
Define $I_0=\{i:|u_i^\star|<\delta\}$, $I_+=\{i:u_i^\star=\delta\}$, and $I_-=\{i:u_i^\star=-\delta\}$. The complete primal solution polyhedron is
\begin{equation}\label{eq_nyc311_huber_polyhedron}
    \mathcal S(\theta)=\{x\colon A_{I_0}x=\theta_{I_0}+u^\star_{I_0},\quad A_{I_+}x\geq\theta_{I_+}+\delta\mathbf 1,\quad A_{I_-}x\leq\theta_{I_-}-\delta\mathbf 1\}.
\end{equation}
We solve the dual projection, construct \eqref{eq_nyc311_huber_polyhedron}, and then compute the global minimal-norm solution from the secondary quadratic program
\begin{equation}\label{eq_nyc311_huber_minnorm_qp}
    \minsol(\theta)
    =\argmin_{x\in\mathcal S(\theta)}\frac12\|x\|_2^2.
\end{equation}
Redundant equalities are removed by SVD, the QP solution is polished on its active face, and primal feasibility, duality, score consistency, and the KKT conditions of \eqref{eq_nyc311_huber_minnorm_qp} are checked independently.

\textit{Nonminimal solutions}\quad The request-type/ZIP incidence graph can have several connected components. Assigning a constant to the request-type coefficients and its negative to the ZIP coefficients within any one component does not affect the lower-level problem. These component shifts span a subspace of $\ker(A)$. We remove the universal shift, which also leaves every possible validation prediction unchanged, and form the remaining component-contrast subspace. For each split, we generate $10$ unit directions $q_k$ in this subspace that we fix for the whole trajectory. For $c\in\{0.5,1,2,3\}$ we use the exact nonminimal solutions
\begin{equation}\label{eq_nyc311_nonminimal_selection}
    x^\star_{c,k}(\theta)
    =\minsol(\theta)
     +c\,\|\minsol(\theta_0)\|_2q_k,
    \qquad k=1,\ldots,10.
\end{equation}
Since $Aq_k=0$, these selections have the same lower-level objective value as \eqref{eq_nyc311_huber_minnorm_qp}.

\textit{Computing the implicit hypergradient}\quad At every $\theta_t$, we solve
\[
    x^\star_\varepsilon(\theta_t)
    =\argmin_x\left\{
       \sum_{i=1}^{250}h_\delta((Ax-\theta_t)_i)
       +\frac{\varepsilon}{2}\|x\|_2^2
    \right\}
\]
along the decreasing grid $\{10^{-1},3\mathbin{\cdot}10^{-2},\ldots,3\mathbin{\cdot}10^{-6},10^{-7}\}$. For each level, $Q_\varepsilon$ is diagonal with entry one when the corresponding residual lies strictly inside the quadratic Huber region and zero in the clipped region. An element of the regularized implicit conservative Jacobian is given by
\[
    V_\varepsilon
    =(A^\top Q_\varepsilon A+\varepsilon I)^{-1}
      A^\top Q_\varepsilon.
\]
We take $Q=Q_{10^{-7}}$ and compute
\begin{equation}\label{eq_nyc311_huber_implicit_jacobian}
    V=(A^\top QA)^\dagger A^\top Q.
\end{equation}
The final three $Q_\varepsilon$ matrices are also compared to record whether the active pattern and derivative have stabilized. In that stable regime, Proposition~\ref{prop_huber_some_limits} certifies $V\in J_\theta\minsol(\theta)$ as the reachable pseudoinverse element.

\textit{Solving the upper-level problem}\quad We apply the same basic template as in \secref{ss_appendix_riboflavin_poisoning} but with the data and hypergradient computation adapted to this setting. Writing $\eta_t=\theta_t-\theta_0$ and taking $\theta_t=\theta_0+\eta_t$ when evaluating the lower solution and hypergradient, and initialize $\eta_0=v_0=0$. With $h_t$ the implicit hypergradient above, the forward-backwards updates mirror those in \eqref{eq:forward-backward}:
\[
    d_t=0.01(4.5)\frac{h_t}{\|h_t\|_2+10^{-12}},\qquad
    v_{t+1}=0.9v_t+d_t,\qquad
    \eta_{t+1}=\Pi(\eta_t+v_{t+1}).
\]

\subsection{Digits LASSO hypercleaning}\label{ss_additional_distillation}

We collect here details about the experiment in \secref{ss_digits_lasso_distillation}.

\textit{Data and splits}\quad The scikit-learn Digits data set has $1797$ observations, each $64$ pixels, and ten classes. For each of five seeds, we split the data into $1078$ samples for the lower-level problem and $359$ for the upper-level. Averaging adjacent horizontal pixels and repeating each average gives a fixed matrix $A\in\mbR^{10\times64}$ with $32$ exact duplicate column pairs. The variable is the synthetic set $\theta\in(\Delta^{10})^{10}$, with $\theta_0=I_{10}$.

\textit{Solving the lower-level problem}\quad With $e^{\lambda}=10^{-3}$, the lower problem is \eqref{eq_digits_lasso_lower}. The positive $\ell_1$ term makes the solution set nonempty and compact even when $\varepsilon=0$, while the duplicate columns guarantee existence of multiple solutions.

\textit{Computing the implicit hypergradient}\quad
For each output class $c$, we use the active set
\[
S_c=\left\{j:\left|(X_N-\gamma\left[\frac{1}{10}A^\top(AX_N-\theta)+\varepsilon X_N\right])_{jc}\right|>\gamma e^\lambda\right\},
\]
taking the derivative of soft-thresholding to be zero at equality. The gradient of the clean upper training loss with respect to the coefficients is
\[
B=\nabla_X L_{\rm tr}(X_N)=\frac{1}{n_{\rm tr}}X_{\rm tr}^\top(X_{\rm tr}X_N-Y_{\rm tr}).
\]
For each class $c$, the active-space adjoint $u_c\in\mathbb R^{|S_c|}$ is obtained from
\[
(\tfrac{1}{10}A_{:,S_c}^\top A_{:,S_c}+\varepsilon I)u_c=B_{S_c,c},\qquad(\widehat g_\theta)_{:,c}=\tfrac{1}{10}A_{:,S_c}u_c.
\]
If $S_c$ is empty, we set $(\widehat g_\theta)_{:,c}=0$. The ten columns form the hypergradient $\widehat g_\theta\in\mathbb R^{10\times10}$ used in forward-backward. Since $\varepsilon>0$, every active system is positive definite and is solved directly.

\textit{Solving the upper-level problem}\quad For $t=0,\ldots,99$, $N$ iterations of forward-backward are applied to the lower-level problem to compute $X_N$ which is used to compute $g_t$, the hypergradient at $(\theta_t,\varepsilon)$, followed by
\[
\theta_{t+1} = \Pi_{(\Delta^{10})^{10}}\left(\theta_t - 0.25\,\frac{g_t}{\max\{\|g_t\|_F,10^{-12}\}} \right).
\]

\begin{table}[htbp]
\centering
\small
\begin{tabular}{@{}rccc@{}}
\toprule
$N$
& Selected $\varepsilon$
& Median validation loss
& IQR \\
\midrule
$1{,}000$
& $3\cdot10^{-4}$
& $0.331981$
& $[0.330044,\;0.341265]$ \\

$10{,}000$
& $10^{-4}$
& $0.329993$
& $[0.329241,\;0.335577]$ \\

$100{,}000$
& $10^{-5}$
& $0.329972$
& $[0.329119,\;0.338587]$ \\

$250{,}000$
& $10^{-6}$
& $0.329275$
& $[0.328927,\;0.336349]$ \\
\bottomrule
\end{tabular}
\caption{Median and IQR (over five seeds) upper-level objective and best $\varepsilon$ from the sampled values after $100$ iterations on the upper-level problem, using $N$ lower-level iterations.}
\label{table_digits_lasso_implicit_minima}
\end{table}

\section{General Theorems}\label{s_classic_results}

\begin{definition}
    [Pseudoinverse \citep{moore_1919, penrose_1955}] \label{def_pseudoinverse} Let $A \in \mbR^{m \times n}$. The pseudoinverse of $A$, denoted by $A^{\dagger}$, is the only matrix in $\mbR^{n \times m}$ verifying 
    \[AA^{\dagger}A = A, \quad A^{\dagger}AA^{\dagger} = A^{\dagger}, \quad (AA^{\dagger})\tp = AA^{\dagger}, \quad (A^{\dagger}A)\tp = A^{\dagger}A\]
    In particular, when $A$ is square nonsingular, $A^{\dagger} = A^{-1}$. Furthermore, if $A$ has SVD (Singular Value Decomposition) $A = U\Sigma V\tp$, then $A^{\dagger} = V \Sigma^{\dagger}U\tp$, where $\Sigma^{\dagger} \in \mbR^{n \times m}$ is defined by
    \[(\Sigma^{\dagger})_{i,j} = 0\ \textnormal{if}\ i \neq j, \quad (\Sigma^{\dagger})_{i,i} = \sigma_i^{\dagger}, \quad  \sigma_i^{\dagger} = \left\{\begin{array}{ll} \sigma_i^{-1} & \textnormal{if}\ \sigma_i > 0,\\ 0 & \textnormal{if}\ \sigma_i = 0.\end{array}\right.\]
\end{definition}
We now present properties of pseudoinverses used throughout, which are straightforwardly verified using the four defining equations of the pseudoinverse and the SVD. 
\begin{lemma}
    [Simple properties of pseudoinverses] \label{lem_pseudoinverse_properties} Let $M \in \mbR^{m \times n}$. Then 
    \begin{itemize}
    \itemsep-0.25em
        \item $(M\ 0)^{\dagger} = (M^{\dagger} ; 0)$ and $(M; 0)^{\dagger} = (M^{\dagger} \ 0)$, $M\tp MM^{\dagger} = M\tp = M^{\dagger} M M\tp$, 
        \item $(MV\tp)^{\dagger} = V M^{\dagger}$ when $V$ is an orthogonal matrix of order $n$ and $(UM)^{\dagger} = M^{\dagger}U\tp$ when $U$ is an orthogonal matrix of order $m$,
        \item $M^{\dagger} = (M\tp M)^{\dagger}M\tp = M\tp (MM\tp)^{\dagger}$.
    \end{itemize}
\end{lemma}

The next theorem recalls some properties of the \textit{proximal operator} and of the \textit{Moreau envelope} \citep{moreau_1965, bauschke_combettes_2017}.
\begin{theorem}
    [Proximal operator and Moreau envelope]\label{thm_prox} 
    Let $f \in \Gamma_0(\mbR^n)$ and $\gamma > 0$, the proximal operator of $\gamma f$ is denoted and defined by 
    \begin{equation*}
        \prox_{\gamma f}(x) := \underset{y \in \mbR^n}{\argmin} \left(\frac{\|y - x\|_2^2}{2} + \gamma f(y)\right)
    \end{equation*}
    which is uniquely defined by strong convexity of the criterion. 
    The Moreau envelope of $f$ is the $\mC^{1,1}$ function $f^{\gamma} \in \Gamma_0(\mbR^n)$ defined, for $\gamma > 0$, by 
    \[f^{\gamma}(x) = \underset{y \in \mbR^n}{\inf}\left(\frac{\|y - x\|_2^2}{2\gamma} + f(y)\right) = f(\prox_{\gamma f}(x)) + \frac{1}{2\gamma}\|\prox_{\gamma f}(x) - x\|_2^2.\]
    In particular, $f^{\gamma}$ and $f$ have the same minimizers, $f^{\gamma}(x) \leq f(x)$ for all $x \in \mbR^n$ and the gradient of $f^{\gamma}$ is given by $\nabla f^{\gamma}(x) = \frac{x - \prox_{\gamma f}(x)}{\gamma}$.
\end{theorem}
\begin{theorem}
    [Clarke Jacobian of the proximal operator {\citep[theorem 3.2]{patrinos_stella_bemporad_2014}}] \label{thm_CF_prox} Let $f \in \Gamma_0(\mbR^n)$ and $\gamma > 0$, for every $P \in \partial (\prox_{\gamma f})(x)$, $P$ is symmetric positive semidefinite and verifies $\|P\| \leq 1$, that is, the eigenvalues of $P$ are in $[0, 1]$. 
\end{theorem}
This supplies symmetry and spectral bounds but not definability and conservativeness. When the proximal mapping is locally Lipschitz and definable, it is path-differentiable, thus its Clarke Jacobian is a conservative mapping. For a parameterized proximal mapping, the analogous conclusion requires these properties jointly in the input and parameter variables.

We now recall specific results when this theory is applied to $f = \|\cdot\|_1$. 
\begin{corollary}
    [Proximal operator of $\|\cdot\|_1$] \label{coro_soft_T} The proximal operator of the 1-norm $\|\cdot\|_1$ is the \textbf{soft-thresholding} operator, given by 
    \[\prox_{\gamma \|\cdot\|_1}(x) = 
    \sign(x)(|x| - \gamma)_+ = \left(\sign(x_i)(\max(|x_i| - \gamma, 0))\right)_{i \in {[1:n]}}.\]
\end{corollary}
\begin{corollary}
    [Clarke Jacobian of $\prox_{\gamma\|\cdot\|_1}$] \label{coro_CF_l1} The Clarke Jacobian of $\prox_{\gamma\|\cdot\|_1}$, is given by
    \[\partial \prox_{\gamma \|\cdot\|_1}(x) = \prod_{i=1}^n \partial \prox_{\gamma |\cdot|}(x_i), \ x = (x_i)_{i \in [1:n]} \in \mbR^n \quad \partial \prox_{\gamma |\cdot|}(x_i) = \left\{ \begin{array}{ll} 1 & |x_i| > \gamma, \\ {[}0, 1] & |x_i| = \gamma, \\ 0 & |x_i| < \gamma.\end{array} \right.\]
\end{corollary}
We now recall properties of Tikhonov regularization \citep{tikhonov_1963}.
\begin{theorem}
    [Hierarchical minimization]\label{thm_Browder} Let $f \in \Gamma_0(\mbR^n)$ with $\mS := \argmin\ f \neq \varnothing$ and $\minsol := \Pi_{\mS}(0) = \argmin_{x \in \mS}\|x\|_2$. For $\varepsilon > 0$, let $\xseps$ be the unique minimizer of $f + \varepsilon \|\cdot\|_2^2/2$. Then, for $\varepsilon > 0$ and $0 < \varepsilon_1 \leq \varepsilon_2$, 
    \[\|\xseps\|_2 \leq \|\minsol\|_2,\qquad \|\xs_{\varepsilon_2}\|_2 \leq \|\xs_{\varepsilon_1}\|_2 \leq \|\minsol\|_2, \qquad \lim_{\etz} \|\xseps - \minsol\|_2 = 0.\]
\end{theorem}

\begin{definition}
    [H\"{o}lderian error bound]\label{def_Holder_function} A proper function $f$ with a nonempty set of minimizers $\mS$ is said to verify a H\"{o}lderian error bound with exponent $p \geq 1$ if there exist constants $r > \min f$ and $\kappa > 0$ such that, when $f(x) < r$, $f(x) - \min f \geq \kappa \mathrm{dist}(x, \mS)^p$.
\end{definition}
\begin{theorem}
    [Rate with Hölderian error bound {\citep[Proposition~4.10]{maulen-soto_fadili_attouch_2025}}]\label{thm_holder_bound} Let $f \in \Gamma_0(\mbR^n)$ have a nonempty set of minimizers $\mS$ and satisfy a Hölderian error bound with exponent $p$ as in \defref{def_Holder_function}. Let $\minsol := \Pi_{\mS}(0)$ and $\xseps$ be the unique minimizer of $f + \varepsilon \|\cdot\|^2/2$, then there exist constants $\bar{\varepsilon}$ and $C_0$ such that 
    \[\forall\ \varepsilon \in ]0, \bar{\varepsilon}],\qquad \|\xseps - \minsol\| \leq C_0 \varepsilon^{\frac{1}{2p}}.\]
\end{theorem}

\section{Definability of the Tikhonov Construction}\label{s_definability}
We briefly recall definability, then prove that the Tikhonov solutions and regularized Jacobians form definable families under our assumptions. See \citet{vandendries_miller_1996, coste1999introduction, ioffe_2009} for detailed introductions to definability and o-minimal geometry.

\subsection{A Glimpse of Definability} \label{ss_definability_intro}

\begin{definition}
    [o-minimal structure]\label{def_ominimal} An o-minimal structure on $(\mbR, +, \times)$ is a collection $\mO = (\mO_p)_{p \in \mbN}$ such that each $\mO_p$ is a collection of subsets of $\mbR^p$ verifying the following axioms:
\begin{itemize}
\itemsep-0.25em
    \item $\mO_1$ consists exactly of the finite unions of points and open intervals;
    \item $\mO_p$ contains $\mbR^p$ and is stable by complement, finite union, and finite intersection;
    \item $A \in \mO_p \Rightarrow A \times \mbR \in \mO_{p+1}$ and $\mbR \times A \in \mO_{p+1}$;
    \item if $A \in \mO_{p+1}$ and $\pi : \mbR^{p+1} \rightarrow \mbR^p$ is the canonical projection onto $\mbR^p$, then $\pi(A) \in \mO_p$;
    \item $\mO_p$ contains the real algebraic subsets of $\mbR^p$, that is, the sets of the form $\{x\in \mbR^p : g(x) = 0\}$ where $g$ is a polynomial.
\end{itemize}
A subset of $\mbR^p$ is definable if it belongs to $\mO_p$. A function $f : \mbR^n \rightarrow \mbR^m$ is definable if its graph is definable.
\end{definition}

Important examples include semialgebraic sets, which are finite unions of sets of the form
\[\bigcap^k_{i=1} \{x \in \mbR^p : g_i(x) < 0, h_i(x) = 0\}\qquad g_i, h_i\ \textnormal{polynomial}.\]
The logarithm and exponential are also definable in suitable o-minimal structures. Definability is preserved by standard operations, including composition. The definable class therefore contains semialgebraic functions, neural networks built from definable elementary functions, and outputs of backpropagation \citep[Section~5]{bolte_pauwels_2019}.

\subsection{Definability of the Tikhonov Solution Maps}

\begin{proposition}[Definability of the Tikhonov solution maps]\label{prop:def_tikhonov}
Assume \ref{assum:f_reg} and \ref{assum:f_def}. Then $(\theta,\varepsilon)\mapsto\xsepst$ is definable on $\Theta\times\mbR_+^*$. The minimal-norm solution $\minsol$ and the joint selection $\xs$ defined in \secref{ss_Schechtman_limits} are also definable.
\end{proposition}

\begin{proof}
For $\varepsilon\geq0$, set $f_\varepsilon(x,\theta):=f(x,\theta)+\varepsilon\|x\|_2^2/2$ and consider the full solution graph
\[
\mathcal T:=\left\{(\theta,\varepsilon,x)\in\Theta\times\mbR_+\times\mbR^n:x\in\argmin_y f_\varepsilon(y,\theta)\right\}.
\]
The set $\mathcal A:=\{(\theta,\varepsilon,x,y)\in\Theta\times\mbR_+\times\mbR^n\times\mbR^n:f_\varepsilon(y,\theta)<f_\varepsilon(x,\theta)\}$ is definable, and
\[
\mathcal T=(\Theta\times\mbR_+\times\mbR^n)\setminus\pi_{\theta,\varepsilon,x}(\mathcal A).
\]
Since definable sets are stable under projection and complement, $\mathcal T$ is definable.

For every $\theta\in\Theta$ and $\varepsilon>0$, the regularized objective $f_\varepsilon(\cdot,\theta)$ is $\varepsilon$-strongly convex and therefore has a unique minimizer. Thus
\[
\operatorname{gph}\big((\theta,\varepsilon)\mapsto\xsepst\big)=\mathcal T\cap(\Theta\times\mbR_+^*\times\mbR^n)
\]
is definable.

The unregularized solution correspondence has graph $\{(\theta,x):(\theta,0,x)\in\mathcal T\}$. Its values are nonempty, closed, and convex, so each contains a unique element of minimal norm. The graph of $\minsol$ is
\[
\left\{(\theta,x):(\theta,0,x)\in\mathcal T\ \text{and there is no }y\text{ such that }(\theta,0,y)\in\mathcal T\text{ and }\|y\|_2^2<\|x\|_2^2\right\},
\]
which is definable. Finally, $\operatorname{gph}(\xs)$ is the union of the graph of $(\theta,\varepsilon)\mapsto\xsepst$ on $\varepsilon>0$ and the graph of $\minsol$ at $\varepsilon=0$, hence it is definable as well.
\end{proof}

\subsection{Definability of the Regularized Conservative Jacobians}
\label{ss_def_Jac_nonzero}

We now give the corresponding definability statement for the conservative Jacobians obtained by implicit differentiation. We first recall an elementary fact used in the proof of \thmref{thm_composite_pathdiff} \citep[see][]{bolte_daniilidis_lewis_shiota_2007}.

\begin{lemma}[Definability of Clarke Jacobians]
\label{lem:def_clarke_jac}
Let \(U\subset\mbR^d\) be open and definable, and let
\(G:U\to\mbR^q\) be locally Lipschitz and definable. Then the Clarke Jacobian $\partial G:U\setvalued\mbR^{q\times d}$ has definable graph.
\end{lemma}

\begin{proposition}[Definability of the regularized implicit conservative Jacobians]\label{prop_defin_derivative}
Assume \ref{assum:f_reg}, \ref{assum:f_def} and \ref{assum:range}. For $0<\varepsilon\leq\bar\eps$, let $J_{\xseps}$ be the implicit conservative Jacobian in \eqref{eq:regularized_implicit_jacobian}, and let $J_{\xs}(\theta,\varepsilon):=J_{\xseps}(\theta)$ be the associated family defined in \secref{ss_Schechtman_limits}. Then $J_{\xs}$ has definable graph on $O\times]0,\bar\eps]$, and every fixed-$\varepsilon$ slice $J_{\xseps}$ is a conservative Jacobian of $\theta\mapsto\xseps(\theta)$ on $O$.
\end{proposition}

\begin{proof}
By \propref{prop:def_tikhonov}, $(\theta,\varepsilon)\mapsto\xsepst$ is definable on $O\times]0,\bar\eps]$. Moreover,
\[
D\in J_{\xs}(\theta,\varepsilon)
\quad\Longleftrightarrow\quad
\exists\,(H_\varepsilon\ M_\varepsilon)\in J_F(\xsepst,\theta)\ \text{such that }(H_\varepsilon+\varepsilon I_n)D+M_\varepsilon=0.
\]
By Assumption~\ref{assum:range}, $J_F$ has definable graph, and the matrix equation above is polynomial in $(\varepsilon,H_\varepsilon,M_\varepsilon,D)$. Thus $\operatorname{gph}(J_{\xs})$ is the projection of a definable set and is therefore definable.

For fixed $\varepsilon>0$, the regularized objective is definable and $\varepsilon$-strongly convex. Hence \citet[Theorem~4.2]{bolte_pauwels_silveti-falls_2024} shows that $\xseps$ is path-differentiable, in particular locally Lipschitz. Assumption~\ref{assum:range} makes every $H_\varepsilon+\varepsilon I_n$ along the regularized solutions positive definite, so \thmref{thm_IFT_other} gives \eqref{eq:regularized_implicit_jacobian}. Thus $J_{\xseps}$ is a conservative Jacobian of $\xseps$ on $O$.
\end{proof}

\section{Results from \secref{s_C11}}\label{s_C11_proofs}
This section first justifies the general formula for limit elements then the pseudoinverse case, then details the computations used for the problem with the Huber loss. 
\subsection{Proof of the Residual Formula \thmref{thm_formula_memory}}\label{ss_proof_residual_formula}
\begin{proof}[\thmref{thm_formula_memory}]
    Fix the sequence $\varepsilon_k\downarrow0$ and the blocks $(H_k\ M_k)$ from the theorem. Assumption~\ref{assum:range}, applied with $K=\{\theta\}$, gives a constant $C_\theta$ and matrices $\widehat W_k$ such that $M_k=H_k\widehat W_k$ and $\|\widehat W_k\|\leq C_\theta$. Since $H_k$ is symmetric PSD, $P_k:=H_k^{\dagger}H_k$ is the orthogonal projector onto $\range(H_k)$, and therefore
    \[
        W_k=H_k^{\dagger}M_k=P_k\widehat W_k,
        \qquad \|W_k\|\leq C_\theta,
        \qquad H_kW_k=M_k.
    \]
    The eigenvalues of $R_k=(H_k+\varepsilon_kI_n)^{-1}H_k$ are $\lambda_i(H_k)/(\lambda_i(H_k)+\varepsilon_k)\in[0,1]$. Hence $0\preceq R_k\preceq I_n$, and $V_k=-R_kW_k$ is bounded as well.

    Let $V$ be a cluster point and pass to a subsequence along which $V_k\to V$. By \thmref{thm_Browder}, $x_{\varepsilon_k}^{\star}(\theta)\to\minsol(\theta)$. Using the local boundedness of $J_F$ and the boundedness of $(W_k)_k$ and $(R_k)_k$, we may pass to a further subsequence such that
    \[
        (H_k\ M_k)\to(H\ M),
        \qquad W_k\to W,
        \qquad R_k\to R.
    \]
    The chosen sequence and the closed graph property give $(H\ M)\in\mR(\theta)$, and therefore
    \[
        M=HW,
        \qquad V=-RW,
        \qquad 0\preceq R\preceq I_n.
    \]

    The matrices $H_k$ and $R_k$ commute, and
    \[
        \|H_kR_k-H_k\|
        =\max_i\frac{\varepsilon_k\lambda_i(H_k)}{\lambda_i(H_k)+\varepsilon_k}
        \leq\varepsilon_k.
    \]
    Passing to the limit yields $HR=H=RH$. Thus $R$ acts as the identity on $\range(H)$ and leaves $\ker(H)$ invariant. With $P:=H^{\dagger}H=\Pi_{\range(H)}$, we obtain
    \[
        S:=R-P=\Pi_{\ker H}R\,\Pi_{\ker H},
        \qquad 0\preceq S\preceq\Pi_{\ker H}.
    \]
    Since $H^{\dagger}M=PW$, we conclude that
    \[
        V=-RW=-H^{\dagger}M-SW.
    \]
    Since $V_k\in J_{\xs}(\theta,\varepsilon_k)$, taking $\theta_k\equiv\theta$ in \eqref{eq:J_min} gives $V\in J_{\minsol}(\theta)$. Finally, $M=HW$ and $HS=0$, so
    \[
        HV+M=-HH^{\dagger}M-HSW+M=0.
    \]
\end{proof}
\subsection{Proof of the Pseudoinverse Formula \thmref{thm_formula}}\label{ss_proof_pseudoinverse_formula}
\begin{proof}[\thmref{thm_formula}]
    Fix $(H\ M)\in\mR(\theta)$, and choose $\varepsilon_k\downarrow0$ and $(H_k\ M_k)\in J_F(x_{\varepsilon_k}^{\star}(\theta),\theta)$ realizing this limit. Discarding finitely many terms, we may assume that $0<\varepsilon_k\leq\bar\eps$. Define
    \[
        R_k:=(H_k+\varepsilon_kI_n)^{-1}H_k,
        \qquad W_k:=H_k^{\dagger}M_k,
        \qquad V_k:=-(H_k+\varepsilon_kI_n)^{-1}M_k=-R_kW_k.
    \]
    Assumption~\ref{assum:range}, applied with $K=\{\theta\}$, gives $C_\theta>0$ and matrices $\widehat W_k$ such that $M_k=H_k\widehat W_k$ and $\|\widehat W_k\|\leq C_\theta$. Since $H_k^{\dagger}H_k$ is an orthogonal projector, $W_k=H_k^{\dagger}H_k\widehat W_k$ is bounded and $H_kW_k=M_k$. Passing to a subsequence, we may assume that $W_k\to W$, and hence $M=HW$.

    We claim that $R_k\to H^{\dagger}H$. Let $H_k=\sum_{i=1}^n\lambda_i(H_k)u_{i,k}u_{i,k}\tp$ be a spectral decomposition with $\lambda_1(H_k)\geq\cdots\geq\lambda_n(H_k)\geq0$. Then
    \[
        R_k=\sum_{i=1}^n\frac{\lambda_i(H_k)}{\lambda_i(H_k)+\varepsilon_k}u_{i,k}u_{i,k}\tp.
    \]
    Suppose first that \ref{assum:eigenvalues_gap} holds. For large $k$, we also have $\varepsilon_k\leq\varepsilon_G$. With $P_k:=\Pi_{\range(H_k)}$,
    \[
        \|R_k-P_k\|=\max_{i:\lambda_i(H_k)>0}\frac{\varepsilon_k}{\lambda_i(H_k)+\varepsilon_k}\leq\frac{\varepsilon_k}{\lambda^*}\longrightarrow0,
    \]
    where the maximum is taken to be $0$ if $H_k=0$. The fixed spectral gap and $H_k\to H$ also give, using that the number of eigenvalues of $H_k$ above $\lambda^*$ is eventually identical to $\mathrm{rank}(H)$, $P_k\to\Pi_{\range(H)}=H^{\dagger}H$.

    Suppose instead that \ref{assum:eigenvalues_vanishing} holds, and set $r:=\mathrm{rank}(H)$. The first $r$ eigenvalues of $H_k$ converge to the positive eigenvalues of $H$, while the remaining eigenvalues converge to zero. If $P_k^r$ denotes the spectral projector onto the first $r$ eigenspaces, with $P_k^0:=0$, then
    \[
        \|R_k-P_k^r\|=\max\left\{\max_{i\leq r}\frac{\varepsilon_k}{\lambda_i(H_k)+\varepsilon_k},
        \max_{i>r}\frac{\lambda_i(H_k)}{\lambda_i(H_k)+\varepsilon_k}\right\}\longrightarrow0.
    \]
    We take for convention the maximum over an empty index set to be $0$. The first maximum tends to $0$ because the limiting eigenvalues are positive, while the second tends to $0$ by \ref{assum:eigenvalues_vanishing}. If $r=0$ or $r=n$, then $P_k^r\to\Pi_{\range(H)}$ is immediate because $P_k^0=0$ and $P_k^n=I_n$. If $0<r<n$, the positive and null spectra of $H$ are separated, so the continuity of the corresponding spectral projector gives the same conclusion. Thus $P_k^r\to\Pi_{\range(H)}=H^{\dagger}H$, which proves the claim in both cases.

    It follows that
    \[
        V_k=-R_kW_k\longrightarrow-H^{\dagger}HW=-H^{\dagger}M.
    \]
    Since $V_k\in J_{\xs}(\theta,\varepsilon_k)$ and $x_{\varepsilon_k}^{\star}(\theta)\to\minsol(\theta)$ by \thmref{thm_Browder}, the definition \eqref{eq:J_min} gives $-H^{\dagger}M\in J_{\minsol}(\theta)$. The reachable block $(H\ M)$ was arbitrary, which proves \eqref{formula_eigenvalues_controlled}.
\end{proof}
\subsection{Uniform Bounds for Weighted Gram Systems}\label{ss_proof_weighted_gram}
The proof of \corref{coro_scaled_gram_weighted} requires a general estimate for Gram
systems regularized by a positive diagonal matrix. We state and prove this estimate
first in \lemref{lem_scaled_gram_bound} because it will also be used for the LASSO,
then we prove \corref{coro_scaled_gram_weighted}.

\begin{lemma}[Uniform bound for diagonally regularized Gram systems]
\label{lem_scaled_gram_bound}
Let $r,s\in\mbN$ with $r,s\geq1$ and $B\in\mbR^{r\times s}$, and let $\chi(B)$
denote the quantity defined in \eqref{eq:chi_A} with $A$ replaced by $B$. Then, for
every positive definite diagonal matrix $D\in\mbR^{s\times s}$,
\[
    \left\|(D+B\tp B)^{-1}B\tp\right\|\leq\chi(B).
\]
\end{lemma}

\begin{proof}[\lemref{lem_scaled_gram_bound}]
Fix $v\in\mbR^r$ and write $u=(D+B\tp B)^{-1}B\tp v$. We first represent $u$
as part of a weighted minimal-norm solution of a linear system. Note the identity
\begin{equation}\label{eq_push_through}
    (D+B\tp B)^{-1}B\tp
    =D^{-1}B\tp(I_r+BD^{-1}B\tp)^{-1}
\end{equation}
which follows by multiplying both sides by $D+B\tp B$ and rearranging terms. Set
\[
    A_0:=(B\ I_r),
    \qquad
    \Omega:=\diag(D^{-1},I_r).
\]
Then $I_r+BD^{-1}B\tp=A_0\Omega A_0\tp$, and \eqref{eq_push_through}
gives
\[
    u=(I_s\ 0)z^{\star},
    \qquad
    z^{\star}:=\Omega A_0\tp(A_0\Omega A_0\tp)^{-1}v.
\]
Equivalently, $z^{\star}$ is the unique solution of
\begin{equation}\label{eq_weighted_min_norm}
    \min_{z\in\mbR^{s+r}}\ \frac12 z\tp\Omega^{-1}z
    \qquad\text{subject to}\qquad A_0z=v.
\end{equation}

We next use the fact that $z^{\star}$ is a convex combination of the basic
solutions of $A_0z=v$. To see this, let the columns of
$N\in\mbR^{(s+r)\times s}$ form a basis of $\ker(A_0)$, and take
$z_p=(0;v)$ as a particular solution. Every feasible point has the form
$z_p+N\beta$, so \eqref{eq_weighted_min_norm} is a diagonally weighted
least-squares problem in $\beta$. The geometric description of weighted
least-squares solutions in \citet{ben-tal_teboulle_1990} gives, for $\beta^{\star}$ such that $z^{\star} = z_p + N\beta^{\star}$,
\[
    \beta^{\star}\in\mathrm{conv}\left\{
    -(N_{\mI,:})^{-1}(z_p)_\mI:
    |\mI|=s,\ N_{\mI,:}\text{ nonsingular}
    \right\}.
\]
For every such $\mI$, the corresponding point
\[
    \zeta^\mI:=z_p-N(N_{\mI,:})^{-1}(z_p)_\mI
\]
is feasible and vanishes on $\mI$. Moreover, $A_{0,:,\mI^c}$ is nonsingular. Indeed,
otherwise a nonzero vector supported on $\mI^c$ would lie in $\ker(A_0)$, and its
representation $N\beta$ would satisfy $N_{\mI,:}\beta=0$, contradicting the
nonsingularity of $N_{\mI,:}$. Thus $\zeta^\mI$ is the basic solution supported on
$\mI^c$, which is a basis of $A_0$. Since $\beta\mapsto z_p+N\beta$ is affine, we have proved that
\begin{equation}\label{eq_barycentric}
    z^{\star}\in\mathrm{conv}\{z^\mJ:\mJ\text{ is a basis of }A_0\}.
\end{equation}

It remains to bound the first block of each basic solution. Any basis of
$A_0=(B\ I_r)$ can be written as
\[
    \mJ=\mI_1\cup\{s+k:k\in K\},
\]
where $\mI_1\subset[1:s]$, $K\subset[1:r]$, and $|\mI_1|+|K|=r$.
Expanding along the selected identity columns shows that $B_{K^c,\mI_1}$ is
nonsingular. If $z^\mJ=(u';w)$ is the corresponding basic solution, then $u'$ is supported on $\mI_1$, 
\[
    \begin{pmatrix}
        B_{K^c, \mI_1} & I_{K^c, K} \\
        B_{K, \mI_1} & I_{K, K}
    \end{pmatrix} \begin{pmatrix} u' \\ w \end{pmatrix} = \begin{pmatrix} v_{K^c} \\ v_K \end{pmatrix}
\]
and the rows indexed by $K^c$ give $(u')_{\mI_1}=(B_{K^c,\mI_1})^{-1}v_{K^c}$ (since $I_{K^c, K} = 0$). Consequently,
\[
    \|(I_s\ 0)z^\mJ\|\leq\chi(B)\|v\|.
\]
This also covers $\mI_1=\emptyset$, when the first block is zero. Combining this
bound with \eqref{eq_barycentric} yields
\[
    \|(D+B\tp B)^{-1}B\tp v\|
    =\|(I_s\ 0)z^{\star}\|
    \leq\chi(B)\|v\|,
\]
which proves the lemma.
\end{proof}

\begin{proof}[\corref{coro_scaled_gram_weighted}]
Fix $u\in\mbR^m$, and set $B:=\sqrt{Q}A$ and $y:=\sqrt{Q}u$. Then
\[
    (\delta I_n+A\tp QA)^{-1}A\tp Qu
    =(\delta I_n+B\tp B)^{-1}B\tp y.
\]
The proof of \lemref{lem_scaled_gram_bound}, with $D=\delta I_n$, represents the
right-hand side as the first block of a convex combination of basic solutions of
$(B\ I_m)z=y$. Consider one such basic solution, corresponding to
\[
    \mJ=\mI_1\cup\{n+k:k\in K\}.
\]
Its first block is supported on $\mI_1$, and the rows indexed by $K^c$ give
\[
    (z^\mJ)_{\mI_1}=(B_{K^c,\mI_1})^{-1}y_{K^c}.
\]
Here $|\mI_1|=|K^c|$. Nonsingularity of
$B_{K^c,\mI_1}=\sqrt{Q_{K^c}}A_{K^c,\mI_1}$ implies that $Q_{K^c}$ is positive
definite and $A_{K^c,\mI_1}$ is nonsingular. Since
$y_{K^c}=\sqrt{Q_{K^c}}u_{K^c}$, the matching row weights cancel:
\[
    (z^\mJ)_{\mI_1}
    =\left(\sqrt{Q_{K^c}}A_{K^c,\mI_1}\right)^{-1}
      \sqrt{Q_{K^c}}u_{K^c}
    =(A_{K^c,\mI_1})^{-1}u_{K^c}.
\]
Thus the first block of every basic solution has norm at most
$\chi(A)\|u\|$; when $\mI_1=\emptyset$, it is zero. The same bound holds for
their convex combination, proving
\[
    \left\|(\delta I_n+A\tp QA)^{-1}A\tp Q\right\|\leq\chi(A).
\]

For fixed $Q$, an SVD of $B=\sqrt{Q}A$ gives
\[
    (\delta I_n+B\tp B)^{-1}B\tp\sqrt{Q}
    \longrightarrow B^{\dagger}\sqrt{Q}
    \qquad\text{as }\delta\downarrow0.
\]
Passing to the limit in the first estimate proves the second one.
\end{proof}

\begin{remark}[Scaled projections and pseudoinverses]\label{rem_scaled_pinv}
The estimate in \lemref{lem_scaled_gram_bound} uses arguments common to results on scaled projections and pseudoinverses. For a full row rank matrix $A_0$, the weighted pseudoinverses $\Omega A_0\tp(A_0\Omega A_0\tp)^{-1}$ are bounded uniformly over positive definite diagonal matrices $\Omega$ \citep{stewart_1989,oleary_1990,forsgren_1996}. This is used in the condition measure $\bar\chi_{A_0}$ often employed for complexity analysis of interior-point methods \citep{vavasis_ye_1996}. Such condition measures can be large but we only need a uniform bound. To the best of our knowledge, \lemref{lem_scaled_gram_bound} does not already appear somewhere in this literature, although the underlying arguments are not unknown.
\end{remark}

\subsection{Path-differentiability for Huber Regression \corref{coro_Huber_app}}\label{ss_proof_huber_pathdiff}
\begin{proof}[\corref{coro_Huber_app}]
    Let $f(x,\theta):=\phi(Ax-\theta)$. The Huber loss is finite, convex, semialgebraic, and $\mC^{1,1}$. Moreover, $h(t)\geq |t|-1/2$, so $\phi$ is coercive. Minimizing $\phi$ over the closed affine space $\range(A)-\theta$ therefore gives a solution of \eqref{H} for every $\theta\in\mbR^m$. Since $\nabla\phi$ is $1$-Lipschitz,
    \[
        \|F(x,\theta)-F(x',\theta')\|\leq\|A\|\bigl(\|A\|\|x-x'\|+\|\theta-\theta'\|\bigr),
    \]
    and Assumptions \ref{assum:f_reg} and \ref{assum:f_def} hold with $\Theta=\mbR^m$. The chain-rule construction and calculation in \secref{ss_Huber} verify Assumption~\ref{assum:range} globally for any $\bar\eps>0$, with the uniform bound $\|W_{\varepsilon}\|\leq\chi(A)$. The conclusion follows from \thmref{thm_general}.
\end{proof}

\vspace{-0.2cm}

\subsection{Proof of the Huber Limiting-Derivative Formula \propref{prop_huber_some_limits}}\label{ss_proof_huber_limits}
\begin{proof}[\propref{prop_huber_some_limits}]
    Choose $\varepsilon_k\downarrow0$ and $Q_{\varepsilon_k}\to Q$ as in the definition of $\mR^{\mathrm{Hub}}(\theta)$, and write $y_k:=Ax_{\varepsilon_k}^{\star}(\theta)-\theta$. The coordinatewise description of $\partial(\nabla\phi)$ shows that, for all large $k$, the constant matrix $Q$ itself belongs to $\partial(\nabla\phi)(y_k)$. If $Q_{ii}\in]0,1[$, then eventually $(Q_{\varepsilon_k})_{ii}\in]0,1[$ and hence $|(y_k)_i|=1$. If $Q_{ii}=0$, then eventually $(Q_{\varepsilon_k})_{ii}<1$ and hence $|(y_k)_i|\geq1$. Finally, if $Q_{ii}=1$, then eventually $(Q_{\varepsilon_k})_{ii}>0$ and hence $|(y_k)_i|\leq1$.

    Set $B:=\sqrt{Q}A$. For all large $k$, we may therefore select
    \[
        V_k:=(\varepsilon_kI_n+B\tp B)^{-1}B\tp\sqrt{Q}
        \in J_{\xs}(\theta,\varepsilon_k).
    \]
    An SVD of $B$ gives $V_k\to B^{\dagger}\sqrt{Q}$. Together with $x_{\varepsilon_k}^{\star}(\theta)\to\minsol(\theta)$, the definition \eqref{eq:J_min} yields $B^{\dagger}\sqrt{Q}\in J_{\minsol}(\theta)$. Finally, since $B\tp B=A\tp QA$, $B\tp\sqrt{Q}=A\tp Q$, and $(B\tp B)^{\dagger}B\tp=B^{\dagger}$ (see \lemref{lem_pseudoinverse_properties}), we have
    \[
        B^{\dagger}\sqrt{Q}=(A\tp QA)^{\dagger}A\tp Q,
    \]
    which proves the claim.
\end{proof}

\section{Results from \secref{s_involved_pbs}}\label{s_composite}
We recall problem \eqref{pb_composite}, its Tikhonov regularization \eqref{pb_composite_epsilon}, and the fixed-point formulation
\eqref{composite_fixed_point}. For $D=(Q,P,H^f,M^f)\in\mD_\varepsilon(x,\theta)$, the blocks of the conservative Jacobian \eqref{composite_constructed_field} of the residual
$\Phi_\varepsilon^\gamma$ and the reduced operators of
\eqref{composite_Sigma_V} are, with $\Pi_Q=QQ^{\dagger}$
\[
\begin{aligned}
    \Heps(D)&=I_n-Q\big((1-\gamma\varepsilon)I_n-\gamma H^f\big),
    &\Meps(D)&=\gamma QM^f-P,\\
    \Sigma(D)&=Q^\dagger-\Pi_Q+\gamma\Pi_QH^f\Pi_Q,
    &V(D)&=Q^\dagger\Meps(D)
        -\gamma\Pi_QH^f(I_n-\Pi_Q)\Meps(D),\\
    T_\varepsilon(D)&=\Sigma(D)+\gamma\varepsilon\Pi_Q.
\end{aligned}
\]
The subsections below follow the order of the
results in \secref{s_involved_pbs}: the regularized conservative Jacobian,
the LASSO uniform derivative bound, projected inversion, composite
path-differentiability and cluster-point formulas, and the reachable binary
LASSO derivatives.

\subsection{Path-differentiability of Regularized Composite Problems \propref{prop_composite_regularized}}
\label{ss_proof_composite_regularized}

\begin{proof}[\propref{prop_composite_regularized}]
Fix $0<\varepsilon\leq\bar\eps$. The smooth function
\[
    x\longmapsto f(x,\theta)+\frac{\varepsilon}{2}\|x\|_2^2
\]
is $\varepsilon$-strongly convex and has an $(L+\varepsilon)$-Lipschitz gradient. The Tikhonov term and its unique minimizer are definable, and the joint Clarke Jacobians are definable by \lemref{lem:def_clarke_jac}; hence all the objects in the fixed-point construction and in \eqref{composite_regularized_jacobian} are definable.
Moreover, the fixed stepsize satisfies
\[
    0<\gamma<\frac{2}{L+2\bar\eps}<\frac{2}{L+\bar\eps}
    \leq\frac{2}{L+\varepsilon}.
\]
The gradient and the parameterized proximal map are therefore jointly path-differentiable. The stepsize argument of \citet[Remark~3.6]{bolte_pauwels_silveti-falls_2024} applies in the same way when strong convexity lies in the smooth term. Thus \citet[Theorem~3.5]{bolte_pauwels_silveti-falls_2024}, applied componentwise on $O$, gives the path-differentiability of $\theta\mapsto\xsepst$ and exactly the definable conservative Jacobian in \eqref{composite_regularized_jacobian}.
\end{proof}

\subsection{Proof of the LASSO Uniform Derivative Bound \propref{prop_LASSO_unif_bound}}
\label{ss_proof_LASSO_uniform_bound}
We use the problem notation introduced in \secref{ss_LASSO}.

\begin{proof}[\propref{prop_LASSO_unif_bound}]
    If $Q = 0$ the matrix is zero and the bound is trivial. Otherwise, up to permuting the coordinates of $x$ (which multiplies the considered matrix on the left by a permutation matrix and permutes the columns of $A$, changing neither the operator norm nor $\chi(A)$), assume $Q = \diag(Q', 0)$ with $Q' \succ 0$ of order $q \geq 1$, and write $A_1 := A_{:,1:q}$, $A_2 := A_{:,q+1:n}$. Then
    \[I_n - Q + \gamma Q(A\tp A + \varepsilon I_n) = \begin{pmatrix} I_q - Q' + \gamma\varepsilon Q' + \gamma Q'A_1\tp A_1 & \gamma Q'A_1\tp A_2 \\ 0 & I_{n-q}\end{pmatrix}, \gamma QA\tp = \begin{pmatrix} \gamma Q'A_1\tp \\ 0\end{pmatrix},\]
    so that, by blockwise inversion of the upper triangular matrix,
    \[\left(I_n - Q(I_n - \gamma (A\tp A + \varepsilon I_n))\right)^{-1}\gamma Q A\tp = \begin{pmatrix} \left(I_q - Q' + \gamma\varepsilon Q' + \gamma Q'A_1\tp A_1\right)^{-1}\gamma Q'A_1\tp \\ 0 \end{pmatrix}.\]
    Factoring $\gamma Q' \succ 0$,
    \[I_q - Q' + \gamma\varepsilon Q' + \gamma Q'A_1\tp A_1 = \gamma Q'\left(D_{\varepsilon} + A_1\tp A_1\right), \qquad D_{\varepsilon} := \frac{(Q')^{-1} - I_q}{\gamma} + \varepsilon I_q,\]
    so that the nonzero block equals $(D_{\varepsilon} + A_1\tp A_1)^{-1}A_1\tp$. The entries of $Q'$ lying in $]0, 1]$ and $\varepsilon$ being positive, $D_{\varepsilon}$ is diagonal positive definite, and \lemref{lem_scaled_gram_bound} applied with $B = A_1$ bounds the norm of this block, hence of the whole matrix, by $\chi(A_1) \leq \chi(A)$, the last inequality holding because every square submatrix of $A_1$ is a square submatrix of $A$. Finally, by \corref{coro_LASSO_setup} and \corref{coro_CF_l1}, every element of $J_{\xseps}(\theta)$ is of this form for some $Q \in \diag([0,1]^n)$, whence, for every compact $K$ and every $0<\varepsilon_0\leq\bar\eps$, $\sup\{\|J\| : J \in J_{\xseps}(\theta),\ 0<\varepsilon \leq \varepsilon_0,\ \theta \in K\} \leq \chi(A)$, which is condition $(iv)$ of \thmref{thm_Schechtman}.
\end{proof}

\subsection{Proof of Projected-Inversion \lemref{lem_projected_inversion}}
\label{ss_proof_projected_inversion}

\begin{proof}[\lemref{lem_projected_inversion}]
Since $Q$ is symmetric, $\Pi=QQ^\dagger$ is the orthogonal projector onto $\range(Q)$ and
\[
    Q^\dagger=\Pi Q^\dagger\Pi,
    \qquad
    Q^\dagger Q=QQ^\dagger=\Pi.
\]
The nonzero eigenvalues of $Q$ belong to $]0,1]$, so the eigenvalues of $Q^\dagger-\Pi$ on $\range(Q)$ are of the form $q^{-1}-1\geq0$. Together with $H^f\succeq0$, this proves that $\Sigma$ is symmetric positive semidefinite. It also gives $\Pi\Sigma=\Sigma\Pi=\Sigma$.

The matrix $T_\varepsilon=\Sigma+\gamma\varepsilon\Pi$ is symmetric and vanishes on $\ker(Q)$. For every $u\in\range(Q)$,
\[
    u^\top T_\varepsilon u
    =u^\top\Sigma u+\gamma\varepsilon\|u\|^2
    \geq\gamma\varepsilon\|u\|^2.
\]
It is therefore positive definite on $\range(Q)$, and
\[
    \range(T_\varepsilon)=\range(Q),
    \qquad
    T_\varepsilon^\dagger T_\varepsilon
    =T_\varepsilon T_\varepsilon^\dagger=\Pi.
\]

We next use the following identities (recall that $(I_n - \Pi)Q = 0$)
\[
    (I_n-\Pi)H=I_n-\Pi,
    \qquad
    Q^\dagger H
    =T_\varepsilon+\gamma\Pi H^f(I_n-\Pi).
\]
If $HX=0$, the first identity gives $(I_n-\Pi)X=0$, while the second gives $T_\varepsilon\Pi X=0$. The positive definiteness of $T_\varepsilon$ on $\range(Q)$ implies $\Pi X=0$, and hence $X=0$. Thus $H$ is invertible. For $X=H^{-1}M$, the same identities yield
\[
    (I_n-\Pi)X=(I_n-\Pi)M,
    \qquad
    T_\varepsilon\Pi X
    =Q^\dagger M-\gamma\Pi H^f(I_n-\Pi)M=V.
\]
Multiplying the second equality by $T_\varepsilon^\dagger$ proves \eqref{projected_inversion_components} and the resulting decomposition of $H^{-1}M$.

Finally, on an eigenvector of $\Sigma|_{\range(Q)}$ with eigenvalue $\lambda\geq0$, the matrix $T_\varepsilon^\dagger\Sigma$ acts by
\[
    \frac{\lambda}{\lambda+\gamma\varepsilon}\in[0,1).
\]
It vanishes on $\ker(Q)$, so $\|T_\varepsilon^\dagger\Sigma\|\leq1$. If $\range(V)\subseteq\range(\Sigma)$, then $V=\Sigma\Sigma^\dagger V$ since $\Sigma\Sigma^\dagger$ is a projector. Therefore
\[
    \|T_\varepsilon^\dagger V\|
    =\|T_\varepsilon^\dagger\Sigma\Sigma^\dagger V\|
    \leq\|\Sigma^\dagger V\|.
\]
The decomposition of $H^{-1}M$ and the triangle inequality give the second bound in \eqref{projected_inversion_bound}. If $V=\Sigma W$, then $\Sigma^\dagger V=\Sigma^\dagger\Sigma W$ is the orthogonal projection of $W$ onto $\range(\Sigma)$, so $\|\Sigma^\dagger V\|\leq\|W\|$.
\end{proof}
 
\subsection{Proof of Composite Path-Differentiability \thmref{thm_composite_pathdiff}}
\label{ss_proof_composite_pathdiff}

\begin{proof}[\thmref{thm_composite_pathdiff}]
Set
\[
h(x,\theta):=f(x,\theta)+g(x,\theta).
\]
Under \ref{assum:f_c}, each $h(\cdot,\theta)$ is proper, closed, convex, and has a nonempty set of minimizers. Moreover, $h$ is definable. The graph of its Tikhonov minimizers is definable because it can be described by the absence of a point with strictly smaller objective value. For $\varepsilon>0$, strong convexity of $h + \varepsilon\|\cdot\|^2/2$ makes this minimizer unique. The graph of the minimal-norm selection is likewise definable by minimizing $\|x\|$ over $\argmin h(\cdot,\theta)$. Consequently, the joint selection
\[
\xs(\theta,\varepsilon)
=
\begin{cases}
\xsepst, & \varepsilon>0,\\
\minsol(\theta), & \varepsilon=0,
\end{cases}
\]
is definable on $O\times[0,\bar\eps]$.

For every fixed $0<\varepsilon\leq\bar\eps$,
\propref{prop_composite_regularized} supplies the definable conservative
Jacobian \eqref{composite_regularized_jacobian}. Thus the assembled mapping
\[
(\theta,\varepsilon)\longmapsto J_{\xseps}(\theta)
\]
is definable on $O\times]0,\bar\eps]$.

It remains to establish a bound uniform in $\varepsilon$. Fix a compact set
$K\subset O$, let $\theta\in K$, $\varepsilon\in]0,\bar\eps]$, and take
$D\in\mD_\varepsilon(\xsepst,\theta)$. By \ref{assum:range_c}, the columns
of $V(D)$ belong to $\range(\Sigma(D))$ and
$\|\Sigma(D)^\dagger V(D)\|\leq C_K$. The estimate
\eqref{projected_inversion_bound} therefore gives
\[
\|\Heps(D)^{-1}\Meps(D)\|
\leq
\|(I_n-\Pi_Q)\Meps(D)\|+\|\Sigma(D)^\dagger V(D)\|.
\]
Using \ref{assum:m_term_c}, we therefore obtain
\[
\sup_{\substack{\theta\in K,\;0<\varepsilon\leq\bar\eps\\
                 J\in J_{\xseps}(\theta)}}
\|J\|
\leq B_K+C_K.
\]

For each fixed $\theta$, \thmref{thm_Browder} applied to
$h(\cdot,\theta)\in\Gamma_0(\mbR^n)$ gives
\[
\xsepst\longrightarrow\minsol(\theta)
\qquad\text{as }\varepsilon\downarrow0.
\]
Once the uniform bound above and this pointwise convergence are available,
the part of the proof of \thmref{thm_loc_bound_deriv_imply_cont} following
\eqref{eq:bounded_derivates} uses only these two facts. It therefore applies
here with the bound $B_K+C_K$, showing that $\minsol$ is locally Lipschitz
and that the joint selection $\xs$ is continuous at $O\times\{0\}$.
Compactness then gives $\xseps\to\minsol$ uniformly on compact subsets of~$O$.

We have now verified the four hypotheses of \thmref{thm_Schechtman}.
Therefore $J_{\minsol}$ is a definable conservative Jacobian of $\minsol$,
which is consequently path-differentiable on $O$.
\end{proof}
\subsection{Rank-changing Composite Cluster Points
\propref{prop_composite_cluster}}
\label{ss_proof_composite_cluster}

\begin{proof}[\propref{prop_composite_cluster}]
By \thmref{thm_Browder},
$\xs_{\varepsilon_k}(\theta)\to\minsol(\theta)$. The definition of
$z_{\varepsilon_k}$ and the closed graph property of the two Clarke
Jacobians in $\mD_{\varepsilon_k}$ therefore give
\[
    D_0\in\mD_0(\minsol(\theta),\theta),
    \qquad \mathcal M_k\to\mathcal M_0.
\]
For the selected blocks, \ref{assum:range_c} with
$K=\{\theta\}$ gives a constant $C_\theta$ such that
\[
    V_k=\Sigma_kW_k,\qquad
    W_k:=\Sigma_k^\dagger V_k,\qquad
    \|W_k\|\leq C_\theta.
\]
The matrices $\Sigma_k$ and $\Pi_k$ commute, and both vanish on
$\ker(Q_k)$. On $\range(Q_k)$, the eigenvalues of
\[
    R_k:=T_k^\dagger\Sigma_k
\]
are $\lambda/(\lambda+\gamma\varepsilon_k)$, where $\lambda$ ranges over
the eigenvalues of $\Sigma_k$ on this subspace. Hence
\[
    0\preceq R_k\preceq\Pi_k,\qquad
    T_k^\dagger V_k=R_kW_k,\qquad
    \|T_k^\dagger V_k\|\leq\|W_k\|.
\]
Together with \ref{assum:m_term_c} and
\eqref{comp_subspace_decomposition}, this proves that $(J_k)_k$ is bounded.
Every cluster point of $(J_k)_k$ belongs to
$J_{\minsol}(\theta)$ by \eqref{eq:J_min}.

Fix a subsequence along which $J_k\to J$. The sequences $(\Pi_k)_k$ and
$(T_k^\dagger V_k)_k$ are bounded, so a further subsequence satisfies
$\Pi_k\to\bar\Pi$ and $T_k^\dagger V_k\to\bar Y$. A limit of orthogonal
projectors is an orthogonal projector. Since
\[
    \Pi_kQ_k=Q_k\Pi_k=Q_k,\qquad
    \Pi_kT_k^\dagger V_k=T_k^\dagger V_k,
\]
passing to the limit yields
\[
    \bar\Pi Q_0=Q_0\bar\Pi=Q_0,\qquad \bar\Pi\bar Y=\bar Y.
\]
The first two identities imply
$\range(Q_0)\subseteq\range(\bar\Pi)$. Passing to the limit in
\eqref{comp_subspace_decomposition} gives
\[
    J=-(I_n-\bar\Pi)\mathcal M_0-\bar Y,
\]
which proves the claim.
\end{proof}

\subsection{Stable-\texorpdfstring{$Q$}{Q} Residual Formula
\corref{coro_composite_stable_Q}}
\label{ss_proof_composite_stable_Q}

\begin{proof}[\corref{coro_composite_stable_Q}]
Fix a subsequence along which $J_k\to J$. Since
$\operatorname{rank}(Q_k)=\operatorname{rank}(Q_0)$ eventually, continuity
of the pseudoinverse on matrices of fixed rank gives
\[
    Q_k^\dagger\to Q_0^\dagger,\qquad
    \Pi_k\to\Pi_0,\qquad
    \Sigma_k\to\Sigma_0,\qquad
    V_k\to V_0.
\]
Assumption \ref{assum:range_c} makes
$W_k:=\Sigma_k^\dagger V_k$ bounded. Set
$R_k:=T_k^\dagger\Sigma_k$. Its eigenvalues on $\range(Q_k)$ are
$\lambda/(\lambda+\gamma\varepsilon_k)$, so
$0\preceq R_k\preceq\Pi_k$. After taking a further subsequence, there are
matrices $W$ and $R$ such that
\[
    W_k\to W,\qquad R_k\to R.
\]
The identities $V_k=\Sigma_kW_k$ and
$T_k^\dagger V_k=R_kW_k$, together with
\eqref{comp_subspace_decomposition}, imply
\[
    V_0=\Sigma_0W,\qquad
    J=-(I_n-\Pi_0)\mathcal M_0-RW.
\]
Moreover, $0\preceq R\preceq\Pi_0$. Since $R_k$ commutes with $\Sigma_k$
and
\[
    \|\Sigma_k-\Sigma_kR_k\|
    =\max_{\lambda\in\operatorname{spec}(\Sigma_k)}
      \frac{\gamma\varepsilon_k\lambda}
           {\lambda+\gamma\varepsilon_k}
    \leq\gamma\varepsilon_k,
\]
we obtain
\[
    \Sigma_0R=R\Sigma_0=\Sigma_0.
\]
Multiplying these identities by $\Sigma_0^\dagger$ shows that
$P_\Sigma R=RP_\Sigma=P_\Sigma$. Therefore
$S:=R-P_\Sigma$ satisfies
\[
    0\preceq S\preceq\Pi_0-P_\Sigma,\qquad
    \Sigma_0S=S\Sigma_0=0.
\]
Finally,
\[
    RW=P_\Sigma W+SW
      =\Sigma_0^\dagger V_0+SW,
\]
which proves \eqref{composite_cluster_residual}.
\end{proof}

\subsection{Pseudoinverse Cluster Formula
\corref{coro_composite_pseudoinverse}}
\label{ss_proof_composite_pseudoinverse}

\begin{proof}[\corref{coro_composite_pseudoinverse}]
Fix a cluster point $J$ and the subsequence and limits supplied by
\corref{coro_composite_stable_Q}. With
$P_\Sigma:=\Sigma_0^\dagger\Sigma_0$, we have $S=R-P_\Sigma$, so it is
enough to prove that $R_k\to P_\Sigma$.

Under $(\mathrm{G}_\Sigma)$, let
$P_{\Sigma,k}:=\Sigma_k^\dagger\Sigma_k$. The spectral description in the
proof of \corref{coro_composite_stable_Q} gives
\[
    \|R_k-P_{\Sigma,k}\|
    \leq\frac{\gamma\varepsilon_k}
                  {\lambda^*+\gamma\varepsilon_k}
    \longrightarrow0.
\]
The spectral gap and $\Sigma_k\to\Sigma_0$ also give
$P_{\Sigma,k}\to P_\Sigma$, hence $R_k\to P_\Sigma$.

Under $(\mathrm{V}_\Sigma)$, let $r:=\operatorname{rank}(\Sigma_0)$. The spectral
projector associated with the $r$ largest eigenvalues of $\Sigma_k$
converges to $P_\Sigma$, and the corresponding multipliers
$\lambda_i(\Sigma_k)/(\lambda_i(\Sigma_k)+\gamma\varepsilon_k)$ converge
to one. Every remaining eigenvalue tends to zero, while
$(\mathrm{V}_\Sigma)$ makes its multiplier converge to zero. Consequently
$R_k\to P_\Sigma$ in this case as well. Thus $S=0$ under either condition,
which proves \eqref{composite_cluster_pseudoinverse}.
\end{proof}
\subsection{Reachable Binary LASSO Derivatives
\corref{coro_LASSO_reachable}}
\label{ss_proof_LASSO_reachable}

\begin{proof}[\corref{coro_LASSO_reachable}]
Fix $\theta\in\mbR^m$. The map
$\varepsilon\mapsto\xsepst$ is definable because its graph is the graph of
the unique minimizer of the semialgebraic, strongly convex elastic-net
objective. Hence every $\rho_i(\cdot;\theta)$ in the corollary is definable.
The sets on which $\rho_i(\cdot;\theta)$ is positive, negative, or zero are
therefore definable subsets of $\mbR$. Since each has finitely many connected
components, exactly one of them contains all sufficiently small positive
$\varepsilon$.

By the soft-threshold formula in \secref{ss_LASSO}, the $i$th diagonal
entry of a matrix in $\mQ_\varepsilon(\theta)$ belongs respectively to
$\{1\}$, $\{0\}$, or $[0,1]$ in these three cases. It follows that every
binary diagonal matrix satisfying the coordinatewise condition of the corollary
belongs to $\mQ_\varepsilon(\theta)$ for all sufficiently small
$\varepsilon$. Moreover, \thmref{thm_Browder} gives
$\xsepst\to\minsol(\theta)$, and therefore
\[
    z_\varepsilon(\xsepst,\theta)
    \longrightarrow z_0(\minsol(\theta),\theta).
\]
The closed graph of the Clarke Jacobian then places such a binary matrix in
$\mQ_0(\theta)$, hence in
$\mR_{\mathrm{bin}}^{\mathrm{LASSO}}(\theta)$. Conversely, if
$Q_k\in\mQ_{\varepsilon_k}(\theta)$ converges to a binary diagonal matrix $Q$, the
eventual sign of each $\rho_i(\cdot;\theta)$ forces $Q_{ii}$ to have the
value stated in the corollary. This proves the characterization and, since
each of its coordinate sets is nonempty, the nonemptiness assertion.

Fix $Q\in\mR_{\mathrm{bin}}^{\mathrm{LASSO}}(\theta)$. The characterization
shows that $Q\in\mQ_\varepsilon(\theta)$ for every sufficiently small
$\varepsilon$. Choose any sequence $\varepsilon_k\downarrow0$ in this
range. By \corref{coro_LASSO_setup},
\[
    J_k
    :=
    \left(
    I_n-Q\big(I_n-\gamma(A\tp A+\varepsilon_kI_n)\big)
    \right)^{-1}
    \gamma QA\tp
    \in J_{\xs}(\theta,\varepsilon_k).
\]
If $Q=0$, then $J_k=0=(AQ)^\dagger$ for every $k$. Otherwise, define the
binary-one support
\[
    S_1:=S_1(Q):=\{i:Q_{ii}=1\}.
\]
The active-space formula in \secref{ss_LASSO}, with $Q_*=I_{|S_1|}$,
gives
\[
\begin{aligned}
    {[J_k]}_{S_1,:}
    &=
    \left(
    \varepsilon_kI_{|S_1|}
    +A_{:,S_1}\tp A_{:,S_1}
    \right)^{-1}A_{:,S_1}\tp,\\
    {[J_k]}_{S_1^c,:}&=0.
\end{aligned}
\]
An SVD of $A_{:,S_1}$ yields
\[
    {[J_k]}_{S_1,:}\longrightarrow A_{:,S_1}^\dagger,
    \qquad {[J_k]}_{S_1^c,:}\longrightarrow0,
\]
and the limiting matrix is exactly $(AQ)^\dagger$. Finally,
$\xs(\theta,\varepsilon_k)\to\minsol(\theta)$ by \thmref{thm_Browder}, so
\eqref{eq:J_min} gives $(AQ)^\dagger\in J_{\minsol}(\theta)$. This proves
\eqref{LASSO_reachable_derivatives}.
\end{proof}

\section{Detailed Convergence Rates}\label{s_detailed_rates}
\begin{corollary}
        [Estimates for $\varepsilon$ and $k$]\label{coro_estimates_delta} Let $\delta > 0$. Under the assumptions and notation of \remref{rem_iterate_holder}, if 
        \[0 < \varepsilon < \left(\frac{\delta}{C_0}\right)^{2p}\qquad \mathrm{and}\qquad k > \frac{\log((\delta - C_0\varepsilon^{1/(2p)}) / C_1)}{\log(1 - \gamma \varepsilon)},\]
        then $\|x^k_{\varepsilon}(\theta) - \minsol(\theta)\| \leq \delta$. 
\end{corollary}
\begin{proof}
    The first condition gives $C_0 \varepsilon^{1/(2p)} < \delta$, which deals with the regularization error. Now, the second assumption gives 
    \[k\log(1 - \gamma\varepsilon) < \log((\delta - C_0\varepsilon^{1/(2p)}) / C_1) \qquad \Longleftrightarrow \qquad C_1 (1 - \gamma\varepsilon)^k < \delta - C_0\varepsilon^{1/(2p)}\]
    which gives $C_1 (1 - \gamma\varepsilon)^k + C_0\varepsilon^{1/(2p)} < \delta$ using \remref{rem_iterate_holder}. 
\end{proof}
This show that this upper bound, despite not vanishing, can be made arbitrarily small but at the cost of small $\varepsilon$ and large $k$.

Consider the Least-Squares problem from \secref{s_C11}. Its quadratic error bound has exponent $p = 2$, but the closed form gives $O(\varepsilon)$ Tikhonov bias rather than the generic $O(\varepsilon^{1/4})$ estimate. If $A = 0$, $\xsepst = \minsol(\theta) = 0$ with zero derivatives. Otherwise, define respectively the operator norm and the minimal nonzero singular value of $A$ by
\[\lambda := \|A\|_{\mathrm{op}}, \qquad \sigma := \min\{\sigma_i(A) : \sigma_i(A) > 0\}.\]
\begin{proposition}
    [Complete rates for Least-Squares]\label{prop_rates_LS} Fix $\theta \in \Theta$, $\varepsilon > 0$ and let 
    \[D_{\varepsilon} := (A\tp A + \varepsilon I_n)^{-1}A\tp, \quad D_0 = A^{\dagger}.\]
    The solution and derivatives satisfy
    \[\begin{split}
        \|\xsepst - \minsol(\theta)\| & \leq \frac{\varepsilon}{\sigma(\sigma^2 + \varepsilon)} \|\Pi_{\range(A)}\theta\| \leq \frac{\varepsilon}{\sigma(\sigma^2 + \varepsilon)}\|\theta\|,\\
        \|D_{\varepsilon} - D_0\| & = \frac{\varepsilon}{\sigma(\sigma^2 + \varepsilon)} \leq \frac{\varepsilon}{\sigma^3}.
    \end{split}\]
\end{proposition}

\begin{proof}
    In this setting, the fixed-point mapping reads, for $\varepsilon > 0$ and suitable $\gamma > 0$, 
\[\varphi^{\gamma}_{\varepsilon}(x, \theta) = (I - \varepsilon\gamma I - \gamma A\tp A)x + \gamma A\tp \theta.\]
Let us first consider the iterates and show the first inequalities. Recall the closed-form expressions $\xsepst = (A\tp A + \varepsilon I_n)^{-1}A\tp \theta$ and $\minsol(\theta) = A^{\dagger}\theta$, using the singular value decomposition $A = U\Sigma V\tp$ one gets 
\[\|\xsepst - \minsol(\theta)\| = \|V ((\Sigma\tp\Sigma + \varepsilon I_n)^{-1}\Sigma\tp - \Sigma^{\dagger}) U\tp \Pi_{\range(A)}\theta\|\]
and, coordinate-wise, using that $\sigma_i^{\dagger} = 1/\sigma_i$ if $\sigma_i > 0$ and $0$ otherwise
\[\frac{\sigma_i}{\sigma_i^2 + \varepsilon} - \sigma_i^{\dagger} = \left\{ \begin{array}{cc} 0 & \textnormal{if}\ \sigma_i = 0, \\ -\varepsilon/(\varepsilon\sigma_i + \sigma_i^3) & \textnormal{if}\ \sigma_i > 0. \end{array} \right.\]
Thus we have 
\[\|\xsepst - \minsol(\theta)\| \leq \frac{\varepsilon}{\sigma(\sigma^2 + \varepsilon)} \|\Pi_{\range(A)} \theta\| \leq \frac{\varepsilon}{\sigma^3}\|\theta\|.\]
In particular, we have convergence in $\varepsilon^1$ instead of the generic estimate $\varepsilon^{1/4}$. Now, for the implicit derivatives, we must estimate $\|(A\tp A + \varepsilon I_n)^{-1}A\tp - A^{\dagger}\|$ which is done similarly and notably gives the same convergence rate. 
\end{proof}

\vskip 0.2in
\bibliography{biblio}

\end{document}